\documentclass{amsart}

\usepackage{amsfonts,mathrsfs}
\usepackage{amsmath,amssymb,amsthm, amscd, bm}
\usepackage{tikz}
\usetikzlibrary{arrows,calc,matrix,topaths,positioning,scopes,shapes,decorations,decorations.markings} 
\usepackage{dsfont}
\usepackage{epsfig}
\usepackage{tikz-cd}
\usepackage{hyperref}
\usepackage{fullpage}
\usepackage[foot]{amsaddr}
\usepackage{adjustbox}

\makeatletter
\renewcommand{\email}[2][]{%
  \ifx\emails\@empty\relax\else{\g@addto@macro\emails{,\space}}\fi%
  \@ifnotempty{#1}{\g@addto@macro\emails{\textrm{(#1)}\space}}%
  \g@addto@macro\emails{#2}%
}
\makeatother

\author{Julien Korinman}
\address{Universidade de S\~ao Paulo - Instituto de Ci\^encias Matem\'aticas e de Computa\c{c}\~ao.  
					Avenida Trabalhador S\~ao-Carlense, 400 - Centro CEP: 13566-590 - S\~ao Carlos - SP Brazil}
\email{julien.korinman@gmail.com}

\subjclass{$57$R$56$, $57$N$10$, $57$M$25$.}

\dedicatory{This paper is dedicated to Daciberg Lima Gon\c{c}alves for his $70$-th birthday}

\newcommand{\quotient}[2]{{\raisebox{.2em}{$#1$}\left/\raisebox{-.2em}{$#2$}\right.}}

\newcommand{\Tr}{\operatorname{Tr}}

\def\restriction#1#2{\mathchoice
              {\setbox1\hbox{${\displaystyle #1}_{\scriptstyle #2}$}
              \restrictionaux{#1}{#2}}
              {\setbox1\hbox{${\textstyle #1}_{\scriptstyle #2}$}
              \restrictionaux{#1}{#2}}
              {\setbox1\hbox{${\scriptstyle #1}_{\scriptscriptstyle #2}$}
              \restrictionaux{#1}{#2}}
              {\setbox1\hbox{${\scriptscriptstyle #1}_{\scriptscriptstyle #2}$}
              \restrictionaux{#1}{#2}}}
\def\restrictionaux#1#2{{#1\,\smash{\vrule height .8\ht1 depth .85\dp1}}_{\,#2}} 

\newcommand{\heightexch}[3]{
	\begin{tikzpicture}[baseline=-0.4ex,scale=0.5, >=stealth]
	\draw [fill=gray!60,gray!45] (-.7,-.75)  rectangle (.4,.75)   ;
	\draw[#1] (0.4,-0.75) to (.4,.75);
	\draw[line width=1.2] (0.4,-0.3) to (-.7,-.3);
	\draw[line width=1.2] (0.4,0.3) to (-.7,.3);
	\draw (0.65,0.3) node {\scriptsize{$#2$}}; 
	\draw (0.65,-0.3) node {\scriptsize{$#3$}}; 
	\end{tikzpicture}
}
\newcommand{\heightcurve}{
\begin{tikzpicture}[baseline=-0.4ex,scale=0.5]
\draw [fill=gray!20,gray!45] (-.7,-.75)  rectangle (.4,.75)   ;
\draw[-] (0.4,-0.75) to (.4,.75);
\draw[line width=1.2] (-.7,-0.3) to (-.4,-.3);
\draw[line width=1.2] (-.7,0.3) to (-.4,.3);
\draw[line width=1.15] (-.4,0) ++(-90:.3) arc (-90:90:.3);
\end{tikzpicture}
}

\newcommand{\heightcurveright}{
\begin{tikzpicture}[baseline=-0.4ex,scale=0.5]
\draw [fill=gray!20,gray!45] (-.7,-.75)  rectangle (.4,.75)   ;
\draw[-] (-0.7,-0.75) to (-.7,.75);
\draw[line width=1.2] (0.1,-0.3) to (.4,-.3);
\draw[line width=1.2] (0.1,0.3) to (.4,.3);
\draw[line width=1.15] (.1,0) ++(90:.3) arc (90:270:.3);
\end{tikzpicture}
}

\newcommand{\heightexchright}[3]{
	\begin{tikzpicture}[baseline=-0.4ex,scale=0.5, >=stealth]
	\draw [fill=gray!60,gray!45] (-.7,-.75)  rectangle (.4,.75)   ;
	\draw[#1] (-0.7,-0.75) to (-0.7,.75);
	\draw[line width=1.2] (0.4,-0.3) to (-.7,-.3);
	\draw[line width=1.2] (0.4,0.3) to (-.7,.3);
	\draw (-1,0.3) node {\scriptsize{$#2$}}; 
	\draw (-1,-0.3) node {\scriptsize{$#3$}}; 
	\end{tikzpicture}
}

\begin{document}

\theoremstyle{plain}
\newtheorem{theorem}{Theorem}[section]
\newtheorem{main_theorem}[theorem]{Main Theorem}
\newtheorem{proposition}[theorem]{Proposition}
\newtheorem{corollary}[theorem]{Corollary}
\newtheorem{corollaire}[theorem]{Corollaire}
\newtheorem{lemma}[theorem]{Lemma}
\theoremstyle{definition}
\newtheorem{notations}[theorem]{Notations}
\newtheorem*{notations*}{Notations}
\newtheorem{definition}[theorem]{Definition}
\newtheorem{Theorem-Definition}[theorem]{Theorem-Definition}
\theoremstyle{remark}
\newtheorem{remark}[theorem]{Remark}
\newtheorem{conjecture}[theorem]{Conjecture}
\newtheorem{example}[theorem]{Example}

\title[Unicity for representations of reduced stated skein algebras]{Unicity for representations of reduced stated skein algebras}
%
%
%

\date{}
\maketitle


\begin{abstract} 
We prove that both stated skein algebras and their reduced versions at odd roots of unity are almost-Azumaya and compute the rank of a reduced stated skein algebra over its center, extending a theorem of Frohman, Kania-Bartoszynska and L\^e to the case of open punctured surfaces. We deduce that generic irreducible representations of the reduced stated skein algebras are of quantum Teichm\"uller type and, conversely, that generic quantum Teichm\"uller type representations are irreducible.
\end{abstract}


\section{Introduction}

\paragraph{\textbf{Skein algebras and their representations}}

\par  A \textit{punctured surface} is a pair $\mathbf{\Sigma}=(\Sigma,\mathcal{P})$, where $\Sigma$ is a compact oriented surface and $\mathcal{P}$ is a (possibly empty) finite subset of $\Sigma$ which intersects non-trivially each boundary component.  We write $\Sigma_{\mathcal{P}}:= \Sigma \setminus \mathcal{P}$.
  The set $\partial \Sigma\setminus \mathcal{P}$ consists of a disjoint union of open arcs which we call \textit{boundary arcs}.
   \\ \textbf{Warning:} In this paper, the punctured surface $\mathbf{\Sigma}$ will be called open if the surface $\Sigma$ has non empty boundary and closed otherwise. This convention differs from the traditional one, where some authors refer to open surface a punctured surface $\mathbf{\Sigma}=(\Sigma, \mathcal{P})$ with $\Sigma$ closed and $\mathcal{P}\neq \emptyset$ (in which case $\Sigma_{\mathcal{P}}$ is not closed).

\vspace{2mm}
\par The \textit{Kauffman-bracket skein algebras} were introduced by Bullock and Turaev as a tool to study the $\mathrm{SU}(2)$ Witten-Reshetikin-Turaev topological quantum field theories (\cite{Wi2, RT}). They are associative unitary algebras $\mathcal{S}_{\omega}(\mathbf{\Sigma})$ indexed by a closed punctured surface $\mathbf{\Sigma}$ and a non-zero complex number $\omega$. Despite of the apparent simplicity of their definition, these algebras possesses deep connections with character varieties, knot theory and TQFTs (see e.g. \cite{MarcheCharVarSkein} for a survey). They appear in TQFTs through their finite dimensional irreducible representations as in \cite{BHMV2, BCGPTQFT}. Such a representation exists if and only if  $\omega$ is a root of unity  and the problem of classifying those representations is quite difficult. Significative progresses towards such a classification were made recently by Bonahon and Wong \cite{BonahonWong1, BonahonWong2} using quantum Teichm\"uller theory. Given a topological triangulation $\Delta$ of a punctured surface $\mathbf{\Sigma}=(\Sigma, \mathcal{P})$, whose set of vertices is $\mathcal{P}$, the \textit{balanced Chekhov-Fock algebra} $\mathcal{Z}_{\omega}(\mathbf{\Sigma}, \Delta)$ is a refinement of Chekhov and Fock's quantum Teichm\"uller space in \cite{ChekhovFock}. The representation theory of $\mathcal{Z}_{\omega}(\mathbf{\Sigma}, \Delta)$ is quite easy to study since this algebra is Azumaya of constant rank, hence it is semi-simple and its simple modules are in $1$-to-$1$ correspondence with the characters over its center and all have the same dimension, which is the square root of the rank of $\mathcal{Z}_{\omega}(\mathbf{\Sigma}, \Delta)$ over its center. In \cite{BonahonWongqTrace}, the authors defined an algebra embedding  $\Tr_{\omega} : \mathcal{S}_{\omega}(\mathbf{\Sigma}) \hookrightarrow \mathcal{Z}_{\omega}(\mathbf{\Sigma}, \Delta)$ named the \textit{quantum trace}. In particular, for any irreducible representation $r_0 : \mathcal{Z}_{\omega}(\mathbf{\Sigma}, \Delta) \rightarrow \mathrm{End}(V)$, one can define a representation 
$$ r : \mathcal{S}_{\omega}(\mathbf{\Sigma}) \xrightarrow{\Tr_{\omega}}  \mathcal{Z}_{\omega}(\mathbf{\Sigma}, \Delta) \xrightarrow{r_0} \mathrm{End}(V).$$
Such a representation will be referred to as a \textit{quantum Teichm\"uller representation} of the skein algebra. In the case where $\mathcal{P}=\emptyset$, the definition of quantum Teichm\"uller representations is more delicate (see \cite{BonahonWong3}). The family of quantum Teichm\"uller representations is quite large and Bonahon and Wong conjectured that a generic irreducible representation of $\mathcal{S}_{\omega}(\mathbf{\Sigma})$ is a quantum Teichm\"uller representation. More precisely, let $\mathcal{Z}^0$ denote the center of $\mathcal{S}_{\omega}(\mathbf{\Sigma})$. A representation $r : \mathcal{S}_{\omega}(\mathbf{\Sigma})\rightarrow \mathrm{End}(V)$ which is either irreducible or of quantum Teichm\"uller type sends the elements of $\mathcal{Z}^0$ to scalar operators, hence induces a character $\chi_r \in \mathrm{Specm}(\mathcal{Z}^0)$. Bonahon and Wong conjectured the existence of an open dense Zariski subset $\mathcal{U} \subset \mathrm{Specm}(\mathcal{Z}^0)$ such that any irreducible representation with induced character in $\mathcal{U}$ is of quantum Teichm\"uller type.
\vspace{2mm}
\par The unicity representations conjecture was solved recently by Frohman, Kania-Bartoszynska and L\^e in \cite{FrohmanKaniaLe_UnicityRep}, who proved a stronger result by introducing the concept of \textit{almost-Azumaya algebra}. A unital associative algebra $\mathcal{A}$ over an algebraic closed field is said almost-Azumaya if it becomes Azumaya after localization by a non zero central element. By \cite[Theorem $3.6$]{FrohmanKaniaLe_UnicityRep}, an algebra is almost-Azumaya if:
\begin{enumerate}
\item $\mathcal{A}$ is finitely generated as an algebra, 
\item $\mathcal{A}$ is prime, 
\item $\mathcal{A}$ is finitely generated as module over its center $\mathcal{Z}$.
\end{enumerate}
An algebra satisfying these conditions will be called \textit{affine almost-Azumaya}. Its \textit{PI-degree} $D$ is defined by $D^2=\dim_{Q(\mathcal{Z})} (\mathcal{A}\otimes_{\mathcal{Z}}Q(\mathcal{Z}))$ where $Q(\mathcal{Z})$ is the field of fractions of $\mathcal{Z}$.

\begin{theorem}[Frohman, Kania-Bartoszynska, L\^e\cite{FrohmanKaniaLe_UnicityRep, FrohmanKaniaLe_DimSkein}]\label{theorem_FKL} 
\begin{enumerate}
\item If $\mathcal{A}$ is an affine almost-Azumaya algebra of PI-degree $D$, then:
\begin{enumerate}
\item The map $\chi : \mathrm{Irrep} \to \mathrm{Specm}(\mathcal{Z})$, sending an irreducible representation to the associated character over its center, is surjective; 
\item any irreducible representation of $\mathcal{A}$ has dimension at most $D$; 
\item there exists a Zariski  open dense subset $\mathcal{U} \subset \mathrm{Specm}(\mathcal{Z})$ such that for any two irreducible representations $V_1, V_2$ of $\mathcal{A}$ such that $\chi(V_1)=\chi(V_2)\in \mathcal{U}$, then $V_1$ and $V_2$ are isomorphic and have dimension $D$. Moreover any representation sending $\mathcal{Z}$ to scalar operators and whose induced character lies in $\mathcal{U}$ is semi-simple.
\end{enumerate}
\item For a closed punctured surface $\mathbf{\Sigma}=(\Sigma, \mathcal{P})$ with $\omega\in \mathbb{C}^*$ a root of unity of order $N>1$, then $\mathcal{S}_{\omega}(\mathbf{\Sigma})$ is affine almost-Azumaya of PI-degree $D$, where $D$ is the dimension of the simple modules of $\mathcal{Z}_{\omega}(\mathbf{\Sigma},\Delta)$.
\end{enumerate}
\end{theorem}

Theorem \ref{theorem_FKL} implies Bonahon and Wong's unicity representations conjecture.
\vspace{2mm}
\par Bonahon-Wong \cite{BonahonWongqTrace} and L\^e \cite{LeStatedSkein} generalized  the notion of Kauffman-bracket skein algebras to open punctured surfaces, where in addition to closed curves the algebras are generated by arcs whose endpoints are endowed with a sign $\pm$ (a state). The motivation for the introduction of these so-called \textit{stated skein algebras} is their good behavior for the operation of gluing two boundary arcs together. This property provides a deeper understanding of the quantum trace. The quantum trace $\Tr_{\omega} : \mathcal{S}_{\omega}(\mathbf{\Sigma}) \rightarrow \mathcal{Z}_{\omega}(\mathbf{\Sigma}, \Delta)$ is still defined for open punctured surfaces, though it is no longer injective. The kernel of the quantum trace was described by Costantino and L\^e in \cite{CostantinoLe19} and the quotient of the stated skein algebra by this kernel is called the \textit{reduced stated skein algebra} and denoted $\overline{\mathcal{S}}_{\omega}(\mathbf{\Sigma})$. The definition of quantum Teichm\"uller representations of the reduced stated skein algebras extends straightforwardly. 

\vspace{4mm}
\paragraph{\textbf{Main results}}
\vspace{2mm}
\par The purpose of this paper is to extend Theorem \ref{theorem_FKL} to open punctured surfaces, that is to prove the

\begin{theorem}\label{theorem1}
If $\omega$ is a root of unity of odd order $N>1$,  both the stated skein algebra $\mathcal{S}_{\omega}(\mathbf{\Sigma})$ and the reduced stated skein algebra $\overline{\mathcal{S}}_{\omega}(\mathbf{\Sigma})$ are affine almost-Azumaya. Moreover the PI-degree of $\overline{\mathcal{S}}_{\omega}(\mathbf{\Sigma})$  is equal to the PI-degree of  $\mathcal{Z}_{\omega}(\mathbf{\Sigma}, \Delta)$.
\end{theorem}

By combining Theorem \ref{theorem_FKL} with Theorem \ref{theorem1}, we obtain an extension of the unicity representations conjecture for open punctured surfaces. More precisely, denoting by $\mathcal{Z}$ the center of $\mathcal{Z}_{\omega}(\mathbf{\Sigma}, \Delta)$ and by $\mathcal{Z}^0$ the center of $\overline{\mathcal{S}}_{\omega}(\mathbf{\Sigma})$, one obtains the:

\begin{corollary}\label{corollary1}
\begin{enumerate}
\item There exists a Zariski open dense subset $\mathcal{V}\subset \mathrm{Specm}(\mathcal{Z})$ such that any simple $\mathcal{Z}_{\omega}(\mathbf{\Sigma}, \Delta)$-module, whose induced character lies in $\mathcal{V}$, induces a simple $\overline{\mathcal{S}}_{\omega}(\mathbf{\Sigma})$ module.
\item There exists a Zariski open dense subset $\mathcal{U}\subset \mathrm{Specm}(\mathcal{Z}^0)$ such that any simple $\overline{\mathcal{S}}_{\omega}(\mathbf{\Sigma})$-module, whose induced character lies in $\mathcal{U}$, is isomorphic to a quantum Teichm\"uller representation.
\end{enumerate}
\end{corollary}

The dimension of quantum Teichm\"uller representations, hence the dimension of generic irreducible $\overline{\mathcal{S}}_{\omega}(\mathbf{\Sigma})$ representation, were computed in \cite{BonahonWong2} when $\mathbf{\Sigma}$ is closed and \cite{KojuQuesneyQNonAb} when $\mathbf{\Sigma}$ is open and is equal to $N^{3g-3+s+n_{\partial}}$, where $N$ is the order of the root of unity $\omega$, $g$ is the genus of $\Sigma$ (assumed to be connected here), $s=|\mathcal{P}|$ is the number of punctures and $n_{\partial}$ is the number of boundary components of $\Sigma$.

\vspace{2mm}
\paragraph{\textbf{Plan of the paper}}
In the second section we briefly recall the definitions and basic properties of the (reduced) stated skein algebras, balanced Chekhov-Fock algebras and the quantum trace. In the third section we prove that both the stated skein algebras and their reduced versions are finitely generated as algebras whereas in the fourth section we prove that they are finitely generated as modules over their centers. The fifth section is devoted to the construction of  suitable valuations on the reduced stated skein algebras derived from the quantum trace. This is the most technical and original part of the paper and can probably have applications beyond the scope of this paper.  In the last section we eventually use these valuations to characterise the center of a reduced stated algebra and compute its rank over its center.

\vspace{2mm}
\paragraph{\textbf{Acknowledgments.}} The author thanks  F.Costantino, L.Funar , A.Quesney for useful discussions and T.L\^e for clarifying an important issue concerning the Chebyshev morphism. He also thanks the anonymous referee for valuable suggestions that improved the clarity of the paper. He acknowledges  support from the JSPS, the grant ANR  ModGroup, the GDR Tresses, the GDR Platon and the GEAR Network.

\section{Skein algebras and the quantum trace}
\subsection{Stated skein algebras and their reduced version}

\begin{definition}
A \textit{punctured surface} is a pair $\mathbf{\Sigma}=(\Sigma,\mathcal{P})$ where $\Sigma$ is a compact oriented surface and $\mathcal{P}$ is a finite subset of $\Sigma$ which intersects non-trivially each boundary component. A \textit{boundary arc} is a connected component of $\partial \Sigma \setminus \mathcal{P}$. We write $\Sigma_{\mathcal{P}}:= \Sigma \setminus \mathcal{P}$.
\end{definition}

\textbf{Definition of stated skein algebras}
\\ Before stating precisely the definition of stated skein algebras, let us sketch it informally. Given a punctured surface $\mathbf{\Sigma}$, the stated skein algebra $\mathcal{S}_{\omega}(\mathbf{\Sigma})$ is the quotient of the $\mathbb{C}$-vector space spanned by isotopy classes of stated tangles in $\Sigma_{\mathcal{P}}\times (0,1)$  by some local skein relations. The left part of Figure \ref{fig_statedtangle} illustrates such a stated tangle: each point of $\partial T \subset \partial \Sigma_{\mathcal{P}}$ is equipped with a sign $+$ or $-$ (the state). Here the stated tangle is the union of three stated arcs and one closed curve. In order to work with two-dimensional pictures, we will consider the projection of tangles in $\Sigma_{\mathcal{P}}$ as in the right part of Figure \ref{fig_statedtangle}; such a projection will be referred to as a diagram.

\begin{figure}[!h] 
\centerline{\includegraphics[width=6cm]{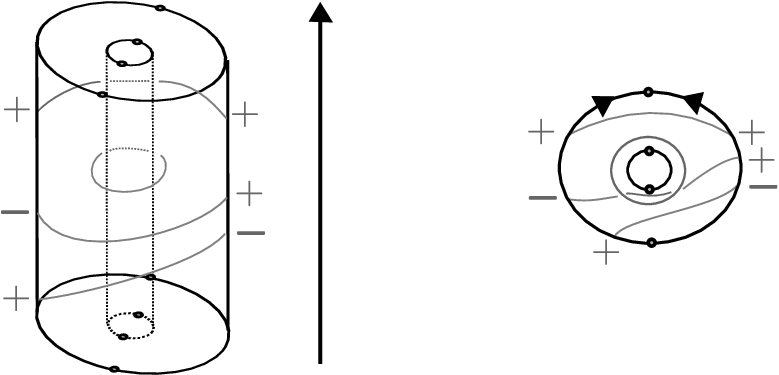} }
\caption{On the left: a stated tangle. On the right: its associated diagram. The arrows represent the height orders. } 
\label{fig_statedtangle} 
\end{figure} 

\par   A \textit{tangle} in $ \Sigma_{\mathcal{P}}\times (0,1)$   is a  compact framed, properly embedded $1$-dimensional manifold $T\subset \Sigma_{\mathcal{P}}\times (0,1)$ such that for every point of $\partial T \subset \partial \Sigma_{\mathcal{P}}\times (0,1)$ the framing is parallel to the $(0,1)$ factor and points to the direction of $1$.  Here, by framing, we refer to a section of the unitary normal bundle of $T$. The \textit{height} of $(v,h)\in \Sigma_{\mathcal{P}}\times (0,1)$ is $h$.  If $b$ is a boundary arc and $T$ a tangle, we impose that no two points in $\partial_bT:= \partial T \cap b\times(0,1)$  have the same heights, hence the set $\partial_bT$ is totally ordered by the heights. Two tangles are isotopic if they are isotopic through the class of tangles that preserve the boundary height orders. By convention, the empty set is a tangle only isotopic to itself.
 
\vspace{2mm}
\par Let $\pi : \Sigma_{\mathcal{P}}\times (0,1)\rightarrow \Sigma_{\mathcal{P}}$ be the projection with $\pi(v,h)=v$. A tangle $T$ is in \textit{generic position} if for each of its points, the framing is parallel to the $(0,1)$ factor and points in the direction of $1$ and is such that $\pi_{\big| T} : T\rightarrow \Sigma_{\mathcal{P}}$ is an immersion with at most transversal double points in the interior of $\Sigma_{\mathcal{P}}$. Every tangle is isotopic to a tangle in generic position. We call \textit{diagram}  the image $D=\pi(T)$ of a tangle in generic position, together with the over/undercrossing information at each double point. An isotopy class of diagram $D$ together with a total order of $\partial_b D:=\partial D\cap b$ for each boundary arc $b$, define uniquely an isotopy class of tangle. When choosing an orientation $\mathfrak{o}(b)$ of a boundary arc $b$ and a diagram $D$, the set $\partial_bD$ receives a natural order by setting that the points are increasing when going in the direction of $\mathfrak{o}(b)$. We will represent tangles by drawing a diagram and an orientation (an arrow) for each boundary arc, as in Figure \ref{fig_statedtangle}. When a boundary arc $b$ is oriented we assume that $\partial_b D$ is ordered according to the orientation. A \textit{state} of a tangle is a map $s:\partial T \rightarrow \{-, +\}$. A pair $(T,s)$ is called a \textit{stated tangle}. We define a \textit{stated diagram} $(D,s)$ in a similar manner.

 \vspace{2mm}
\par  Let  $\omega\in \mathbb{C}^*$ a non-zero complex number and write $A:=\omega^{-2}$.

\begin{definition}\label{def_stated_skein}\cite{LeStatedSkein} 
  The \textit{stated skein algebra}  $\mathcal{S}_{\omega}(\mathbf{\Sigma})$ is the  free $\mathbb{C}$-module generated by isotopy classes of stated tangles in $\Sigma_{\mathcal{P}}\times (0, 1)$ modulo the following relations \eqref{eq: skein 1} and \eqref{eq: skein 2}, 
  	\begin{equation}\label{eq: skein 1} 
\begin{tikzpicture}[baseline=-0.4ex,scale=0.5,>=stealth]	
\draw [fill=gray!45,gray!45] (-.6,-.6)  rectangle (.6,.6)   ;
\draw[line width=1.2,-] (-0.4,-0.52) -- (.4,.53);
\draw[line width=1.2,-] (0.4,-0.52) -- (0.1,-0.12);
\draw[line width=1.2,-] (-0.1,0.12) -- (-.4,.53);
\end{tikzpicture}
=A
\begin{tikzpicture}[baseline=-0.4ex,scale=0.5,>=stealth] 
\draw [fill=gray!45,gray!45] (-.6,-.6)  rectangle (.6,.6)   ;
\draw[line width=1.2] (-0.4,-0.52) ..controls +(.3,.5).. (-.4,.53);
\draw[line width=1.2] (0.4,-0.52) ..controls +(-.3,.5).. (.4,.53);
\end{tikzpicture}
+A^{-1}
\begin{tikzpicture}[baseline=-0.4ex,scale=0.5,rotate=90]	
\draw [fill=gray!45,gray!45] (-.6,-.6)  rectangle (.6,.6)   ;
\draw[line width=1.2] (-0.4,-0.52) ..controls +(.3,.5).. (-.4,.53);
\draw[line width=1.2] (0.4,-0.52) ..controls +(-.3,.5).. (.4,.53);
\end{tikzpicture}
\hspace{.5cm}
\text{ and }\hspace{.5cm}
\begin{tikzpicture}[baseline=-0.4ex,scale=0.5,rotate=90] 
\draw [fill=gray!45,gray!45] (-.6,-.6)  rectangle (.6,.6)   ;
\draw[line width=1.2,black] (0,0)  circle (.4)   ;
\end{tikzpicture}
= -(A^2+A^{-2}) 
\begin{tikzpicture}[baseline=-0.4ex,scale=0.5,rotate=90] 
\draw [fill=gray!45,gray!45] (-.6,-.6)  rectangle (.6,.6)   ;
\end{tikzpicture}
;
\end{equation}

\begin{equation}\label{eq: skein 2} 
\begin{tikzpicture}[baseline=-0.4ex,scale=0.5,>=stealth]
\draw [fill=gray!45,gray!45] (-.7,-.75)  rectangle (.4,.75)   ;
\draw[->] (0.4,-0.75) to (.4,.75);
\draw[line width=1.2] (0.4,-0.3) to (0,-.3);
\draw[line width=1.2] (0.4,0.3) to (0,.3);
\draw[line width=1.1] (0,0) ++(90:.3) arc (90:270:.3);
\draw (0.65,0.3) node {\scriptsize{$+$}}; 
\draw (0.65,-0.3) node {\scriptsize{$+$}}; 
\end{tikzpicture}
=
\begin{tikzpicture}[baseline=-0.4ex,scale=0.5,>=stealth]
\draw [fill=gray!45,gray!45] (-.7,-.75)  rectangle (.4,.75)   ;
\draw[->] (0.4,-0.75) to (.4,.75);
\draw[line width=1.2] (0.4,-0.3) to (0,-.3);
\draw[line width=1.2] (0.4,0.3) to (0,.3);
\draw[line width=1.1] (0,0) ++(90:.3) arc (90:270:.3);
\draw (0.65,0.3) node {\scriptsize{$-$}}; 
\draw (0.65,-0.3) node {\scriptsize{$-$}}; 
\end{tikzpicture}
=0,
\hspace{.2cm}
\begin{tikzpicture}[baseline=-0.4ex,scale=0.5,>=stealth]
\draw [fill=gray!45,gray!45] (-.7,-.75)  rectangle (.4,.75)   ;
\draw[->] (0.4,-0.75) to (.4,.75);
\draw[line width=1.2] (0.4,-0.3) to (0,-.3);
\draw[line width=1.2] (0.4,0.3) to (0,.3);
\draw[line width=1.1] (0,0) ++(90:.3) arc (90:270:.3);
\draw (0.65,0.3) node {\scriptsize{$+$}}; 
\draw (0.65,-0.3) node {\scriptsize{$-$}}; 
\end{tikzpicture}
=\omega
\begin{tikzpicture}[baseline=-0.4ex,scale=0.5,>=stealth]
\draw [fill=gray!45,gray!45] (-.7,-.75)  rectangle (.4,.75)   ;
\draw[-] (0.4,-0.75) to (.4,.75);
\end{tikzpicture}
\hspace{.1cm} \text{ and }
\hspace{.1cm}
\omega^{-1}
\heightexch{->}{-}{+}
- \omega^{-5}
\heightexch{->}{+}{-}
=
\heightcurve.
\end{equation}
The product of two classes of stated tangles $[T_1,s_1]$ and $[T_2,s_2]$ is defined by  isotoping $T_1$ and $T_2$  in $\Sigma_{\mathcal{P}}\times (1/2, 1) $ and $\Sigma_{\mathcal{P}}\times (0, 1/2)$ respectively and then setting $[T_1,s_1]\cdot [T_2,s_2]=[T_1\cup T_2, s_1\cup s_2]$. Figure \ref{fig_product} illustrates this product.
\end{definition}
\par For a closed punctured surface, $\mathcal{S}_{\omega}(\mathbf{\Sigma})$ coincides with the classical (Turaev's) Kauffman-bracket skein algebra.

\begin{figure}[!h] 
\centerline{\includegraphics[width=8cm]{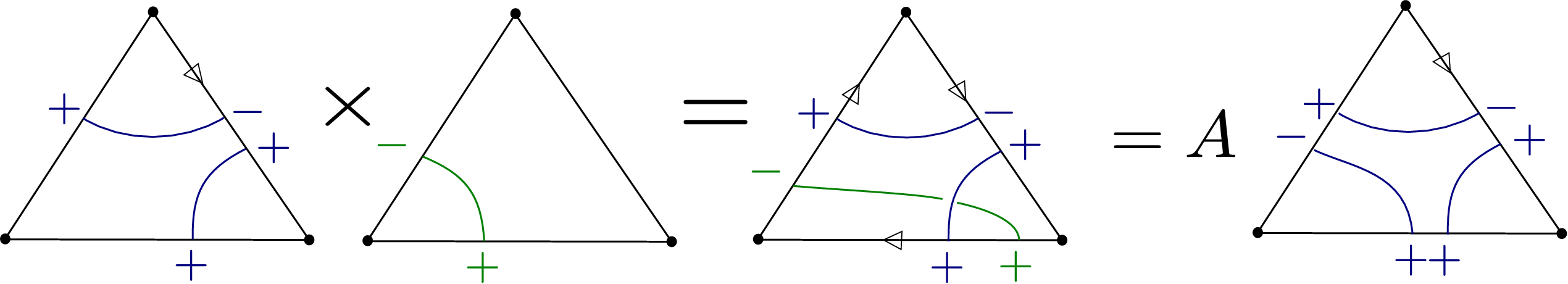} }
\caption{An illustration of the product in stated skein algebras.} 
\label{fig_product} 
\end{figure} 

\textbf{The reduced stated skein algebras}
\\ 
\par A connected diagram is called a \textit{closed curve} if it is closed and an \textit{arc} if it is open. For $p \in \mathcal{P} \cap \partial \Sigma$, a \textit{corner arc} at $p$ is an arc 
$\alpha_p$ such that there exists an arc $\alpha' \subset \partial \Sigma$ with $\partial \alpha_p=\partial \alpha'$, and such that $\alpha_p\cup \alpha'$ bounds a disc in $\Sigma$ whose only intersection with $\mathcal{P}$ is $\{p\}$ (see Figure \ref{fig_bad_arc}). A stated corner arc is called a \textit{bad arc} if the state are $-$ followed by $+$ while we go along the arc counterclockwise around the puncture $p$ as in Figure \ref{fig_bad_arc}.

\begin{figure}[!h] 
\centerline{\includegraphics[width=4cm]{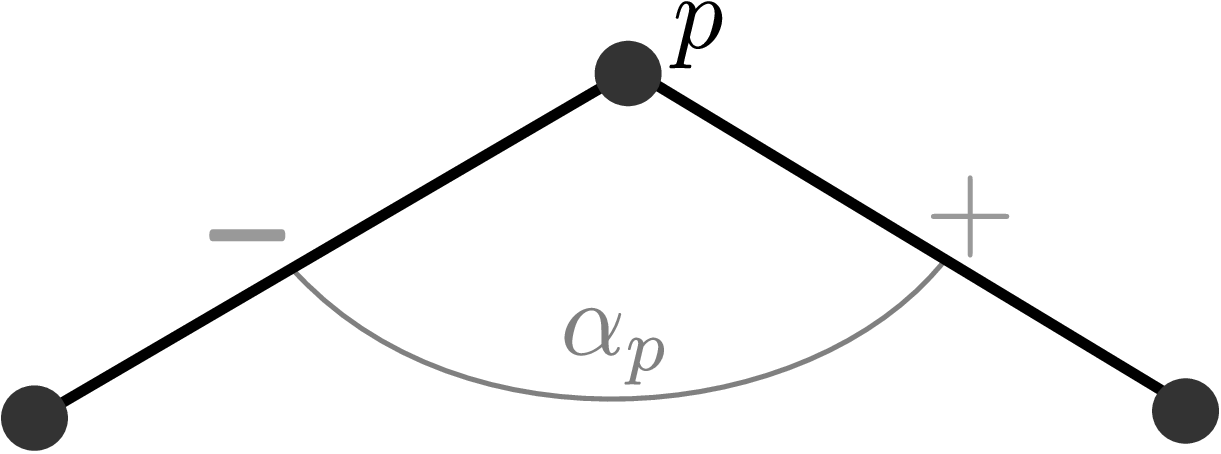} }
\caption{A bad arc.} 
\label{fig_bad_arc} 
\end{figure} 

\begin{definition}\label{def_reduced_skein}\cite{CostantinoLe19} The \textit{reduced stated skein algebra} $\overline{\mathcal{S}}_{\omega}(\mathbf{\Sigma})$ is the quotient of the stated skein algebra $\mathcal{S}_{\omega}(\mathbf{\Sigma})$ by the ideal generated by bad arcs.
\end{definition}

\textbf{Bases for stated skein algebras}
\\ A closed component of a diagram $D$ is trivial if it bounds an embedded disc in $\Sigma_{\mathcal{P}}$. An open component of $D$ is trivial if it can be isotoped, relatively to its boundary, inside some boundary arc. A diagram is \textit{simple} if it has neither double point nor trivial component. By convention, the empty set is a simple diagram. Let $\mathfrak{o}^+$ denote the orientation of the boundary arcs induced by the orientation of $\Sigma$ and for each boundary arc $b$ we write $<_{\mathfrak{o}^+}$ the induced total order on $\partial_b D$. A state $s: \partial D \rightarrow \{ - , + \}$ is $\mathfrak{o}^+-$\textit{increasing} if for any boundary arc $b$ and any two points $x,y \in \partial_b D$, then $x<_{\mathfrak{o}^+} y$ implies $s(x)< s(y)$, with the convention $- < +$. 

\begin{definition}\label{def_basis}
\begin{enumerate}
\item We denote by $\mathcal{B}\subset \mathcal{S}_{\omega}(\mathbf{\Sigma})$ the set of classes of stated diagrams $(D,s)$ such that $D$ is simple and $s$ is $\mathfrak{o}^+$-increasing. 
\item We denote by $\overline{\mathcal{B}} \subset \overline{\mathcal{S}}_{\omega}(\mathbf{\Sigma})$ the set of classes of stated diagrams $(D,s)\in \mathcal{B}$ which do not contain any bad arc.
\end{enumerate}
\end{definition}
By \cite[Theorem $2.11$]{LeStatedSkein}  the set $\mathcal{B}$ is a  basis of $\mathcal{S}_{\omega}(\mathbf{\Sigma})$ and by \cite[Theorem $7.1$]{CostantinoLe19} the set $\overline{\mathcal{B}}$ is a basis of $\overline{\mathcal{S}}_{\omega}(\mathbf{\Sigma})$.
\vspace{2mm}
\par 
\textbf{Gluing maps}
\\ Let $a$, $b$ be two distinct boundary arcs of $\mathbf{\Sigma}$ and let $\mathbf{\Sigma}_{|a\#b}$ be the punctured surface obtained from $\mathbf{\Sigma}$ by gluing $a$ and $b$. Denote by $\pi : \Sigma_{\mathcal{P}} \rightarrow (\Sigma_{|a\#b})_{\mathcal{P}_{|a\#b}}$ the projection and $c:=\pi(a)=\pi(b)$. Let $(T_0, s_0)$ be a stated framed tangle of ${\Sigma_{|a\#b}}_{\mathcal{P}_{|a\#b}} \times (0,1)$ transversed to $c\times (0,1)$ and such that the heights of the points of $T_0 \cap c\times (0,1)$ are pairwise distinct and the framing of the points of $T_0 \cap c\times (0,1)$ is vertical. Let $T\subset \Sigma_{\mathcal{P}}\times (0,1)$ be the framed tangle obtained by cutting $T_0$ along $c$. 
Any two states $s_a : \partial_a T \rightarrow \{-,+\}$ and $s_b : \partial_b T \rightarrow \{-,+\}$ give rise to a state $(s_a, s, s_b)$ on $T$. 
Both the sets $\partial_a T$ and $\partial_b T$ are in canonical bijection with the set $T_0\cap c$ by the map $\pi$. Hence the two sets of states $s_a$ and $s_b$ are both in canonical bijection with the set $\mathrm{St}(c):=\{ s: c\cap T_0 \rightarrow \{-,+\} \}$. 

\begin{definition}\label{def_gluing_map}
Let $i_{|a\#b}: \mathcal{S}_{\omega}(\mathbf{\Sigma}_{|a\#b}) \rightarrow \mathcal{S}_{\omega}(\mathbf{\Sigma})$ be the linear map given, for any $(T_0, s_0)$ as above, by: 
$$ i_{|a\#b} \left( [T_0,s_0] \right) := \sum_{s \in \mathrm{St}(c)} [T, (s, s_0 , s) ].$$
\end{definition}

\begin{figure}[!h] 
\centerline{\includegraphics[width=8cm]{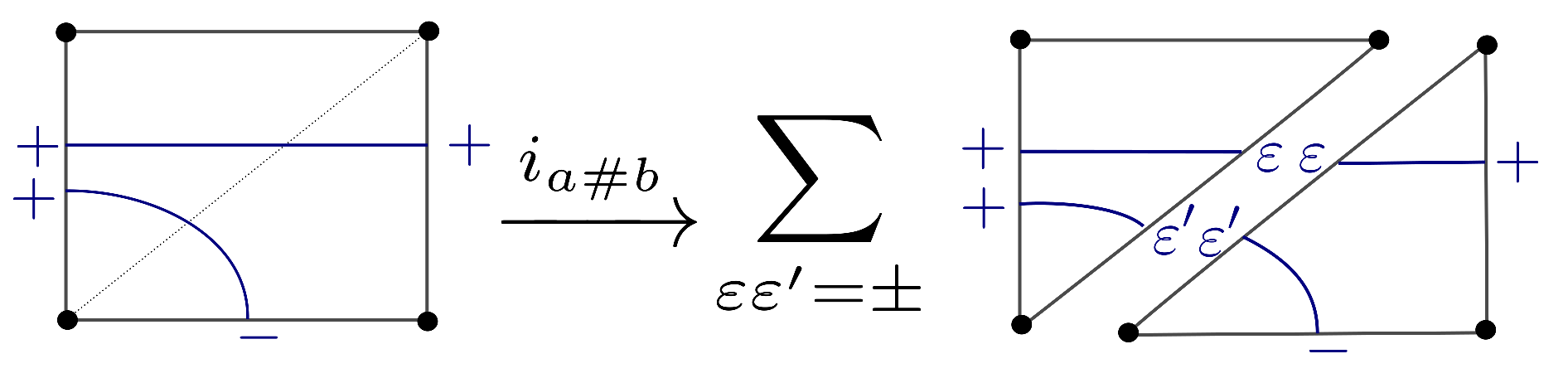} }
\caption{An illustration of the gluing map $i_{a\#b}$.} 
\label{fig_gluingmap} 
\end{figure} 

\begin{theorem}\label{theorem_gluing}\cite[Theorem $3.1$]{LeStatedSkein}, \cite[Theorem $7.6$]{CostantinoLe19} 
\begin{enumerate}
\item The linear map $i_{|a\#b}: \mathcal{S}_{\omega}(\mathbf{\Sigma}_{|a\#b}) \rightarrow \mathcal{S}_{\omega}(\mathbf{\Sigma})$ is an injective morphism of algebras. Moreover the gluing operation is coassociative in the sense that if $a,b,c,d$ are four distinct boundary arcs, then we have $i_{|a\#b} \circ i_{|c\#d} = i_{|c\#d} \circ i_{|a\#b}$.
\item The morphism $i_{|a\#b}$ induces, by passing to the quotient, an injective algebra morphism (still denoted by the same letter) $i_{|a\#b}: \overline{\mathcal{S}}_{\omega}(\mathbf{\Sigma}_{|a\#b}) \rightarrow \overline{\mathcal{S}}_{\omega}(\mathbf{\Sigma})$.
\end{enumerate}
\end{theorem}

\textbf{Triangulations}

\begin{definition}\label{def_triangulation}
\begin{itemize}
\item The \textit{triangle} $\mathbb{T}$ is a disc with three punctures on its boundary.
\item A \emph{small} punctured surface is one of the following four connected punctured surfaces: the sphere with one or two punctures; the disc with one or two punctures on its boundary.   
\item A punctured surface is said to \emph{admit a triangulation} if each of its connected components has at least one puncture and is not small. 
\item Suppose $\mathbf{\Sigma}=(\Sigma, \mathcal{P})$ admits a triangulation.  
		A \textit{topological triangulation} $\Delta$ of $\mathbf{\Sigma}$ is a collection $\mathcal{E}(\Delta)$ of arcs in $\Sigma$, named \textit{edges}, which satisfies the following conditions: 
		the endpoints of the edges belong to $\mathcal{P}$;  
		the interior of the edges are pairwise disjoint and do not intersect $\mathcal{P}$; 
		the edges are not contractible and are pairwise non isotopic relatively to their endpoints; 
		the boundary arcs of $\mathbf{\Sigma}$ belong to $\mathcal{E}(\Delta)$. 
	Moreover, the collection $\mathcal{E}(\Delta)$ is required to be maximal for these properties. A pair $(\mathbf{\Sigma}, \Delta)$ is called a \textit{triangulated punctured surface}.
	\end{itemize}
\end{definition}

\par The only connected punctured surfaces with boundary which do not admit a triangulation are the disc with one boundary puncture, for which both the stated skein algebra and the reduced stated skein algebra is $\mathbb{C}$, and the disc with two boundary punctures, for which the stated skein algebra is isomorphic to $\mathbb{C}_q[\mathrm{SL}_2]$ (see \cite{KojuQuesneyClassicalShadows}) and the reduced stated skein algebra is isomorphic to $\mathbb{C}[X^{\pm 1}]$. In both cases Theorem \ref{theorem1} is trivial, hence in the rest of the paper we will only consider punctured surfaces which admit a triangulation.
\vspace{2mm}
\par The closure of each connected component of  $\Sigma \setminus \mathcal{E}(\Delta)$ is called a \textit{face} and the set of faces is denoted by $F(\Delta)$. Note that the indexing $\mathbb{T}\in F(\Delta)$ denotes a face while the indexed $\mathbb{T}$ denotes the triangle. The author hopes that this abuse of notation is harmless.
Given a topological triangulation $\Delta$, the punctured surface is obtained from the disjoint union 
$\bigsqcup_{\mathbb{T}\in F(\Delta)} \mathbb{T}$ of triangles by gluing the triangles along the boundary arcs corresponding to the edges of the triangulation.  By composing the associated gluing maps, one obtains an injective morphism of algebras: 
$$i^{\Delta} : \overline{\mathcal{S}}_{\omega}(\mathbf{\Sigma}) \hookrightarrow \otimes_{\mathbb{T} \in F(\Delta)} \overline{\mathcal{S}}_{\omega}(\mathbb{T}). $$

\textbf{Basic skein relations}
\\  Define the two matrices: 

$$ C= \begin{pmatrix} C_+^+ & C_-^+ \\ C_+^- & C_-^- \end{pmatrix} := \begin{pmatrix} 0 & \omega \\ -\omega^{5} & 0 \end{pmatrix}; \mbox{ note that  }C^{-1}= -A^3 C =  \begin{pmatrix} 0 & -\omega^{-5} \\ \omega^{-1} & 0 \end{pmatrix}, $$
and 
$$ \mathscr{R} =  
\begin{pmatrix}
 \mathscr{R}_{++}^{++} & \mathscr{R}_{+-}^{++} &\mathscr{R}_{-+}^{++} &\mathscr{R}_{--}^{++} \\
\mathscr{R}_{++}^{+-} &\mathscr{R}_{+-}^{+-} &\mathscr{R}_{-+}^{+-} &\mathscr{R}_{--}^{+-}  \\
\mathscr{R}_{++}^{-+} &\mathscr{R}_{+-}^{-+} &\mathscr{R}_{-+}^{-+} &\mathscr{R}_{--}^{-+} \\
\mathscr{R}_{++}^{--} &\mathscr{R}_{+-}^{--} &\mathscr{R}_{-+}^{--} &\mathscr{R}_{--}^{--} 
  \end{pmatrix}
:= \begin{pmatrix} A & 0 & 0 & 0 \\ 0 & 0 &A^{-1} & 0 \\ 0 & A^{-1} & A-A^{-3} & 0 \\ 0 & 0 & 0 & A \end{pmatrix}. $$

We refer to \cite{Faitg_LGFT_SSkein} for a quantum group interpretation of those matrices.

We now list three families of skein relations, which are straightforward consequences of the definition, and will be used in the paper. Let $i,j \in \{ -, + \}$.

\par $\bullet$  \textit{The trivial arc relations:} 

\begin{equation}\label{trivial_arc_rel}
\begin{tikzpicture}[baseline=-0.4ex,scale=0.5,>=stealth]
\draw [fill=gray!45,gray!45] (-.7,-.75)  rectangle (.4,.75)   ;
\draw[->] (0.4,-0.75) to (.4,.75);
\draw[line width=1.2] (0.4,-0.3) to (0,-.3);
\draw[line width=1.2] (0.4,0.3) to (0,.3);
\draw[line width=1.1] (0,0) ++(90:.3) arc (90:270:.3);
\draw (0.65,0.3) node {\scriptsize{$i$}}; 
\draw (0.65,-0.3) node {\scriptsize{$j$}}; 
\end{tikzpicture}
= C^i_j 
\hspace{.2cm}
\begin{tikzpicture}[baseline=-0.4ex,scale=0.5,>=stealth]
\draw [fill=gray!45,gray!45] (-.7,-.75)  rectangle (.4,.75)   ;
\draw[-] (0.4,-0.75) to (.4,.75);
\end{tikzpicture}
, \hspace{.4cm}
\begin{tikzpicture}[baseline=-0.4ex,scale=0.5,>=stealth]
\draw [fill=gray!45,gray!45] (-.7,-.75)  rectangle (.4,.75)   ;
\draw[->] (-0.7,-0.75) to (-.7,.75);
\draw[line width=1.2] (-0.7,-0.3) to (-0.3,-.3);
\draw[line width=1.2] (-0.7,0.3) to (-0.3,.3);
\draw[line width=1.15] (-.4,0) ++(-90:.3) arc (-90:90:.3);
\draw (-0.9,0.3) node {\scriptsize{$i$}}; 
\draw (-0.9,-0.3) node {\scriptsize{$j$}}; 
\end{tikzpicture}
=(C^{-1})^i_j 
\hspace{.2cm}
\begin{tikzpicture}[baseline=-0.4ex,scale=0.5,>=stealth]
\draw [fill=gray!45,gray!45] (-.7,-.75)  rectangle (.4,.75)   ;
\draw[-] (-0.7,-0.75) to (-0.7,.75);
\end{tikzpicture}.
\end{equation}

\par $\bullet$  \textit{The cutting arc relations:}

\begin{equation}\label{cutting_arc_rel}
\heightcurveright
= \sum_{i,j = \pm} C^i_j
 \hspace{.2cm} 
\heightexchright{->}{i}{j}
, \hspace{.4cm} 
\heightcurve =
\sum_{i,j = \pm} (C^{-1})_j^i
\hspace{.2cm}
\heightexch{->}{i}{j}
\end{equation}.

\par $\bullet$ \textit{The height exchange relations:}

\begin{equation}\label{height_exchange_rel}
\heightexch{->}{i}{j}= \sum_{k,l = \pm} \mathscr{R}_{i j}^{k l} 
 \hspace{.2cm} 
 \heightexch{<-}{l}{k}
 , \hspace{.4cm}
 \heightexch{<-}{j}{i} = \sum_{k,l = \pm} (\mathscr{R}^{-1})_{i j}^{k l} 
 \hspace{.2cm}
 \heightexch{->}{k}{l}.
\end{equation}

We refer to \cite{LeStatedSkein} for proofs. 

\vspace{2mm}
\par \textbf{Chebyshev morphisms}
\\  The center of the skein algebra of a closed punctured surface was characterised in \cite{FrohmanKaniaLe_UnicityRep}, using the results of \cite{BonahonWong1} stated bellow. The center of the stated skein algebras is not known when the punctured surface is open. However a large family of central elements were found in \cite{KojuQuesneyClassicalShadows}; let us describe them. For $\alpha_{\varepsilon \varepsilon'}$ a stated arc, we denote by $\alpha_{\varepsilon \varepsilon'}^{(N)}$ the class of the stated tangle made of $N$ parallel copies of $\alpha_{\varepsilon \varepsilon'}$. More precisely, the underlying tangle $\alpha^{(N)}$ of $\alpha_{\varepsilon \varepsilon'}^{(N)}$ is made of $N$ parallel copies $\alpha^1\cup \ldots \cup \alpha^N$ of $\alpha$ obtained by pushing $\alpha=\alpha^1$ along the direction of the framing. In the case where both endpoints of $\alpha$, say $s(\alpha)$ and $t(\alpha)$, are in the same boundary arc, we need to be more specific about the height order: if $h(s(\alpha)) < h(t(\alpha))$, we impose that $h(s(\alpha^1))<\ldots <h(s(\alpha^N)) < h(t(\alpha^1)) <\ldots <h(t(\alpha^N))$. The state of $\alpha^{(N)}$ sends every point $s(\alpha^i)$ to the state of $s(\alpha)$ and every point $t(\alpha^i)$ to the state of $t(\alpha)$. Note that when the two endpoints of $\alpha$ lye in distinct boundary arcs, one has $\alpha_{\varepsilon \varepsilon'}^{(N)}= (\alpha_{\varepsilon \varepsilon'})^N$ (at the power $N$).

\begin{definition}  The $N$-th Chebyshev polynomial of first kind is the polynomial  $T_N(X) \in \mathbb{Z}[X]$ defined by the recursive formulas $T_0(X)=2$, $T_1(X)=X$ and $T_{n+2}(X)=XT_{n+1}(X) -T_n(X)$ for $n\geq 0$.
\end{definition}

\begin{theorem}\label{theorem_center_skein}[\cite{BonahonWong1} when $\mathbf{\Sigma}$ is closed, \cite[Theorem $1.2$]{KojuQuesneyClassicalShadows} when $\mathbf{\Sigma}$ is open]
	Suppose that $\omega$ is a root of unity of odd order $N>1$. There exists an embedding (named \textit{Chebyshev morphism})
	\begin{equation*}
	j_{\mathbf{\Sigma}} : \mathcal{S}_{+1}(\mathbf{\Sigma}) \hookrightarrow \mathcal{Z}\left( \mathcal{S}_{\omega}(\mathbf{\Sigma}) \right)
	\end{equation*}
	of the (commutative) stated skein algebra with parameter $+1$ into the center of the stated skein algebra with parameter $\omega$. Moreover, the morphism $j_{\mathbf{\Sigma}}$  is characterized by the property that it sends a closed curve $\gamma$ to $T_N(\gamma)$ and a stated arc $\alpha_{\varepsilon \varepsilon'}$ to $\alpha_{\varepsilon \varepsilon'}^{(N)}$.
\end{theorem}

\begin{remark}The fact that Chebyshev morphisms for open punctured surfaces are only known when $\omega$ is root of unity of odd order is the main reason why we do not treat the case of roots of unity of even order in this paper. For closed surfaces, Bonahon and Wong still defined a similar morphism for roots of unity of even order and the proof in \cite{FrohmanKaniaLe_UnicityRep} of the unicity representations conjecture holds also for those cases.  T.L\^e recently informed the author that, together with W.Bloomquist, they extended the Chebyshev morphism for open punctured surfaces and roots of unity of even order (\cite{BloomquistLe}).
\end{remark}

The Chebyshev morphism obviously induces, by passing to the quotient, an injective morphism of algebras (still denoted by the same letter)
\begin{equation*}
	j_{\mathbf{\Sigma}} : \overline{\mathcal{S}}_{+1}(\mathbf{\Sigma}) \hookrightarrow \mathcal{Z}\left( \overline{\mathcal{S}}_{\omega}(\mathbf{\Sigma}) \right).
\end{equation*}

\subsection{Balanced Chekhov-Fock algebras}\label{subsec_CF}

\par The balanced Chekhov-Fock algebras belong to a family of algebras named quantum tori; let us briefly review their definition and main properties.
\vspace{2mm}
\par 
A \textit{quadratic pair} is a pair $\mathbb{E}=(E, \left(\cdot, \cdot\right))$ where $E$ is a free finitely generated $\mathbb{Z}$-module and $\left(\cdot, \cdot\right) : E \times E \rightarrow \mathbb{Z}$ is a skew-symmetric bilinear map. The associated \textit{quantum torus}  $\mathcal{T}_{\omega}(\mathbb{E})$ is the quotient of the algebra freely generated by elements $Z^e, e\in E$ by the ideal generated by the relations $Z^{e_1+e_2}= \omega^{(e_1,e_2)} Z^{e_1}Z^{e_2}$ for $e_1,e_2 \in E$. Equivalently, given $\{e_1, \ldots, e_n\}$ a basis of $E$,   $\mathcal{T}_{\omega}(\mathbb{E})$ is the algebra generated by elements $(Z^{e_i})^{\pm 1}$ modulo relations $Z^{e_i}Z^{e_j}=\omega^{-2(e_i,e_j)}Z^{e_j}Z^{e_i}$. Suppose that $\omega$ is a root of unity of order $N>1$ and consider the composition
$$ \left(\cdot, \cdot \right)_N : E\times E \xrightarrow{(\cdot, \cdot )} \mathbb{Z} \rightarrow \mathbb{Z}/N\mathbb{Z}.$$
The center $\mathcal{Z}$ of $\mathcal{T}_{\omega}(\mathbb{E})$ is spanned by the elements $Z^{e_0}$ with $e_0$ in the kernel $E^0$ of $ \left(\cdot, \cdot \right)_N$ thus the rank of $\mathcal{T}_{\omega}(\mathbb{E})$ as a $\mathcal{Z}$-module is the index of $E^0$ in $E$. By \cite[Proposition $7.2$]{DeConciniProcesiBook} the quantum torus $\mathcal{T}_{\omega}(\mathbb{E})$ is Azumaya of constant rank, hence the character map $\chi : \mathrm{Irrep}(\mathcal{T}_{\omega}(\mathbb{E})) \rightarrow \mathrm{Specm}(\mathcal{Z})$ is a bijection between isomorphism classes of irreducible representations and characters over the center $\mathcal{Z}$. Moreover, every irreducible representation has dimension $\sqrt{R}$, where $R$ denotes the rank of $\mathcal{T}_{\omega}(\mathbb{E})$ over its center.

\vspace{2mm}
\par Let $(\mathbf{\Sigma}, \Delta)$ be a triangulated punctured surface. A map $\mathbf{k} : \mathcal{E}(\Delta) \rightarrow \mathbb{Z}$ is called \textit{balanced} if for any three edges $e_1,e_2,e_3$ bounding a face of $\Delta$, the integer $\mathbf{k}(e_1)+\mathbf{k}(e_2)+\mathbf{k}(e_3)$ is even. If $\mathbf{k}_1, \mathbf{k}_2$ are balanced, their sum $\mathbf{k}_1+\mathbf{k}_2$ is balanced. We denote by $K_{\Delta}$ the abelian group of balanced maps. For  $e$ and $e'$ two edges, denote by $a_{e,e'}$ the number of faces $\mathbb{T}\in F(\Delta)$ such that $e$ and $e'$ are edges of $\mathbb{T}$ and such that we pass from $e$ to $e'$ in the counter-clockwise direction in $\mathbb{T}$. The \textit{Weil-Petersson}  form $\left(\cdot, \cdot \right)^{WP}: K_{\Delta} \times K_{\Delta} \rightarrow \mathbb{Z}$ is the skew-symmetric form defined  by $\left( \mathbf{k}_1, \mathbf{k}_2\right)^{WP}:= \sum_{e,e'} \mathbf{k}_1(e)\mathbf{k}_2(e')(a_{e,e'}-a_{e',e})$.

\begin{definition}
The \textit{balanced Chekhov-Fock algebra} $\mathcal{Z}_{\omega}(\mathbf{\Sigma}, \Delta)$ is the quantum torus associated to the quadratic pair $(K_{\Delta}, (\cdot, \cdot)^{WP})$.
\end{definition}

The centers of the balanced Chekhov-Fock algebras at odd roots of unity are described as follows.

\begin{definition}\label{def_central_elements}
\begin{itemize}
\item  Let $p\in \mathcal{P}\cap \mathring{\Sigma}$ be an inner puncture. For each edge $e\in \mathcal{E}(\Delta)$, denote by $\mathbf{k}_p(e)\in \{0,1,2\}$ the number of endpoints of $e$ equal to $p$. The \textit{central inner puncture element} is  $H_p:=Z^{\mathbf{k}_p}\in \mathcal{Z}_{\omega}(\mathbf{\Sigma}, \Delta)$.
\item  Let  $\partial$ a connected component of $\partial \Sigma$. For each edge $e$, denote by $\mathbf{k}_{\partial}(e)\in \{0, 1, 2\}$ the number of endpoints of $e$ lying in $\partial$. The \textit{central boundary element} is $H_{\partial}:=Z^{\mathbf{k}_{\partial}}\in \mathcal{Z}_{\omega}(\mathbf{\Sigma}, \Delta)$. 
\item Suppose that $\omega$ is a root of unity of order $N>1$, the \textit{Frobenius morphism} $j_{(\mathbf{\Sigma}, \Delta)} : \mathcal{Z}_{+1}(\mathbf{\Sigma}, \Delta) \rightarrow \mathcal{Z}_{\omega}(\mathbf{\Sigma}, \Delta)$ sending a balanced monomial $Z^{\mathbf{k}}$ to $Z^{N\mathbf{k}}$, is an injective morphism of algebras whose image lies in the center.
\end{itemize}
\end{definition}

\begin{proposition}[\cite{BonahonWong2} when $\mathbf{\Sigma}$ is closed, \cite{KojuQuesneyQNonAb} when $\mathbf{\Sigma}$ is open]\label{prop_center_CF} If $\omega$ is a root of unity of odd order $N>1$, the center of $\mathcal{Z}_{\omega}(\mathbf{\Sigma}, \Delta)$  is generated by the image of the Frobenius morphism together with the central inner puncture and boundary elements.
\end{proposition}

The balanced Chekhov-Fock algebras admit triangular decompositions, similar to the stated skein algebras, defined as follows. Let 
$$i^{\Delta} : \mathcal{Z}_{\omega}(\mathbf{\Sigma},\Delta) \hookrightarrow \otimes_{\mathbb{T} \in F(\Delta)} \mathcal{Z}_{\omega}(\mathbb{T})$$
be the linear map sending $Z^{\mathbf{k}}$ to $\otimes_{\mathbb{T}\in F(\Delta)} Z^{\mathbf{k}_{\mathbb{T}}}$, where given an edge $e_{\mathbb{T}}\in \mathcal{E}(\mathbb{T})$ corresponding to an edge $e\in \mathcal{E}(\Delta)$, one set $\mathbf{k}_{\mathbb{T}}(e_{\mathbb{T}}) := \mathbf{k}(e)$. The linear map $i^{\Delta}$ is an injective algebra morphism.

\subsection{The quantum trace}

We now define an injective algebra morphism $\Tr_{\omega}^{\Delta} : \overline{\mathcal{S}}_{\omega}(\mathbf{\Sigma}) \hookrightarrow \mathcal{Z}_{\omega}(\mathbf{\Sigma}, \Delta)$. First consider the case where $\mathbf{\Sigma}=\mathbb{T}$ is the triangle. Consider $\alpha_1, \alpha_2, \alpha_3$ the three arcs in $\mathbb{T}$ drawn in Figure \ref{fig_triangle} and denote by $e_1,e_2,e_3$ the three edges of $\mathbb{T}$. The reduced stated skein algebra $\overline{\mathcal{S}}_{\omega}(\mathbb{T})$ is generated by the classes of the stated arcs $(\alpha_i)_{\varepsilon \varepsilon}$, for $i=1,2,3$ and $\varepsilon \in \{-, +\}$, moreover one has $(\alpha_i)_{--} = ((\alpha_i)_{++})^{-1}$. Define the balanced maps $\mathbf{k}_1, \mathbf{k}_2, \mathbf{k}_3 \in K_{\mathbb{T}}$ by $\mathbf{k}_i (e_i)=0$ and  $\mathbf{k}_i (e_j)=1$
for $j\neq i$. By \cite[Theorem $7.11$]{CostantinoLe19}, the linear map $\Tr_{\omega}^{\mathbb{T}} : \overline{\mathcal{S}}_{\omega}(\mathbb{T}) \rightarrow \mathcal{Z}_{\omega}(\mathbb{T})$, sending $(\alpha_i)_{++}$ to $Z^{\mathbf{k}_i}$ and sending $(\alpha_i)_{--}$ to $Z^{-\mathbf{k}_i}$, extends to an isomorphism of algebras.

\begin{figure}[!h] 
\centerline{\includegraphics[width=2cm]{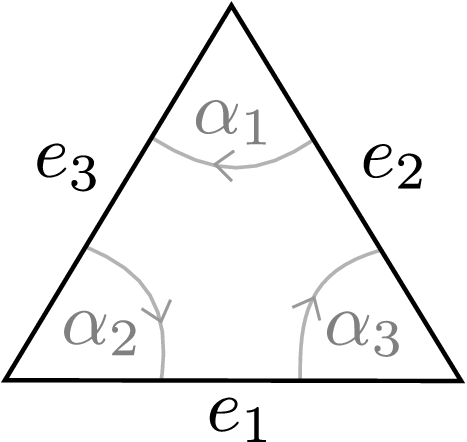} }
\caption{The triangle and some arcs.} 
\label{fig_triangle} 
\end{figure} 

\begin{definition}[\cite{BonahonWongqTrace, LeStatedSkein}]\label{def_qtrace} Let $(\mathbf{\Sigma}, \Delta)$ be a triangulated punctured surface. The \textit{quantum trace} is the unique algebra morphism  $\Tr_{\omega}^{\Delta} : \overline{\mathcal{S}}_{\omega}(\mathbf{\Sigma}) \hookrightarrow \mathcal{Z}_{\omega}(\mathbf{\Sigma}, \Delta)$ making the following diagram commuting: 
$$
\begin{tikzcd}
\overline{\mathcal{S}}_{\omega}(\mathbf{\Sigma})
\arrow[r, hook, "i^{\Delta}"] \arrow[d, dotted, hook, "\Tr_{\omega}^{\Delta}"] &
\otimes_{\mathbb{T}\in F(\Delta)} \overline{\mathcal{S}}_{\omega}(\mathbb{T}) 
\arrow[d, "\cong"' , "\otimes_{\mathbb{T}} \Tr_{\omega}^{\mathbb{T}}"] \\
\mathcal{Z}_{\omega}(\mathbf{\Sigma}, \Delta)
\arrow[r, hook, "i^{\Delta}"] &
\otimes_{\mathbb{T}\in \Delta} \mathcal{Z}_{\omega}(\mathbb{T})
\end{tikzcd}
$$

\end{definition}
\begin{remark}\label{remark_prime} Since $\mathcal{Z}_{\omega}(\mathbb{T})$ is a quantum torus, it has no zero divisor and since both $i^{\Delta} : \mathcal{S}_{\omega}(\mathbf{\Sigma}) \hookrightarrow \otimes_{\mathbb{T}} \overline{\mathcal{S}}_{\omega}(\mathbb{T})$ and $i^{\Delta} : \overline{\mathcal{S}}_{\omega}(\mathbf{\Sigma}) \hookrightarrow \otimes_{\mathbb{T}} \overline{\mathcal{S}}_{\omega}(\mathbb{T})$ are embeddings, both the stated skein algebra and  the reduced stated skein algebra have no zero divisor. In particular, they are prime algebras.
\end{remark}

 \par 
We now describe the image of a basis element through the quantum trace.
 We first prove a preliminary statement.

\begin{lemma}\label{lemma_product}
 Let $(D,s)$ be stated diagram such that $D$ is simple and $s$ is $\mathfrak{o}^+$ increasing. Suppose that $D=D_1 \bigsqcup D_2$ is a disjoint union of two sub-diagrams and write $s_1, s_2$ the restriction of $s$ to $D_1$ and $D_2$ respectively. We write $(D,s)= (D_1,s_1)\cup (D_2,s_2)$.
There exists an integer $n\in \mathbb{Z}$ such that in $\mathcal{S}_{\omega}(\mathbf{\Sigma})$ one has the equality $[D,s] = \omega^n [D_1,s_1] [D_2, s_2]$.
\end{lemma}

\begin{proof}
Denote by $T, T_1, T_2$ some framed tangles representing $D, D_1$ and $D_2$ respectively,  with the endpoints height ordered using $\mathfrak{o}^+$ and denote by $T'$ the tangle obtained by putting $T_1$ on top of $T_2$ in the $(0,1)$ direction. The sets $\partial T$, $\partial T'$ and $\partial D$ are in natural bijection and we denote by $s$ and $s'$ the corresponding states on $T$ and $T'$. By definition, one has $[T,s]=[D,s]$ and $[T',s']=[D_1,s_1][D_2,s_2]$. The stated tangle $(T',s')$ is obtained from $(T,s)$ by changing the heights of some points of $\partial T$. Because $s$ is $\mathfrak{o}^+$-increasing, one passes from $(T,s)$ to $(T',s')$ by a finite sequence of elementary moves which consists in exchanging the heights of two boundary points $v$ and $v'$ such that $s(v)\geq s(v')$ and $v\geq_{\mathfrak{o}^+} v'$. Note that the height exchange relations \eqref{height_exchange_rel} write:
\begin{equation*}
\heightexch{<-}{+}{+}=\omega^{2} \heightexch{->}{+}{+}, ~~~
\heightexch{<-}{+}{-}=\omega^{-2} \heightexch{->}{+}{-}, ~~~
\heightexch{<-}{-}{-}=\omega^{2} \heightexch{->}{-}{-}
\end{equation*}
Since each of these elementary move changes the class of the stated tangle by multiplying it by a power of $\omega$, one has $[T',s']= \omega^n [T,s]$ for some $n\in \mathbb{Z}$. Therefore $[D,s]= \omega^n [D_1,s_1][D_2,s_2]$ and the proof is completed.
\end{proof}

Denote by $\mathfrak{o}^-$ the orientation of the boundary arcs of $\mathbb{T}$ which is opposite to $\mathfrak{o}^+$.

\begin{lemma}\label{lemma_triangle_trace} Let $(D,s)$ be a stated diagram in the triangle $\mathbb{T}$ with $D$ simple. Let $\mathfrak{o}$ be an arbitrary orientation of the boundary arcs of $\mathbb{T}$ and let $[D,s]_{\mathfrak{o}} \in \overline{\mathcal{S}}_{\omega}(\mathbb{T})$ be the associated skein element. Let $\mathds{k}=\mathds{k}(D,s)\in K_{\mathbb{T}}$ be the balanced map defined by $\mathds{k}(e_i) := \sum_{v\in D\cap e_i} s(v)$. 
\begin{enumerate}
\item
There exists $P(X)\in \mathbb{Z}[X^{\pm 1}]$ such that $Tr_{\omega}^{\mathbb{T}}([D,s]_{\mathfrak{o}})= P(\omega) Z^{\mathds{k}}$. 
\item Moreover, if $\mathfrak{o}\neq \mathfrak{o}^-$ and $(D,s)$ does not contain any bad arc, then $P(\omega)=\omega^n$ for some $n\in \mathbb{Z}$.
\end{enumerate}
\end{lemma}

\begin{proof}
If $D$ is connected, i.e. if $(D,s)$ is a stated arc, then the Lemma is an immediate consequence of the definition of $Tr_{\omega}^{\mathbb{T}}$. Suppose that  $(D,s)= \alpha_1 \cup \ldots \cup \alpha_n$ with $\alpha_i$ some stated arcs. By the preceding case, $[\alpha_i]=0$ if $\alpha_i$ is a bad arc and $[\alpha_i]=\omega^{n_i}Z^{\mathds{k}_i}$ else.  If $s$ is $\mathfrak{o}^+$ increasing, then by Lemma \ref{lemma_product}, there exists $n\in \mathbb{Z}$ such that $[D,s]_{\mathfrak{o}}= \omega^n [\alpha_1] \ldots [\alpha_n]$ so it vanishes if one of the $\alpha_i$ is a bad arc; else it is equal to $[D,s]_{\mathfrak{o}}= \omega^{n+ \sum_i n_i}Z^{\mathds{k}_1} \ldots Z^{\mathds{k}_n} = \omega^m Z^{\mathds{k}}$ for some $m\in \mathbb{Z}$ since $\mathds{k}=\sum_i \mathds{k}_i$. Eventually, suppose that $s$ is not $\mathfrak{o}^+$ increasing. So there exists two consecutive endpoints $v,w \in \partial D$ in the same boundary arc $e_i$ such that $v<_{\mathfrak{o}_+} w$ and $(s(v), s(w))=(+,-)$. Let us call \textit{bad pair} such a pair of endpoints. Applying the skein relation 
$$ 
\heightexch{->}{-}{+}
= \omega^{-4}
\heightexch{->}{+}{-}
+
\omega \heightcurve
$$
We can write $[D,s]_{\mathfrak{o}}= \omega^{-4} [D_1,s_1] + \omega [D_2,s_2]$ where $\mathds{k}(D,s)=\mathds{k}(D_1,s_1)=\mathds{k}(D_2,s_2)$ and such that $(D_1,s_1)$ and $(D_2,s_2)$ have one bad pair less than $(D,s)$. By induction on the number of bad pairs, we find that $[D,s]_{\mathfrak{o}}$ is a linear combination with coefficients in $\mathbb{Z}[\omega^{\pm 1}]$ of elements $(D_i,s_i)$ such that $\mathds{k}(D,s)=\mathds{k}(D_i,s_i)$ and $s_i$ is $\mathfrak{o}_+$ increasing. Therefore $[D,s]_{\mathfrak{o}}=P(\omega)Z^{\mathds{k}}$ for some $P(X)\in \mathbb{Z}[X^{\pm 1}]$. So we have proved  Item $(1)$ and also Item $(2)$ in the case $\mathfrak{o}=\mathfrak{o}^+$. Let us suppose that $\mathfrak{o}\neq \mathfrak{o}^+$ and $\mathfrak{o}\neq \mathfrak{o}^-$ and that $(D,s)$ does not contain any bad arc and let us prove that $P(\omega)=\omega^n$ for some $n\in \mathbb{Z}$. We will prove the result by induction on the number of connected components of $D$. If $D$ is an arc, the result is immediate. Let us suppose that $D$ contains at least two arcs. By hypothesis on $\mathfrak{o}$, there exists one puncture $p$ of $\mathbb{T}$ such that the orientation of both edges adjacent to $p$ are oriented towards $p$. Let us suppose that $p$ is the puncture between $e_2$ and $e_3$, the other cases are obtained by rotation. There are two cases to consider depending whether the orientation of $e_1$ in $\mathfrak{o}$ coincides with $\mathfrak{o}^+$ (case $1$) or with $\mathfrak{o}^-$ (case $2$), so in pictures we have: 
$$ \mbox{Case 1: } \adjustbox{valign=c}{\includegraphics[width=1.5cm]{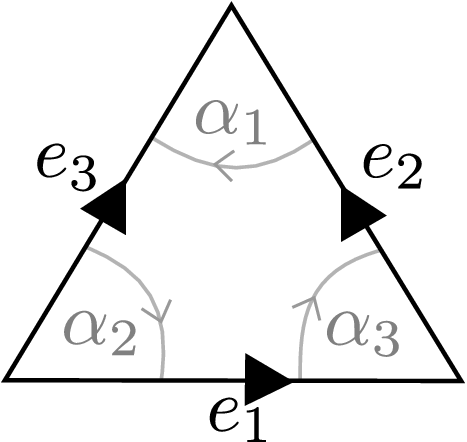}}  \hspace{2cm}  \mbox{Case 2: } \adjustbox{valign=c}{\includegraphics[width=1.5cm]{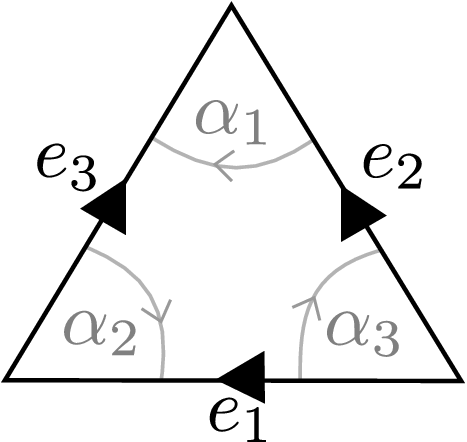}}. $$
Suppose that $D$ is made of $n_1$ parallel copies of $\alpha_1$, $n_2$ copies of $\alpha_2$ and $n_3$ copies of $\alpha_3$. For $i=1,2,3$ and $1\leq k \leq n_i$, we denote by $\alpha_i^{(k)}\subset (D,s)$ the $k$-th stated subarc of $(D,s)$ parallel to $\alpha_i$. First suppose that $n_1>0$. Then $[D,s]_{\mathfrak{o}}=\alpha_1^{(1)} [D_0,s_0]_{\mathfrak{o}}$ where $(D_0,s_0)$ is obtained from $(D,s)$ by removing $\alpha_1^{(1)}$. Applying the induction hypothesis to $[D_0,s_0]_{\mathfrak{o}}$ we have $[D_0,s_0]_{\mathfrak{o}}=\omega^{n_0}Z^{\mathds{k}(D_0,s_0)}$ for some $n_0\in \mathbb{Z}$ so $[D,s]_{\mathfrak{o}}=\omega^n Z^{\mathds{k}}$ for some $n\in \mathbb{Z}$. Next suppose that $n_1=0$.
If we are in Case $1$ and $n_2>0$, then $[D,s]_{\mathfrak{o}}=[D_1,s_1]\alpha_{2}^{(1)}$ where $(D_1,s_1)$ is obtained by removing $\alpha_2^{(1)}$ from $(D,s)$. Similarly if we are in Case $2$ and $n_3>0$, then  $[D,s]_{\mathfrak{o}}=[D_1,s_1]\alpha_{3}^{(1)}$ where $(D_2,s_2)$ is obtained by removing $\alpha_3^{(1)}$ from $(D,s)$. In both cases, we conclude using the induction hypothesis. Suppose we are  in Case $1$ and $n_1=n_2=0$. Let $v,w$ be the two endpoints of $\alpha_3^{(n_3)}$ with $v\in e_2$ and $w\in e_1$. Since $\alpha_3^{(n_3)}$ is not a bad arc, then $(s(v), s(w))\neq (-, +)$. Let $(D_1,s_1)$ be the stated diagram obtained from $(D,s)$ by removing $\alpha_3^{(n_3)}$. If $s(v)=+$, applying the height exchange relations repeatedly, we find that $[D,s]_{\mathfrak{o}}=A^n \alpha_3^{(n_3)} [D_1,s_1]_{\mathfrak{o}}$ for some $n\in \mathbb{Z}$. If $s(w)=-$, again using the height exchange relations, we find that  $[D,s]_{\mathfrak{o}}=A^m [D_1,s_1]_{\mathfrak{o}} \alpha_3^{(n_3)}$ for some $m\in \mathbb{Z}$. In both cases we conclude using the induction hypothesis. Eventually, the case where we are in Case $2$ and $n_1=n_3=0$ is handled similarly by replacing $\alpha_3^{(n_3)}$ by $\alpha_2^{(1)}$. This concludes the proof.

\end{proof}

\par  Let $\mathcal{D}_{\mathbf{\Sigma}}$ be the set of stated diagrams $(D,s)$ such that $D$ is simple, $s$ is $\mathfrak{o}^+$-increasing and $(D,s)$ does not contain bad arcs; hence $\mathcal{D}_{\mathbf{\Sigma}}$ is in natural bijection with the basis $\overline{\mathcal{B}}$. Fix a triangulation $\Delta$ and consider a diagram $D$ isotoped such that it intersects the edges of $\Delta$ transversally with minimal number of intersection points. A \textit{full state} on $D$ is a map $\hat{s} : \mathcal{E}(\Delta)\cap D \rightarrow \{-,+\}$. A pair $(D, \hat{s})$ induces on each face $\mathbb{T}\in F(\Delta)$ a stated diagram in $\mathbb{T}$. A full state $\hat{s}$ is \textit{admissible} if the restriction of $(D,\hat{s})$ on each face does not contain bad arcs. A full state $\hat{s}$ of $D$ induces, by restriction to the boundary arcs, a state $s$. We denote by $\widehat{\mathcal{D}}_{\Delta}$ the set of full stated diagrams $(D, \hat{s})$ such that  its restriction $s$ is $\mathfrak{o}^+$-increasing. There is a natural restriction map $\mathbf{res} : \widehat{\mathcal{D}}_{\Delta} \twoheadrightarrow \mathcal{D}_{\mathbf{\Sigma}}$ and we denote by $\mathrm{St} (D,s) := \mathbf{res}^{-1}(D,s)$ its fibers and by $\mathrm{St}^a(D,s)\subset \mathrm{St}(D,s)$ the subset of admissible states. For $(D, \hat{s})\in \widehat{\mathcal{D}}_{\Delta}$, we denote by $\mathds{k}(D, \hat{s})\in K_{\Delta}$ the balanced map defined by 
$$ \mathds{k}(D, \hat{s}) (e) := \sum_{v\in D\cap e} \hat{s}(v) \quad , e\in \mathcal{E}(\Delta).$$

\begin{lemma}\label{lemma_qtr}
For $(D,s) \in \mathcal{D}_{\mathbf{\Sigma}}$, one has the equality
\begin{equation}\label{eq_qtrace}
 \Tr_{\omega}^{\Delta} ([D,s]) = \sum_{\hat{s} \in \mathrm{St}(D,s)} P_{\hat{s}}(\omega) Z^{\mathds{k}(D,\hat{s})}, 
 \end{equation}
for some Laurent polynomials $P_{\hat{s}}(X^{\pm 1}) \in \mathbb{Z}[X^{\pm 1}]$. Moreover, if $\hat{s}$ is admissible, then $P_{\hat{s}}(\omega)=\omega^{n(\hat{s})}$ for some $n(\hat{s})\in \mathbb{Z}$.
\end{lemma}

\begin{proof}
Let $\mathfrak{o}_{\Delta}$ be an  orientation of the edges of $\Delta$ whose restriction to the boundary arcs coincides with $\mathfrak{o}_{+}$. For each $\mathbb{T}$, this induces an orientation $\mathfrak{o}_{\mathbb{T}}$ of the edges of $\mathbb{T}$ and we chose $\mathfrak{o}_{\Delta}$ such that $\mathfrak{o}_{\mathbb{T}}\neq \mathfrak{o}^-$ for all $\mathbb{T}\in F(\Delta)$. Let $i^{\Delta}: \overline{\mathcal{S}}_{\omega}(\mathbf{\Sigma}) \hookrightarrow \otimes_{\mathbb{T}\in F(\Delta)} \overline{\mathcal{S}}_{\omega}(\mathbb{T})$ be the associated splitting morphism. A full state $\widehat{s}$ induces for each face $\mathbb{T}\in F(\Delta)$ a state $\restriction{\widehat{s}}{\mathbb{T}}$ of $\restriction{D}{\mathbb{T}}:= D\cap \mathbb{T}$ and, by definition of the splitting morphism, one has 
$$ i^{\Delta}\left(\Tr_{\omega}^{\Delta}([D,s])\right) = \sum_{\widehat{s}\in \mathrm{St}(D,s)} \otimes_{\mathbb{T}\in F(\Delta)} [\restriction{D}{\mathbb{T}}, \restriction{\widehat{s}}{\mathbb{T}}]_{\mathfrak{o}_{\mathbb{T}}}.$$
By Lemma \ref{lemma_triangle_trace}, one has $[\restriction{D}{\mathbb{T}}, \restriction{\widehat{s}}{\mathbb{T}}]_{\mathfrak{o}_{\mathbb{T}}}=P_{\widehat{s}, \mathbb{T}}(\omega) Z^{\mathds{k}(\restriction{D}{\mathbb{T}}, \restriction{\widehat{s}}{\mathbb{T}})}$. So setting $P_{\hat{s}}(\omega)=\prod_{\mathbb{T}\in F(\Delta)} P_{\widehat{s}, \mathbb{T}}(\omega)$ and noting that $i^{\Delta} (Z^{\mathds{k}(D,s)})= \otimes_{\mathbb{T}\in F(\Delta)} Z^{\mathds{k}(\restriction{D}{\mathbb{T}}, \restriction{\widehat{s}}{\mathbb{T}})}$, we find that 
$$  i^{\Delta}\left(\Tr_{\omega}^{\Delta}([D,s])\right) = i^{\Delta}\left( \sum_{\hat{s} \in \mathrm{St}(D,s)} P_{\hat{s}}(\omega) Z^{\mathds{k}(D,\hat{s})}\right).$$
This proves the first assertion. Now if $\hat{s}$ is admissible, by Lemma \ref{lemma_triangle_trace}, we have $P_{\hat{s}, \mathbb{T}}= \omega^{n(\hat{s}, \mathbb{T})}$ for some $n(\hat{s}, \mathbb{T})\in \mathbb{Z}$, so $P_{\widehat{s}}(\omega) = \omega^{\sum_{\mathbb{T}\in F(\Delta)}n(\hat{s}, \mathbb{T})}$. This concludes the proof.

\end{proof}

\section{Finite set of generators for the (reduced) stated skein algebras}

\par For $\mathbf{\Sigma}= (\Sigma, \mathcal{P})$ a closed punctured surface, the fact that the skein algebra $\mathcal{S}_{\omega}(\mathbf{\Sigma})$ is finitely generated was first proved by Bullock in \cite{BullockGeneratorsSkein}. This results was refined in \cite{AbdielFrohman_SkeinFrobenius, FrohmanKania_SkeinRootUnity, SantharoubaneSkeinGenerators} by finding smaller sets of generators. The fact that the stated skein algebra associated to a punctured surface with one boundary component and exactly one puncture on the boundary is finitely generated was proved by Faitg in \cite{Faitg_LGFT_SSkein}.
The goal of this section is to prove that, for an arbitrary open punctured surface, both the associated stated skein algebras and their reduced versions are finitely generated. During all this section we fix a punctured surface $\mathbf{\Sigma}=(\Sigma, \mathcal{P})$ such that $\Sigma$ is connected and has non empty boundary. Let us sketch the strategy. We will exhibit finite sets $\mathbb{G}$ of arcs such that $\mathcal{S}_{\omega}(\mathbf{\Sigma})$ is generated by the stated arcs of the form $\alpha_{\varepsilon \varepsilon'}$ for $\alpha \in \mathbb{G}$. Figure \ref{fig_generators} shows such a set $\mathbb{G}$ and how the cutting arc relations \eqref{cutting_arc_rel} can be used to express a simple stated diagram $[D,s]\in \mathcal{B}$ as a polynomial in the generators $\alpha_{\varepsilon \varepsilon'}, \alpha \in \mathbb{G}$. The key property that $\mathbb{G}$ must satisfies is that its elements generate a certain groupoid $\Pi_1(\Sigma_{\mathcal{P}}, \mathbb{V})$. 

\begin{figure}[!h] 
\centerline{\includegraphics[width=10cm]{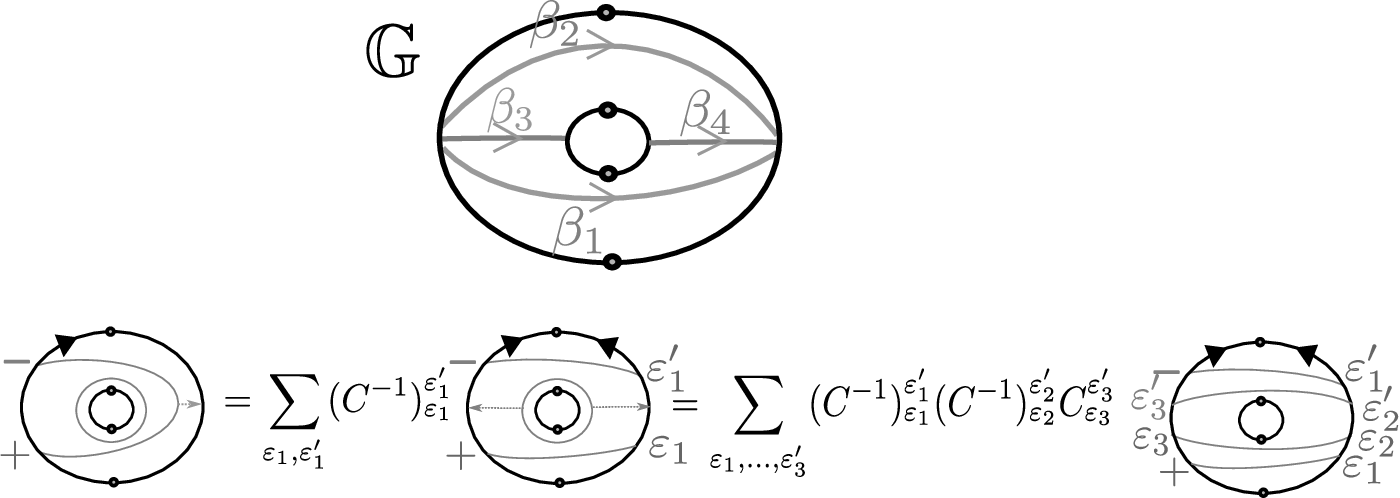} }
\caption{On the top: a set $\mathbb{G}=\{\beta_1, \ldots, \beta_4\}$ of generators of $\Pi_1(\Sigma_{\mathcal{P}}, \mathbb{V})$. On the bottom: an application of the cutting arc relations to express a simple stated diagram in terms of the elements of $\mathcal{A}^{\mathbb{G}}$. We draw dotted arrows to exhibit where we perform the cutting arc relations.} 
\label{fig_generators} 
\end{figure}

\begin{lemma}\label{lemma_gencurves}
The stated skein algebra $\mathcal{S}_{\omega}(\mathbf{\Sigma})$ is generated by the classes of stated arcs.
\end{lemma}

\begin{proof} Recall that  $\mathcal{S}_{\omega}(\mathbf{\Sigma})$ has basis the set of classes $[D,s]$ where $D$ is a simple diagram and $s$ is $\mathfrak{o}^+$ increasing. By Lemma \ref{lemma_product}, an induction on the number of connected components of $D$ shows that $[D,s]$ lies in the algebra generated by the classes of connected stated diagrams. It remains to show that the class of a closed curve $\gamma$ lies in the algebra generated by stated arcs. This is shown by isotoping such a $\gamma$ to bring one of its point close to some boundary arc and applying the cutting arc relation \eqref{cutting_arc_rel}.This concludes the proof.

\end{proof}

\par For each boundary arc $a$ of $\mathbf{\Sigma}$, fix a point $v_a \in a$ and let $\mathbb{V}$ be the set $\{v_a\}_a$. We denote by $\Pi_1(\Sigma_{\mathcal{P}}, \mathbb{V})$ the full subcategory of the fundamental groupoid $\Pi_1(\Sigma_{\mathcal{P}})$ generated by $\mathbb{V}$. Said differently, $\Pi_1(\Sigma_{\mathcal{P}}, \mathbb{V})$ is the small groupoid whose set of objects is $\mathbb{V}$ and such that a morphism (called path) $\alpha : v_1 \rightarrow v_2$ is a homotopy class of continuous map $\varphi_{\alpha} : [0,1] \rightarrow \Sigma_{\mathcal{P}}$ (a \textit{geometric representative} of $\alpha$) with $\varphi_{\alpha}(0)=v_1$ and $\varphi_{\alpha}(1)=v_2$. The composition is the concatenation of paths. For a path $\alpha : v_1 \rightarrow v_2$, we write $s(\alpha)=v_1$ (the source point) and $t(\alpha)=v_2$ (the target point) and $\alpha^{-1} : v_2 \rightarrow v_1$ the path with opposite orientation.  By abuse of notations, we denote by $\Pi_1(\Sigma_{\mathcal{P}}, \mathbb{V})$ the set of paths.

\begin{definition}\label{def_generators}
\begin{enumerate}
\item  A \textit{set of generators} for $\Pi_1(\Sigma_{\mathcal{P}}, \mathbb{V})$ is a set $\mathbb{G}$ of paths in $\Pi_1(\Sigma_{\mathcal{P}}, \mathbb{V})$ such that any path $\alpha \in \Pi_1(\Sigma_{\mathcal{P}}, \mathbb{V})$ decomposes as $\alpha= \alpha_1^{\varepsilon_1} \ldots \alpha_n^{\varepsilon_n}$ with $\varepsilon_i = \pm 1$ and $\alpha_i \in \mathbb{G}$. We also require that each path $\alpha \in \mathbb{G}$ is the homotopy class of some embedding  $\varphi_{\alpha} : [0,1] \rightarrow \Sigma_{\mathcal{P}}$ such that the images of the $\varphi_{\alpha}$  do not intersect outside $\mathbb{V}$. We will always assume implicitly that the geometric representatives $\varphi_{\alpha}$ is part of the data defining a set of generators.
\item For a path $\alpha : v_1 \rightarrow v_2$ and $\varepsilon, \varepsilon' \in \{-,+\}$, we denote by $\alpha_{\varepsilon \varepsilon'} \in \mathcal{S}_{\omega}(\mathbf{\Sigma})$ the class of the stated arc $(\alpha, s)$ where $s(v_1)=\varepsilon$ and $s(v_2)=\varepsilon'$. Set 
$$ \mathcal{A}^{\mathbb{G}} := \{ \alpha_{\varepsilon \varepsilon'} | \alpha \in \mathbb{G}, \varepsilon, \varepsilon' \in \{-,+\}  \} \subset \mathcal{S}_{\omega}(\mathbf{\Sigma}). $$
\end{enumerate}
 \end{definition}

\begin{example}\label{exemple_pres}
For any connected open punctured surface $\mathbf{\Sigma}$, the groupoid $\Pi_1(\Sigma_{\mathcal{P}}, \mathbb{V})$ admits a finite set of generators depicted in Figure \ref{fig_generators_final} and defined as follows. Denote by $a_0, \ldots, a_n$ the boundary arcs, by $\partial_0, \ldots, \partial_r$ the boundary components of $\Sigma$ with $a_0\subset \partial_0$ and write  $v_i:= a_i \cap \mathbb{V}$. 
Let $\overline{\Sigma}$ be the surface obtained from $\Sigma$ by gluing a disc along each boundary component $\partial_i$ for $1\leq i \leq r$, and choose $\alpha_1, \beta_1, \ldots, \alpha_g, \beta_g$ some paths in $\pi_1(\Sigma_{\mathcal{P}}, v_0)(=\mathrm{End}_{\Pi_1(\Sigma_{\mathcal{P}}, \mathbb{V})}(v_0)$), such that their images in $\overline{\Sigma}$ generate the free group $\pi_1(\overline{\Sigma}, v_0)$ (said differently, the $\alpha_i$ and $\beta_i$ are longitudes and meridians of $\Sigma$). For each inner puncture $p$ choose a peripheral curve $\gamma_p \in \pi_1(\Sigma_{\mathcal{P}}, v_0)$ encircling $p$ once and for each boundary puncture $p_{\partial}$ between two boundary arcs $a_i$ and $a_j$, consider the path $\alpha_{p_{\partial}} : v_i \rightarrow v_j $ represented by the corner arc in $p_{\partial}$. Eventually, for each boundary component $\partial_j$, with $1\leq j \leq r$, containing a boundary arc $a_{k_j} \subset \partial_j$,  choose a path $\delta_{\partial_j} : v_0 \rightarrow v_{k_j}$. The set 
$$\mathbb{G}':= \{ \alpha_i, \beta_i, \alpha_p, \delta_{\partial_j} | 1\leq i \leq g, p\in \mathcal{P}, 1\leq j \leq r\}$$
is a generating set for $\Pi_1(\Sigma_{\mathcal{P}}, \mathbb{V})$ and Figure \ref{fig_generators_final} represents a set of geometric representatives for $\mathbb{G}'$. Moreover each of its generators which is not one of the $\delta_{\partial_j}$ can be expressed as a composition of the other ones, therefore a set $\mathbb{G}$ obtained from $\mathbb{G}'$ by removing one of the element of the form $\alpha_i, \beta_i$ or $\gamma_p$, is still a generating set for  $\Pi_1(\Sigma_{\mathcal{P}}, \mathbb{V})$. 
Note that $\mathbb{G}$ has cardinality $2g-2+s+n_{\partial}$, where $g$ is the genus of $\Sigma$, $s:=|\mathcal{P}|$ is the number of punctures and $n_{\partial}:= |\pi_0(\partial \Sigma)|$ is the number of boundary components.

\begin{figure}[!h] 
\centerline{\includegraphics[width=10cm]{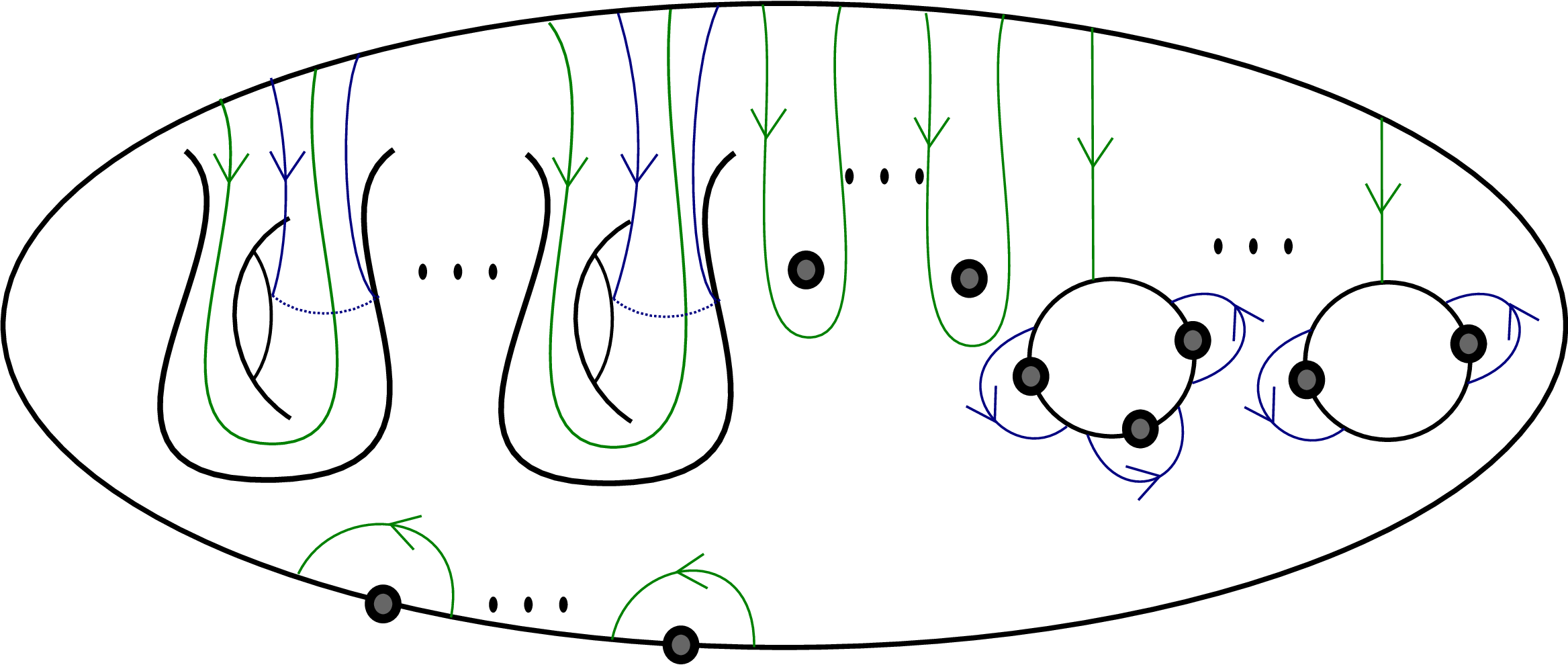} }
\caption{The geometric representatives of a set of generators for $\Pi_1(\Sigma_{\mathcal{P}}, \mathbb{V})$.} 
\label{fig_generators_final} 
\end{figure} 

\end{example}

\begin{proposition}\label{prop_generators}
If $\mathbb{G}$ is a set of generators of $\Pi_1(\Sigma_{\mathcal{P}}, \mathbb{V})$, then the set $ \mathcal{A}^{\mathbb{G}} $
generates $\mathcal{S}_{\omega}(\mathbf{\Sigma}) $ as an algebra. So the image of $\mathcal{A}^{\mathbb{G}}$ in $\overline{\mathcal{S}}_{\omega}(\mathbf{\Sigma})$ generates $\overline{\mathcal{S}}_{\omega}(\mathbf{\Sigma})$ as well.
\end{proposition}

\begin{proof}[Proof of Proposition \ref{prop_generators}]  Let $\widehat{\mathcal{A}}$ be the set of paths $\alpha$ such that for every $\varepsilon, \varepsilon' \in \{-,+\}$, the class $\alpha_{\varepsilon \varepsilon'}$ lies in the subalgebra of $\mathcal{S}_{\omega}(\mathbf{\Sigma}) $ generated by $\mathcal{A}^{\mathbb{G}}$. By definition one has $\mathbb{G} \subset \widehat{\mathcal{A}}$ and by Lemma \ref{lemma_gencurves} it suffices to prove that $\widehat{\mathcal{A}}= \Pi_1(\Sigma_{\mathcal{P}}, \mathbb{V})$. Note that $\alpha \in \widehat{\mathcal{A}}$ implies that $\alpha^{-1} \in \widehat{\mathcal{A}}$ because $(\alpha^{-1})_{\varepsilon' \varepsilon}$ and $\alpha_{\varepsilon \varepsilon'}$ represent the same element in $\mathcal{S}_{\omega}(\mathbf{\Sigma})$. Since $\mathbb{G}$ generates the groupoid, it suffices to prove that if $\alpha, \beta \in \widehat{\mathcal{A}}$ with $t(\alpha)=s(\beta)$, then $\alpha\beta \in \widehat{\mathcal{A}}$. The cutting arc relations \eqref{cutting_arc_rel} imply that $(\alpha\beta)_{\varepsilon \varepsilon'}$ lies in the subalgebra generated by elements $\alpha_{\varepsilon \mu}$ and $\beta_{\mu' \varepsilon'}$ for $\mu, \mu' \in \{-,+\}$, therefore  $\alpha\beta \in \widehat{\mathcal{A}}$.

\end{proof}

\begin{corollary} Let $\mathbf{\Sigma}=(\Sigma, \mathcal{P})$ be a punctured surface with $\Sigma$ connected of genus $g$ with $n_{\partial}\geq 1$ boundary components and with $s:=|\mathcal{P}|$ punctures such that $s_{\partial}$ of them are on the boundary of $\Sigma$ and $\mathring{s}$ are in the interior of $\Sigma$.
\begin{enumerate}
\item The stated skein algebra $\mathcal{S}_{\omega}(\mathbf{\Sigma})$ is generated, as an algebra, by a set of $8g-8+4s+4n_{\partial}$ elements. 
\item The reduced stated skein algebra $\overline{\mathcal{S}}_{\omega}(\mathbf{\Sigma})$ is generated, as an algebra, by a set of $8g-8+4\mathring{s}+3s_{\partial}+4n_{\partial}$ elements if either $g\geq 1$ or $\mathring{s}\geq 1$, or by a set of  $8g-7+4\mathring{s}+3s_{\partial}+4n_{\partial}$ elements elsewhere.
\end{enumerate}
\end{corollary}

\begin{proof} For $\mathbb{G}$ the set of generators in Example \ref{exemple_pres}, Proposition \ref{prop_generators} implies that $\mathcal{A}^{\mathbb{G}}$ generates $\mathcal{S}_{\omega}(\mathbf{\Sigma})$, so the first assertion follows from the computation of its cardinality: 
$$ |\mathcal{A}^{\mathbb{G}}| = 4|\mathbb{G}| = 8g-8+4s+4n_{\partial}$$.
Recall from Example \ref{exemple_pres} that $\mathbb{G}$ is obtained from a bigger set $\mathbb{G}'$ by removing one arbitrary generator, say $g\in \mathbb{G}'$. If either $g\geq 1$ or $\mathring{s}\geq 1$, one can suppose that $g$ is not the path of a corner arc $\alpha_{p_{\partial}}$. In this case, the set $\mathcal{A}^{\mathbb{G}}$ contains $s_{\partial}$ bad arcs. The image in the reduced stated skein algebra of the set obtained from $\mathcal{A}^{\mathbb{G}}$ by removing these bad arcs generates $\overline{\mathcal{S}}_{\omega}(\mathbf{\Sigma})$ and has cardinality $|\mathcal{A}^{\mathbb{G}}|-s_{\partial} = 8g-8+4\mathring{s}+3s_{\partial}+4n_{\partial}$. In the case where $g=\mathring{s}=0$, $\mathbb{G}$ is necessarily obtained from $\mathbb{G}'$ by removing the path of a corner arc, so $\mathcal{A}^{\mathbb{G}}$ contains $s_{\partial}-1$ bad arcs and we conclude likewise.

\end{proof}

\section{The (reduced) stated skein algebras are finitely generated over their centers}

For a punctured surface $\mathbf{\Sigma}=(\Sigma, \mathcal{P})$ such that $\Sigma$ is closed, and $\omega$ a root of unity of odd order $N>1$, the fact that the Kauffman-bracket skein algebra $\mathcal{S}_{\omega}(\mathbf{\Sigma})$ is finitely generated as a module over its center was proved in \cite{AbdielFrohman_SkeinFrobenius}. The goal of this subsection is to extend this result for both the stated skein algebras and their reduced versions associated to open punctured surfaces, namely to prove the 

\begin{proposition}\label{prop_finitely_gen_center}
For $\omega$ a root of unity of odd order $N>1$ and $\mathbf{\Sigma}=(\Sigma, \mathcal{P})$ a punctured surface with $\Sigma$ connected with non-empty boundary, both $\mathcal{S}_{\omega}(\mathbf{\Sigma})$ and $\overline{\mathcal{S}}_{\omega}(\mathbf{\Sigma})$ are finitely generated as modules over their centers.
\end{proposition}

We will then deduce the 

\begin{proposition}\label{prop_almost_Azumaya} 
The algebras $\mathcal{S}_{\omega}(\mathbf{\Sigma})$ and  $\overline{\mathcal{S}}_{\omega}(\mathbf{\Sigma})$ are affine almost-Azumaya.
\end{proposition}

\begin{proof}
The  algebras $\mathcal{S}_{\omega}(\mathbf{\Sigma})$ and  $\overline{\mathcal{S}}_{\omega}(\mathbf{\Sigma})$ are: $(1)$  prime algebras by Remark \ref{remark_prime}, $(2)$ finitely generated as  algebras by Proposition \ref{prop_generators}, $(3)$ finitely generated as  module over their centers by Proposition \ref{prop_finitely_gen_center}, therefore they are affine almost-Azumaya.
\end{proof}

We cut the proof of Proposition \ref{prop_finitely_gen_center} in five lemmas. For an arc $\alpha$, we denote by $\mathcal{A}(\alpha) \subset \mathcal{S}_{\omega}(\mathbf{\Sigma})$ the subalgebra generated by the classes of stated arcs $\alpha_{\varepsilon \varepsilon'}$. For $\beta \in \mathbb{G}$, we also denote by  $\mathcal{A}(\beta)$
the algebra associated to its underlying arc given by the geometric representative.

\begin{lemma}\label{l1}
Let $\mathbb{G}= \{ \beta_1, \ldots, \beta_n \}$ be a set of generators of $\Pi_1(\Sigma_{\mathcal{P}}, \mathbb{V})$ ordered arbitrarily. Then $\mathcal{S}_{\omega}(\mathbf{\Sigma})$ is linearly spanned by elements of the form $x_1x_2 \ldots x_n$, where $x_i \in \mathcal{A}(\beta_i)$.
\end{lemma}

\begin{proof}
By Proposition \ref{prop_generators}, $\mathcal{S}_{\omega}(\mathbf{\Sigma})$ is linearly spanned by monomials of the form $x=(\alpha_{i_1})_{\varepsilon_1 \varepsilon'_1} \ldots (\alpha_{i_k})_{\varepsilon_k \varepsilon'_k}$ so we need to show that each such $x$ is a linear combination of monomials $x'$ of the same form where in $x'$ one has  $a<b$ implies $i_a < i_b$. Recall that we suppose that for $\beta_i \neq \beta_j \in \mathbb{G}$, the geometric representatives of $\beta_i$ and $\beta_j$ do not intersect, so by the height exchange relations \eqref{height_exchange_rel}, any element $(\beta_i)_{\varepsilon_i \varepsilon'_i} (\beta_j)_{\varepsilon_j \varepsilon'_j}$ is a linear combination of elements of the form $(\beta_j)_{\mu_j \mu'_j} (\beta_i)_{\mu_i \mu_i'}$ so the results follows by induction.

\end{proof}

To prove Proposition \ref{prop_finitely_gen_center}, we will use the fact, derived from Theorem \ref{theorem_center_skein}, that , when $\omega$ is a root of unity of odd order $N>1$ and for any stated arc $\alpha_{\varepsilon \varepsilon'}$, the elements $\alpha_{\varepsilon \varepsilon'}^{(N)}$ are central in $\mathcal{S}_{\omega}(\mathbf{\Sigma})$. Let $Z_{\alpha} \subset \mathcal{A}(\alpha)$ the (commutative) subalgebra generated by such elements $\alpha_{\varepsilon \varepsilon'}^{(N)}$. We want to show that $\mathcal{A}(\alpha)$ is finitely generated as a $Z_{\alpha}$-module. Proposition \ref{prop_finitely_gen_center} will then follow from this fact and from Lemma \ref{l1}.

\begin{lemma}\label{l2}
Suppose that $\omega$ is a root of unity of odd order $N>1$ and $\alpha$ an arc whose endpoints lye in two different boundary arcs. Then $\mathcal{A}(\alpha)$ is generated as a $Z_{\alpha}$-module by the finite set
$$S_{\alpha} := \{ \alpha_{++}^a \alpha_{+-}^b \alpha_{--}^c, 0\leq a,b,c \leq N-1 \} \cup \{ \alpha_{++}^a \alpha_{-+}^b \alpha_{--}^c, 0 \leq a,b,c \leq N-1  \}$$.
\end{lemma}

\begin{proof}
When the two endpoints of $\alpha$ lye in distinct boundary arcs, it is well known (see \textit{e.g.} \cite{KojuQuesneyClassicalShadows}) that $\mathcal{A}(\alpha)$ is isomorphic to $\mathbb{C}_q[\mathrm{SL}_2]$ where $q:= A^2$, namely that its elements satisfy the relations: 
\begin{align*}\label{relbigone}
\alpha_{++}\alpha_{+-} &= q^{-1}\alpha_{+-}\alpha_{++} & \alpha_{++}\alpha_{-+}&=q^{-1}\alpha_{-+}\alpha_{++}
\\ \alpha_{--}\alpha_{+-} &= q\alpha_{+-}\alpha_{--} & \alpha_{--}\alpha_{-+}&=q\alpha_{-+}\alpha_{--}
\\ \alpha_{++}\alpha_{--}&=1+q^{-1}\alpha_{+-}\alpha_{-+} &  \alpha_{--}\alpha_{++}&=1 + q\alpha_{+-}\alpha_{-+} 
\\ \alpha_{-+}\alpha_{+-}&=\alpha_{+-}\alpha_{-+} & &
\end{align*}

We deduce from those relations that the set
$$ B := \{ \alpha_{++}^a \alpha_{+-}^b \alpha_{--}^c, a,b,c \geq 0 \} \cup \{ \alpha_{++}^a \alpha_{-+}^b \alpha_{--}^c, a,b,c \geq 0 \}$$
linearly spans $\mathcal{A}(\alpha)$ (actually it is a PBW basis of $\mathbb{C}_q[\mathrm{SL}_2]$). Since the two endpoints of $\alpha$ lye in distinct boundary arcs, one has $\alpha_{\varepsilon \varepsilon'}^{(N)} = (\alpha_{\varepsilon \varepsilon'})^N$, therefore the set $S_{\alpha}$
spans $\mathcal{A}(\alpha)$ as a $Z_{\alpha}$-module.

\end{proof}

\par For an arc $\alpha$ with both endpoints, say $v$ and $w$,  in the same boundary, we need to work a little more.
Let us introduce some terminology. Recall from the end of Section $2.1$ that we defined a tangle  $\alpha^{(n)}$ made of $n$ parallel copies $\alpha^1 \cup \ldots \cup \alpha^n$ of $\alpha$, pushed along the framing direction, with the convention that if the heights of the endpoints are such that $h(v)<h(w)$ then 
 the height order of $\alpha^{(n)}$ is $h(v_1)< \ldots < h(v_n)<h(w_1) < \ldots < h(w_n)$. Let us denote by $\alpha^n = \alpha'^1 \cup \ldots \cup \alpha'^n$ the tangle also made of $n$ parallel copies of $\alpha$ but this time, we require the height order $h(v_1) <h(w_1) < \ldots < h(v_n) <h(w_n)$. Figure \ref{parallel_arcs} illustrates the difference between $\alpha^{(n)}$ and $\alpha^n$. Note that: $(1)$ a class $[\alpha^n, s]$ is a product of stated arcs $\alpha_{\varepsilon \varepsilon'}$ and $(2)$ $\alpha^n$ and $\alpha^{(n)}$ differ by an isotopy than does not preserve the height order.

\begin{figure}[!h] 
\centerline{\includegraphics[width=10cm]{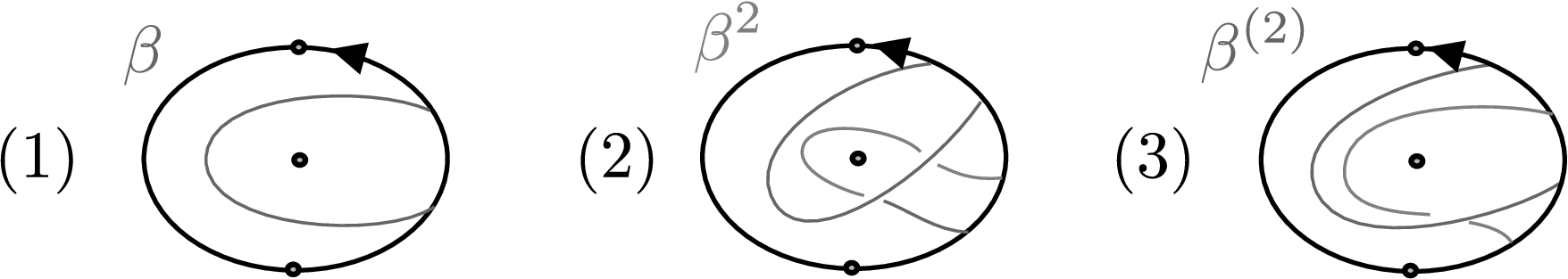} }
\caption{$(1)$ an arc $\beta$, $(2)$ a diagram of $\beta^2$, $(3)$ a diagram of $\beta^{(2)}$.} 
\label{parallel_arcs} 
\end{figure} 

For $\bm{\varepsilon}= (\varepsilon_1, \ldots, \varepsilon_n), \bm{\varepsilon}'= (\varepsilon'_1, \ldots, \varepsilon'_n)\in \{ -, +\}^n$, we denote by $s^{\bm{\varepsilon}, \bm{\varepsilon'}}$ the state on $\alpha^{(n)}$ defined by $s^{\bm{\varepsilon}, \bm{\varepsilon'}}(v_i)= \varepsilon_i$ and $s^{\bm{\varepsilon}, \bm{\varepsilon'}}(w_i)= \varepsilon'_i$. 
\begin{definition}
A state on $\alpha^{(n)}$ will be called \textit{ordered} if it is of the form $s^{\bm{\varepsilon}, \bm{\varepsilon'}}$ where $i<j$ implies both $\varepsilon_i \leq \varepsilon_j$ and $\varepsilon'_i\leq \varepsilon'_j$ (recall we use the convention $-<+$).
\end{definition}

\begin{lemma}\label{l3}
The space $\mathcal{A}(\alpha)$ is linearly spanned by elements $[\alpha^{(n)}, s]$ where $s$ is ordered.
\end{lemma}

\begin{proof} By definition, the space $\mathcal{A}(\alpha)$ is spanned by elements of the form $[\alpha^n, s]$. Since $\alpha^n$ and $\alpha^{(n)}$ only differ by an isotopy which modify the height order, the height exchange relation \eqref{height_exchange_rel} implies that any elements $[\alpha^n ,s ]$ is a linear combination of elements $[\alpha^{(n)}, s']$ for some states $s'$ on $\alpha^{(n)}$. It remains to show that each element $[\alpha^{(n)}, s']$ is a linear combination of elements $[\alpha^{(n)}, s]$ with $s$ ordered. If $s'$ is not ordered, there exists $i$ such that either $(s'(v_i), s'(v_{i+1}))= (+,-)$ or $(s'(w_i), s'(w_{i+1})) = (+,-)$ (or both). Applying the cutting arc (defining) relation \eqref{cutting_arc_rel}
$$ 
\heightexch{->}{-}{+} =
A^2
\heightexch{->}{+}{-}
+ \omega
\heightcurve.$$
and noting that the term $\heightcurve$ is null by the trivial arc relation \eqref{trivial_arc_rel}, we show by induction that there exists an ordered state $s$ and an integer $k$ such that $[\alpha^{(n)}, s']= A^{2k} [\alpha^{(n)}, s]$. This concludes the proof.

\end{proof}

For $[\alpha^{(n)}, s]$, with $s$ ordered and $n>N$ big enough, we would like to have a factorization of the form $[\alpha^{(n)}, s]= [\alpha^{(n-N)}, s'] \alpha_{\varepsilon \varepsilon'}^{(N)}$. The fact that $\mathcal{A}(\alpha)$ is finitely generated as a $Z_{\alpha}$-module will easily follow. The key step to provide such a factorization is the following:

\begin{lemma}\label{l4} Suppose that $A^N=1$ for some $N>1$ (not necessarily odd). Then one has the following skein relations:
\begin{equation}\label{eq_central}
 \adjustbox{valign=c}{\includegraphics[width=1.6cm]{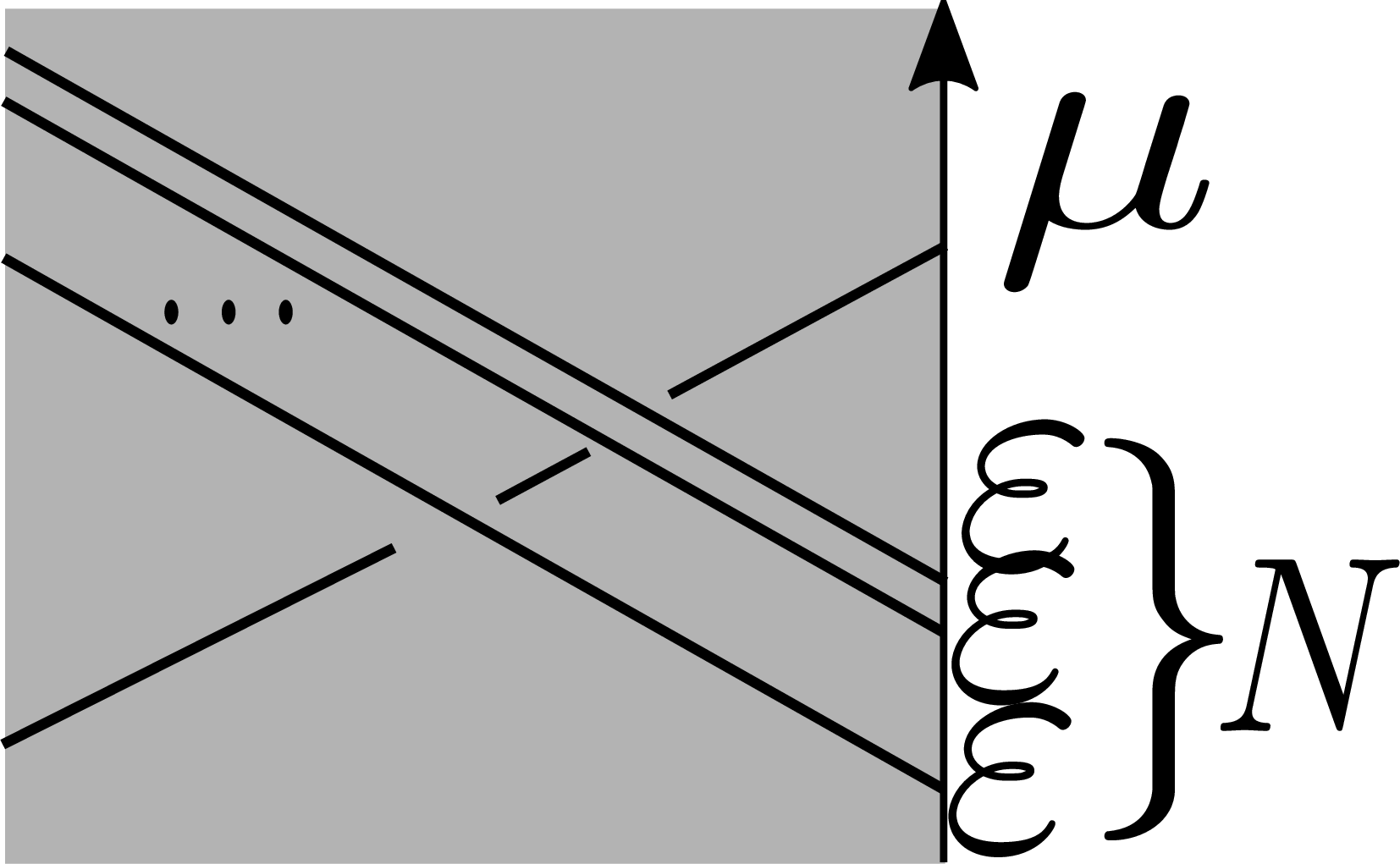}} =  \adjustbox{valign=c}{\includegraphics[width=1.6cm]{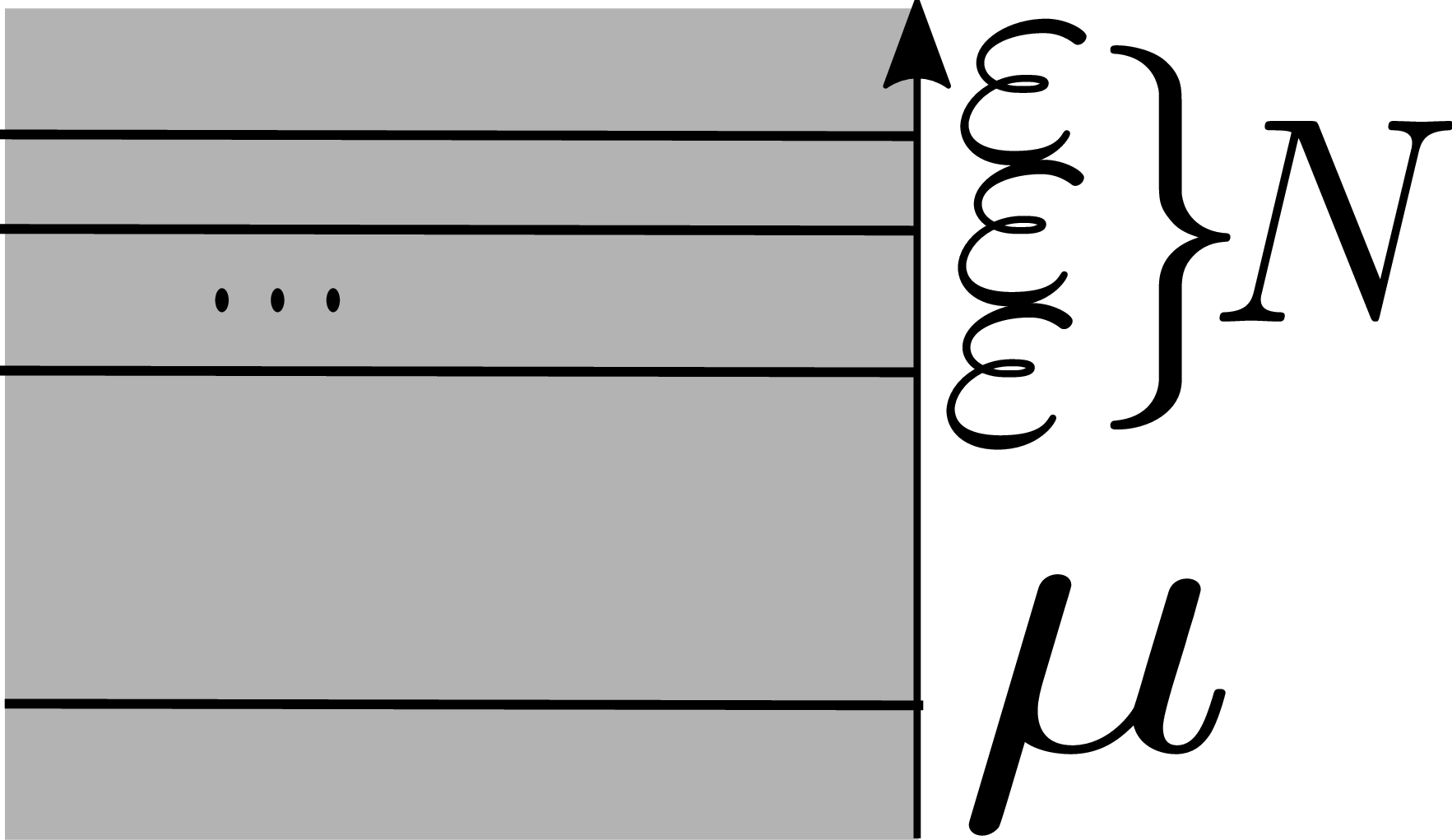}} = \adjustbox{valign=c}{\includegraphics[width=1.6cm]{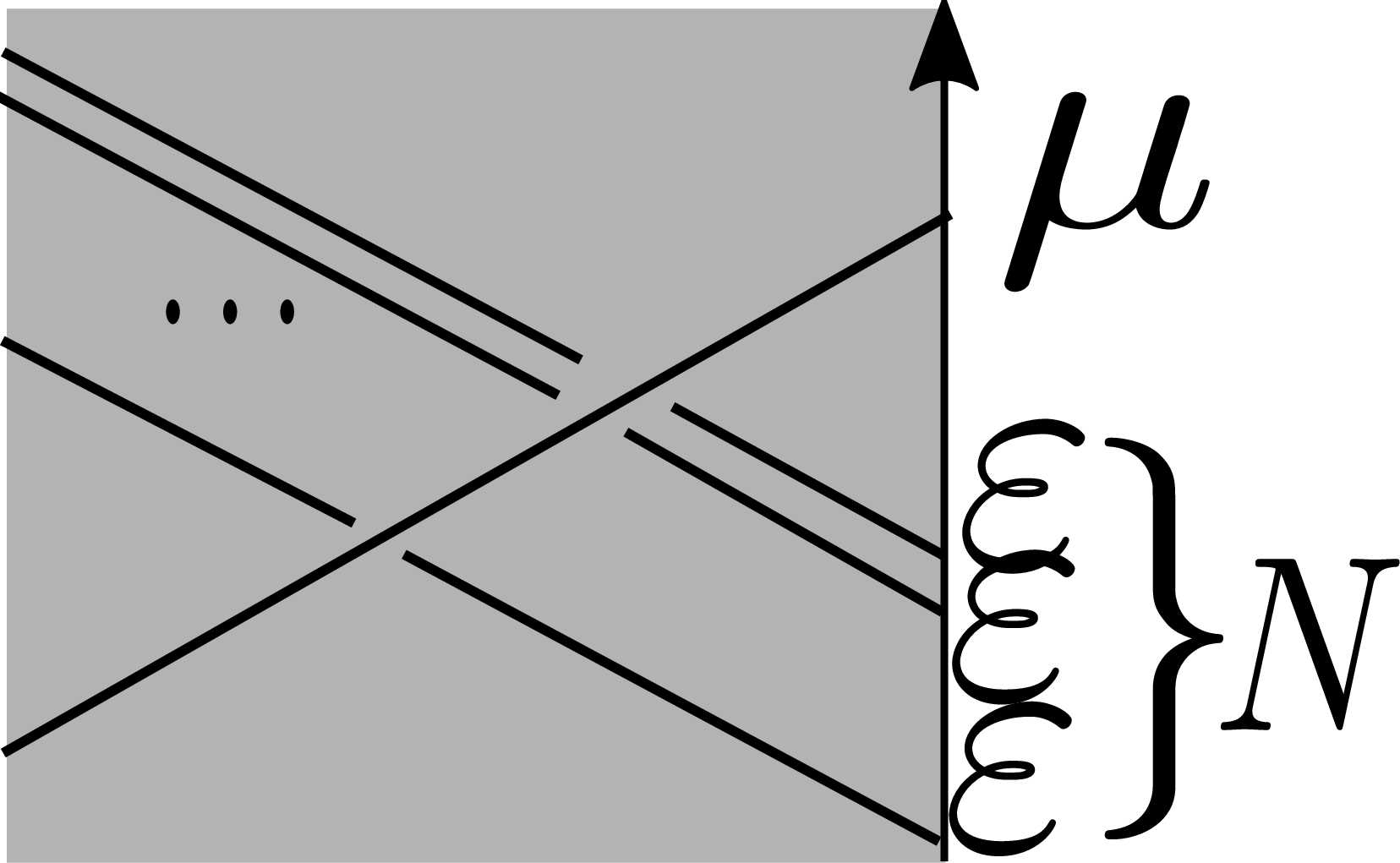}}, 
\end{equation}
where the $N$ strands in Equation \eqref{eq_central} represent the components of a stated tangle of the form $\alpha_{\varepsilon \varepsilon'}^{(N)}$ (so Equation \eqref{eq_central} is not a local skein relation since it depends on how is the tangle outside the drawn box). Therefore, if $\mathcal{T}=(T,s)$ and $\mathcal{T}'= (T',s')$ are two stated tangles such that $\mathcal{T}=\mathcal{T}'\cup \alpha_{\varepsilon \varepsilon'}^{(N)}$, one has the factorization $[\mathcal{T}] = [\mathcal{T}'] \alpha_{\varepsilon \varepsilon'}^{(N)}$ in $\mathcal{S}_{\omega}(\mathbf{\Sigma})$.

\end{lemma}

\begin{proof}
We prove the equality $\adjustbox{valign=c}{\includegraphics[width=1.3cm]{D1.eps}} =  \adjustbox{valign=c}{\includegraphics[width=1.3cm]{D2.eps}}$, the second one is proved similarly and left to the reader. First note that the height exchange relations \eqref{height_exchange_rel} imply
$$ \adjustbox{valign=c}{\includegraphics[width=1cm]{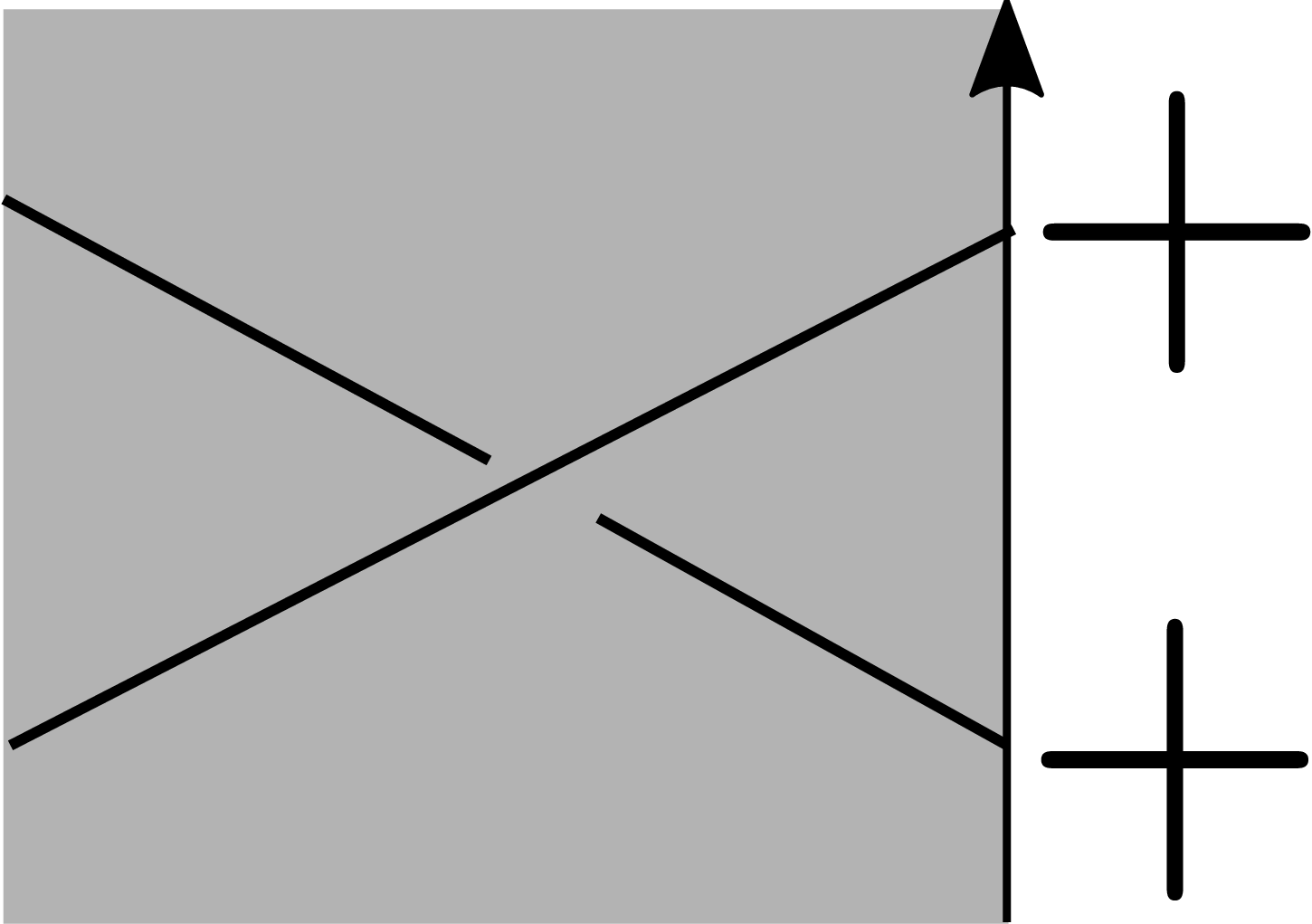}} = A \heightexch{->}{+}{+}, \quad   \adjustbox{valign=c}{\includegraphics[width=1cm]{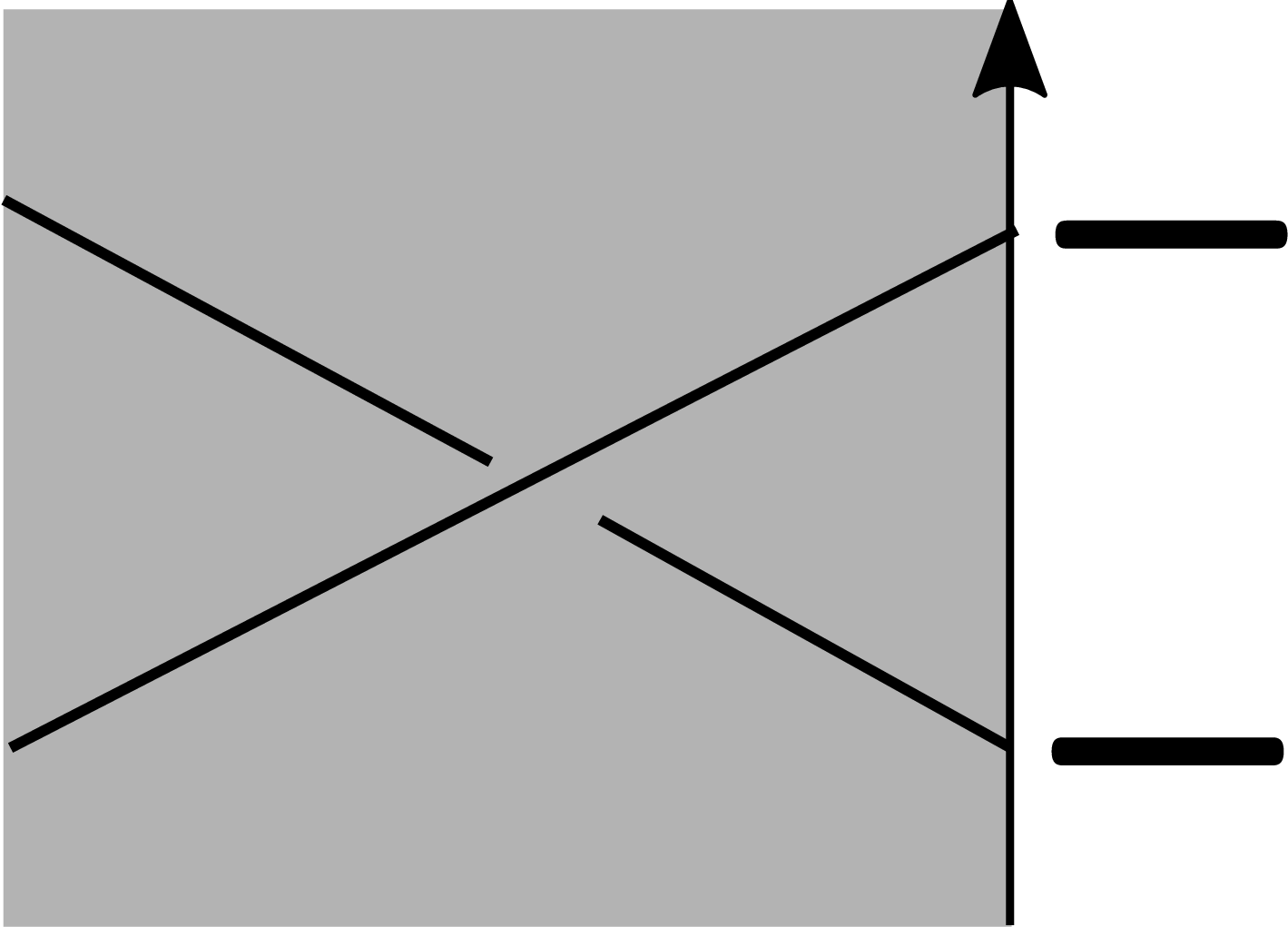}} = A \heightexch{->}{-}{-},  \quad \adjustbox{valign=c}{\includegraphics[width=1cm]{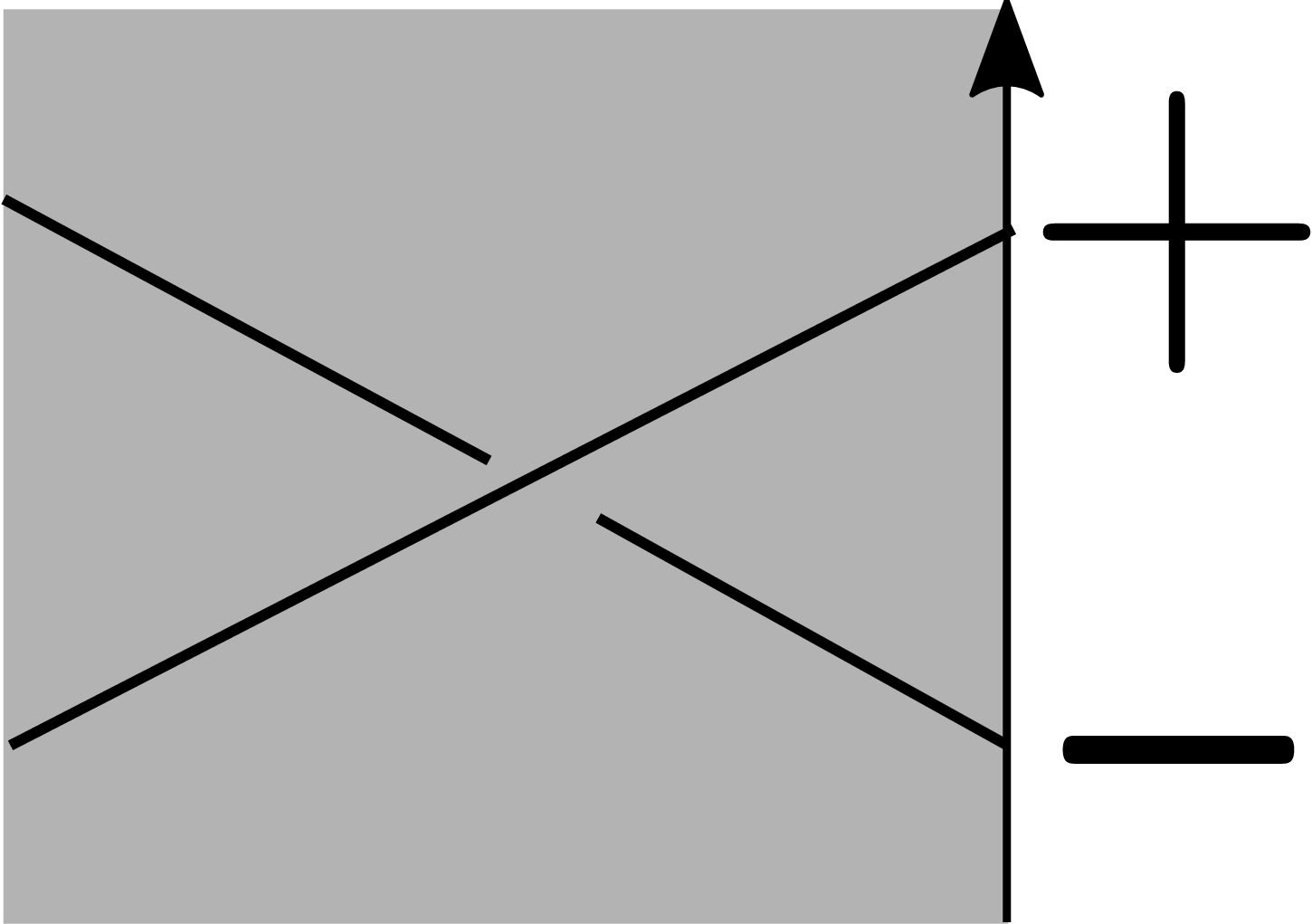}} = A^{-1} \heightexch{->}{-}{+}, 
$$
so one has $\adjustbox{valign=c}{\includegraphics[width=1.3cm]{D1.eps}} = A^N \adjustbox{valign=c}{\includegraphics[width=1.3cm]{D2.eps}}$ if $\varepsilon = \mu$ and $\adjustbox{valign=c}{\includegraphics[width=1.3cm]{D1.eps}} = A^{-N} \adjustbox{valign=c}{\includegraphics[width=1.3cm]{D2.eps}}$ if $(\varepsilon, \mu) = (-,+)$. In each of these cases, we conclude using $A^N=1$. It remains to deal with the case where $(\varepsilon, \mu)=(+,-)$. Using the height exchange relation 
$$  \adjustbox{valign=c}{\includegraphics[width=1cm]{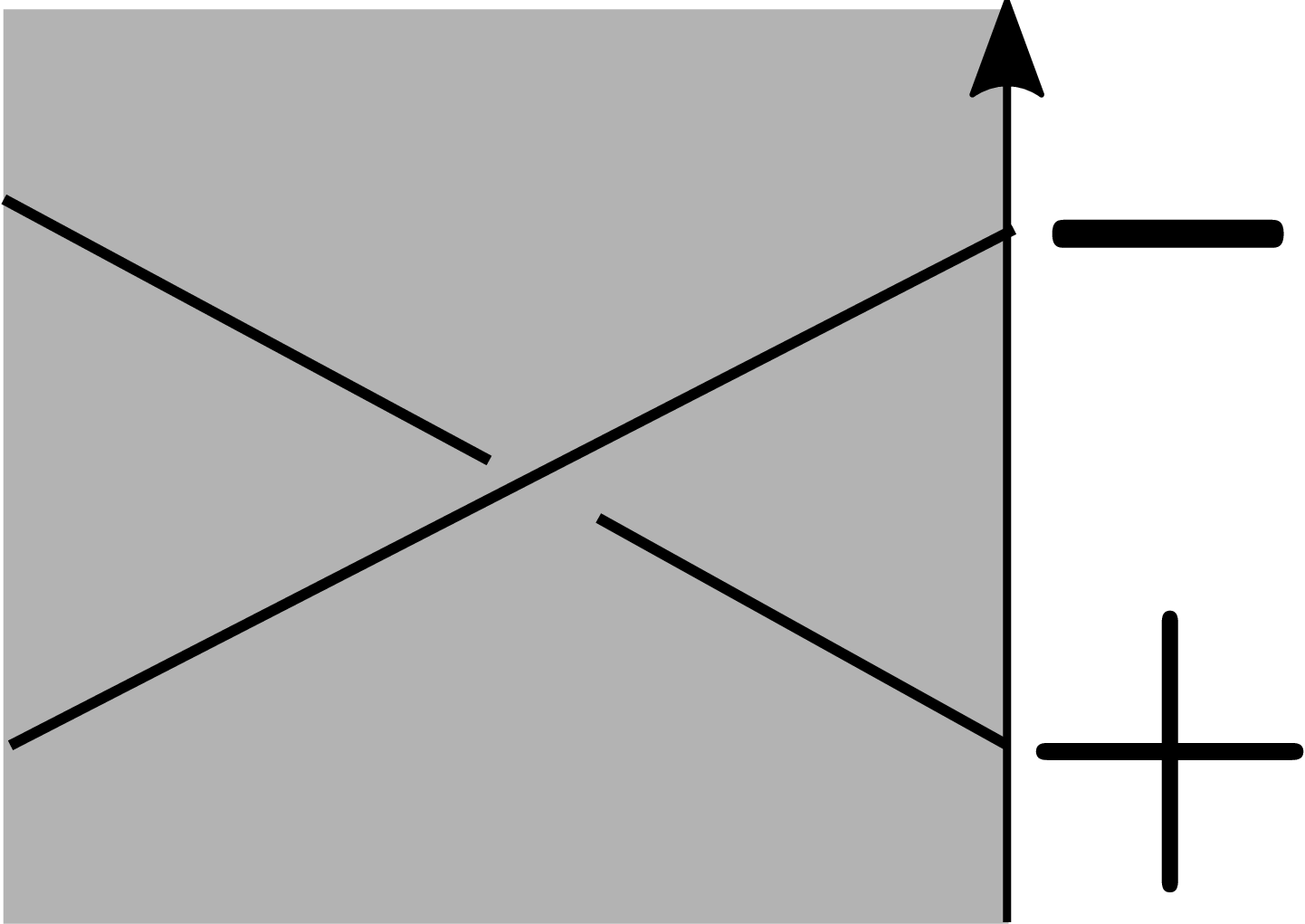}} = A^{-1} \heightexch{->}{+}{-} +(A-A^{-3})  \heightexch{->}{-}{+}$$ 
repeatedly, one has the equalities
$$
\adjustbox{valign=c}{\includegraphics[width=1.6cm]{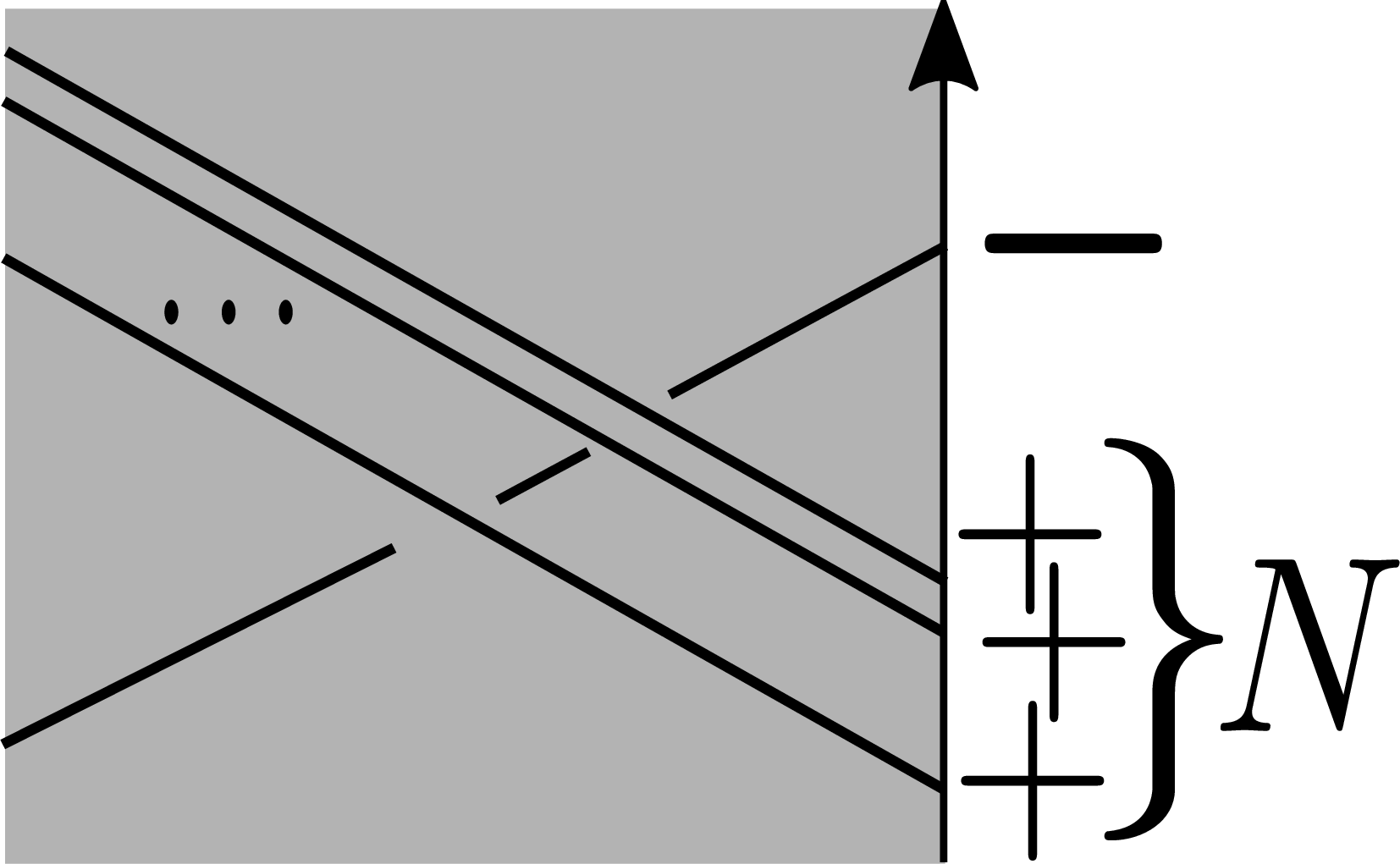}} = A^{-N} \adjustbox{valign=c}{\includegraphics[width=1.6cm]{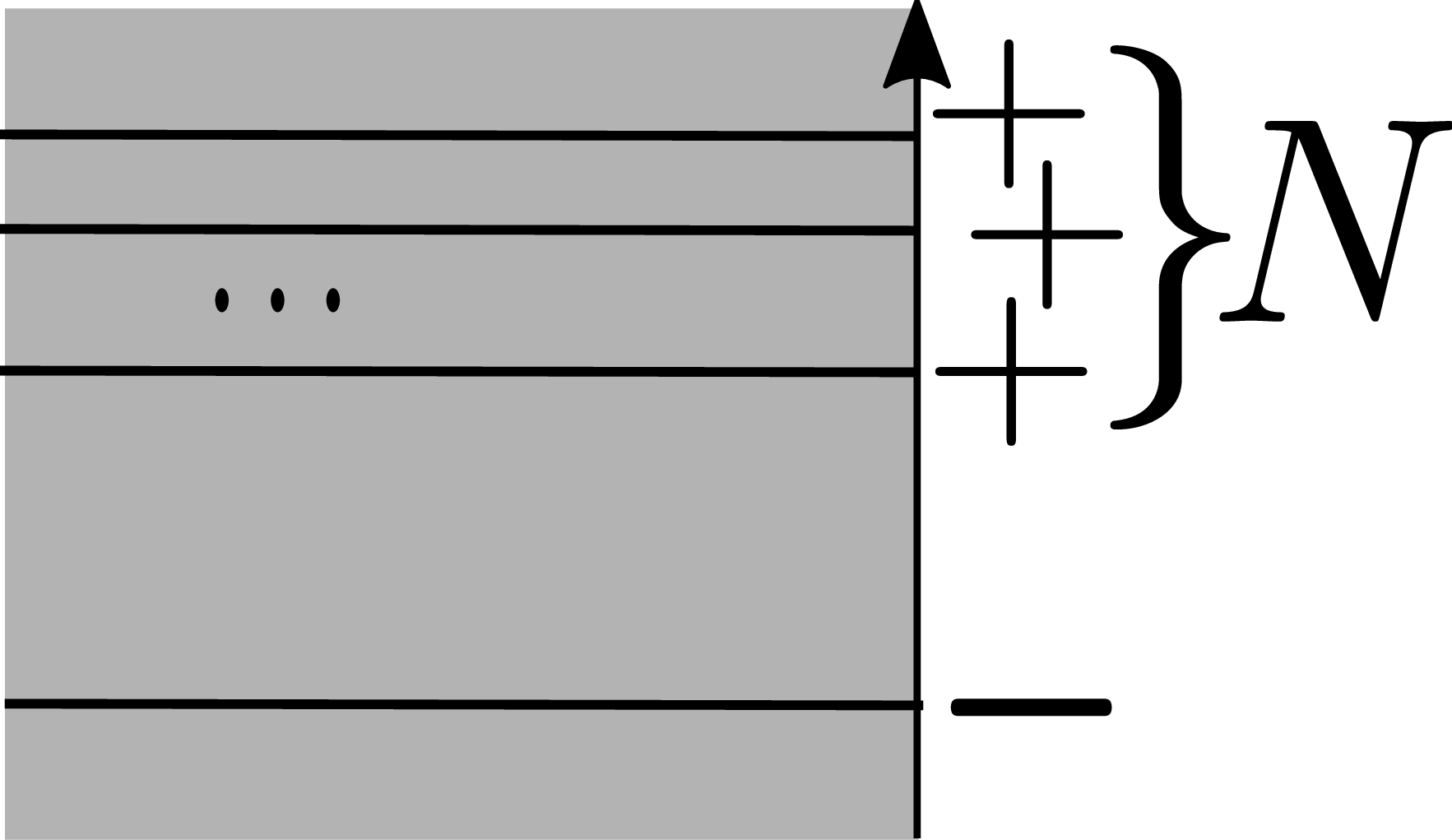}} +(A-A^{-3})\left( \sum_{i=1}^N  \adjustbox{valign=c}{\includegraphics[width=1.6cm]{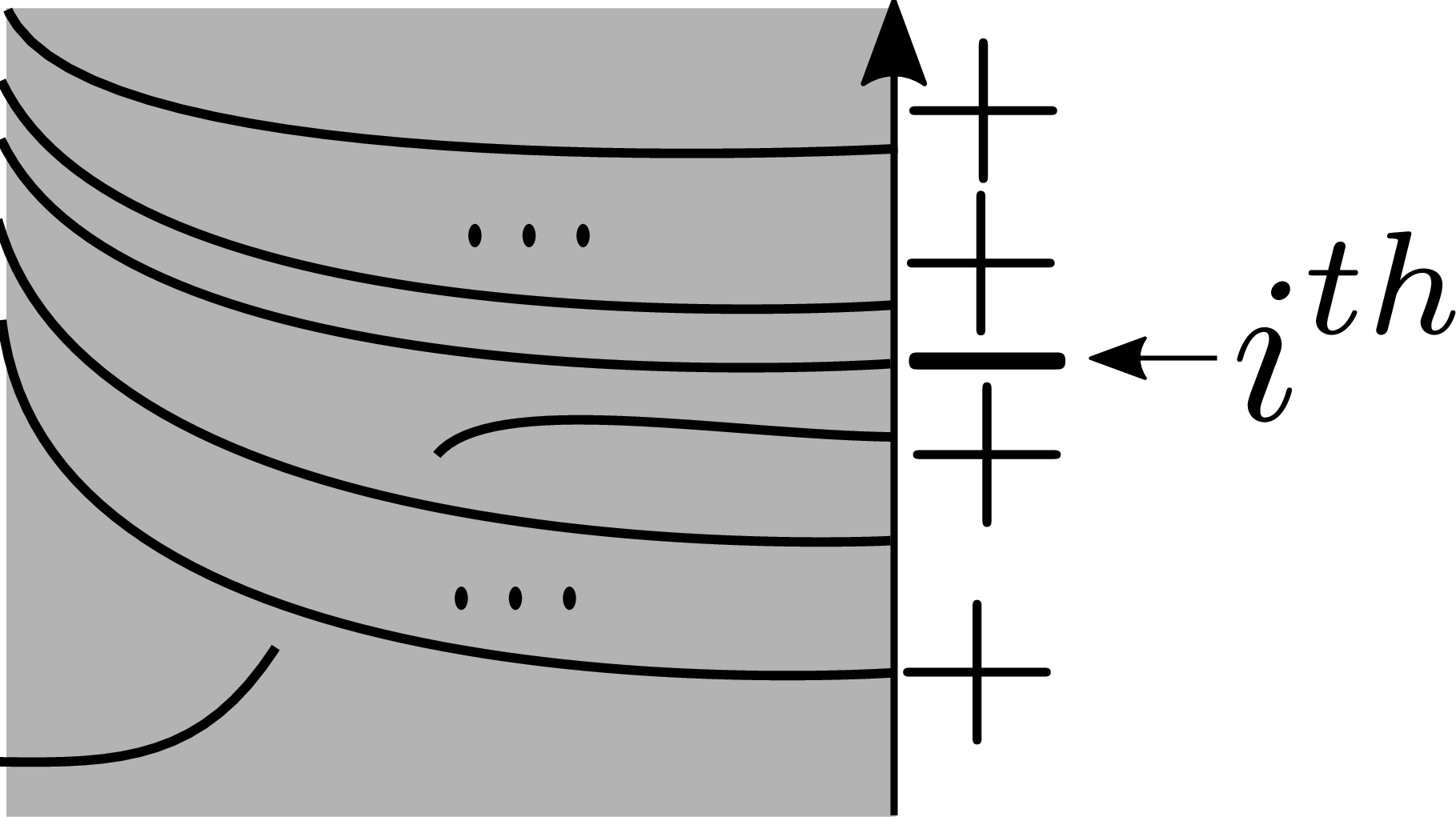}} \right).
$$
Here, in $ \adjustbox{valign=c}{\includegraphics[width=1.6cm]{F3.eps}}$, the $i^{th}$ point (starting at the top) has state $-$ and the other ones $+$ and the $(i+1)^{th}$ point is the endpoint of the strand which is not part of $N$ parallel copies of $\alpha$. Using the relation $ \adjustbox{valign=c}{\includegraphics[width=1cm]{C1.eps}} = A \heightexch{->}{+}{+}$, one see that 
$$ \adjustbox{valign=c}{\includegraphics[width=1.6cm]{F3.eps}} = A^{N-i} \adjustbox{valign=c}{\includegraphics[width=1.6cm]{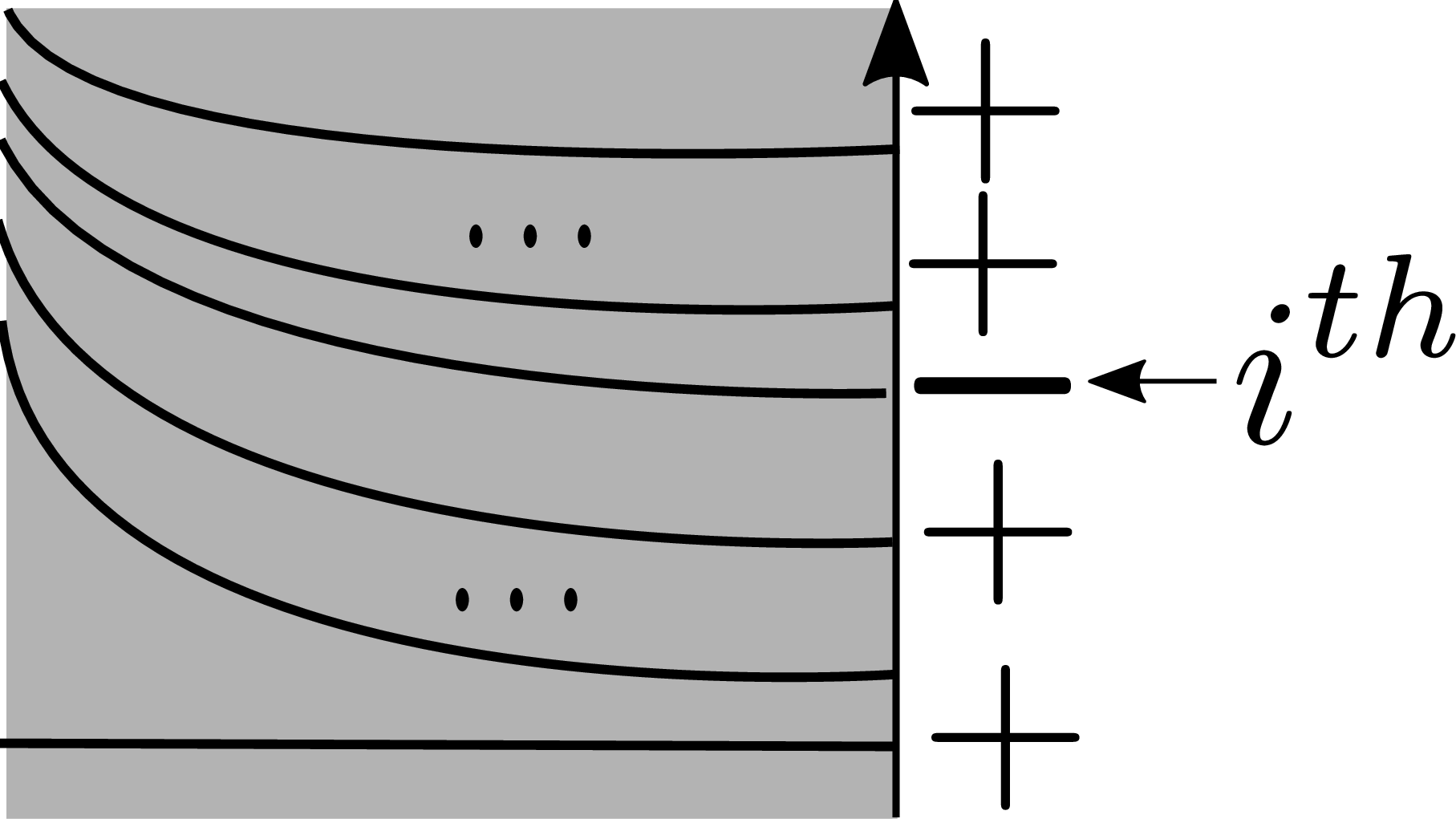}} =: A^{N-i} x_i.$$
Using the cutting arc (defining) relation 
$$\heightexch{->}{-}{+} =
A^2
\heightexch{->}{+}{-}
+ \omega
\heightcurve $$
in the $i^{th}$ and $(i+1)^{th}$ strand of $x_i$, and noting that the term $\heightcurve$ vanishes thanks to the trivial arc relation \ref{trivial_arc_rel} (this is the only moment of the proof where we use the fact that the $N$ strands in Equation \eqref{eq_central} belong to an element $\alpha_{\varepsilon \varepsilon'}^{(N)}$), one finds that $x_i = A^2 x_{i+1}$, so $x_i= A^{2(N-i)}x_N$.  Therefore, using that $A^N=1$,  one finds 
$$ \adjustbox{valign=c}{\includegraphics[width=1.6cm]{F1.eps}} = A^{-N} \adjustbox{valign=c}{\includegraphics[width=1.6cm]{F2.eps}} +(A-A^{-3})\left( \sum_{i=1}^N A^{3(N-i)} \right) x_N  =  \adjustbox{valign=c}{\includegraphics[width=1.6cm]{F2.eps}}.$$
This concludes the proof.

\end{proof}

\begin{lemma}\label{l5}
If $\omega$ is a root of unity of odd order $N>1$, then $\mathcal{A}(\alpha)$ is linearly spanned by the set
$$F= \{ [\alpha^{(n)}, s] (\alpha^{(N)}_{++})^{n_1} (\alpha^{(N)}_{+-})^{n_2}(\alpha^{(N)}_{-+})^{n_3} (\alpha^{(N)}_{--})^{n_4}, s \mbox{ ordered }, 0\leq n \leq 3N-3,  n_1, n_2, n_3, n_4 \geq 0 \}. $$
Therefore, $\mathcal{A}(\alpha)$ is generated as a $K_{\alpha}$-module, by the finite set 
$$S_{\alpha} := \{ [\alpha^{(n)}, s], s \mbox{ ordered }, 0\leq n \leq 3N-3\}.$$
\end{lemma}

Note that the set $S_{\alpha}$ found in Lemma \ref{l2} in the case where the endpoints of $\alpha$ lye in distinct boundary arcs has exactly the same expression than the set $S_{\alpha}$ in Lemma \ref{l5}.

\begin{proof}
The second assertion is an immediate consequence of the first one. By Lemma \ref{l3}, $\mathcal{A}(\alpha)$ is linearly spanned by elements $[\alpha^{(n)}, s]$ with $s$ ordered and $n\geq 0$. We show that each such element is equal to an element of $F$. If $n\leq 3N-3$, this is trivial. Else, we show the result by decreasing induction on $n$. Suppose that $n\geq 3N-2$. Write $s= s^{\bm{\varepsilon}, \bm{\varepsilon'}}$ and define the integers $0 \leq i_0, j_0 \leq n$  by the conditions $\varepsilon_i = -$ if $i\leq i_0$, $\varepsilon_i=+$ if $i>i_0$ and $\varepsilon'_j=-$ if $j\leq j_0$ and $\varepsilon'_j=+$ if $j>j_0$. Suppose that $i_0\leq j_0$, the case $j_0 <i_0$ is similar and left to the reader. Set $a:=i_0, b:= j_0-i_0$ and $c:= n-j_0$ such that $n=a+b+c$. Since $n\geq 3N-2$, at least one of the integer $a,b$ or $c$ is bigger or equal than $N$. Using Lemma \ref{l4}, one has a factorization of the form $[\alpha^{(n)}, s]= [\alpha^{(n-N)}, s'] \alpha_{\mu \mu'}^{(N)}$, where we can choose $(\mu, \mu') = (+,+)$ if $a\geq N$, $(\mu, \mu')= (-,+)$ if $b\geq N$ or $(\mu, \mu')=(-,-)$ if $c\geq N$. We conclude by decreasing induction on $n$. 

\end{proof}

\begin{proof}[Proof of Proposition \ref{prop_finitely_gen_center}]
By Lemmas \ref{l2} and \ref{l5}, for each $\beta \in \mathbb{G}$, one has a finite set $S_{\beta}$ which generates $\mathcal{A}(\alpha)$ as a $Z_{\alpha}$-module. Lemma \ref{l1} implies that the set 
$S:= \{ x_1 \ldots x_n, x_i \in S_{\beta_i} \}$ generates $\mathcal{S}_{\omega}(\mathbf{\Sigma})$ as a module over its center. Therefore the image of $S$ in $\overline{\mathcal{S}}_{\omega}(\mathbf{\Sigma})$ generates the reduced stated skein algebra over its center as well.
\end{proof}

\begin{remark}
Using the set $\mathbb{G}$ of Example \ref{exemple_pres}, the proof of Proposition \ref{prop_finitely_gen_center} provides an explicit set $S$ of generators for $\mathcal{S}_{\omega}(\mathbf{\Sigma})$ as a module over its center from which we deduce an upper bound 
$$|S|=  |S_{\beta}|^{|\mathbb{G}|} = (2N^3-N^2)^{2g-2+s+n_{\partial}}$$ for the rank of $\mathcal{S}_{\omega}(\mathbf{\Sigma})$ as a module over this center. Therefore every irreducible representation of $\mathcal{S}_{\omega}(\mathbf{\Sigma})$ has dimension smaller or equal to the square root $\sqrt{|S|}$. However, this upper bound is far from been optimal. Indeed, we only considered the central elements of the form $\alpha_{\varepsilon \varepsilon'}^{(N)}$ but the center of $\mathcal{S}_{\omega}(\mathbf{\Sigma})$ is much bigger. For instance, the class of a peripheral curve around an inner puncture is a central element that we did not consider. 
\end{remark}

\section{Valuations on reduced stated skein algebras}

\subsection{A strict valuation}

\par Consider a triangulated punctured surface $(\mathbf{\Sigma}, \Delta)$. An \textit{indexing} is a bijection $I : \mathcal{E}(\Delta) \cong \{1, \ldots, n\}$, where $n$ is the cardinal of $\mathcal{E}(\Delta)$. We write $e_i:= I^{-1}(i)$ such that $\mathcal{E}(\Delta)=\{e_1, \ldots, e_n\}$ is totally ordered with $I$ increasing. We define a total ordering $\prec_{\mathcal{I}}$ on $K_{\Delta}$ by setting $\mathbf{k} \prec_{\mathcal{I}} \mathbf{k}'$ if there exists $i\in \{1, \ldots, n\}$ such that $\mathbf{k}(e_j)=\mathbf{k}'(e_j)$ for $j<i$ and $\mathbf{k}(e_i) < \mathbf{k}'(e_i)$ (lexicographic order). We extend $\prec_{\mathcal{I}}$ to $K_{\Delta} \cup \{-\infty\}$ by imposing that $-\infty \prec_{\mathcal{I}} \mathbf{k}$ for all $\mathbf{k}$. 
The following definition is inspired from \cite{MarcheSimonAutomCharVar}, and does not agree with the more classical definition of valuation.

\begin{definition} A map $\mathbf{v} : \overline{\mathcal{S}}_{\omega}(\mathbf{\Sigma}) \rightarrow K_{\Delta}\cup \{-\infty\}$ is called a \textit{valuation} if one has $\mathbf{v}(x)=-\infty$ if and only if $x=0$,  $\mathbf{v}(xy)=\mathbf{v}(x)+\mathbf{v}(y)$ and $\mathbf{v}(x+y) \preceq_{\mathcal{I}}\mathrm{max}(\mathbf{v}(x), \mathbf{v}(y))$ for all $x,y \in  \overline{\mathcal{S}}_{\omega}(\mathbf{\Sigma})$. A valuation $\mathbf{v}$ is \textit{strict} if for $(D_1,s_1), (D_2, s_2) \in \mathcal{D}_{\mathbf{\Sigma}}$, then $(D_1,s_1)\neq (D_2,s_2)$ implies $\mathbf{v}([D_1,s_1]) \neq \mathbf{v}([D_2,s_2])$.
\end{definition}

The quantum trace naturally induces a valuation as follows. 
\begin{definition}[Valuation]
For $0\neq x \in \overline{\mathcal{S}}_{\omega}(\mathbf{\Sigma})$, expands its image through the quantum trace as $\Tr_{\omega}^{\Delta}(x) = \sum_{\mathbf{k} \in K_{\Delta}} x_{\mathbf{k}} Z^{\mathbf{k}}$ and set 
 $$\mathbf{v}(x):= \mathrm{max}\{ \mathbf{k} \mbox{, such that } x_{\mathbf{k}} \neq 0 \}, $$
 where the maximum is chosen for the total ordering $\prec_{\mathcal{I}}$. Extend $\mathbf{v}$ to a valuation $\mathbf{v} : \overline{\mathcal{S}}_{\omega}(\mathbf{\Sigma}) \rightarrow K_{\Delta}\cup \{-\infty\}$ by stating $\mathbf{v}(0)=-\infty$.
 \end{definition}
 The \textit{canonical ordering} $\leq$ of $K_{\Delta}$ is the partial ordering defined by $\mathbf{k}\leq \mathbf{k}'$ if $\mathbf{k}(e)\leq \mathbf{k}'(e)$ for all $e\in \mathcal{E}(\Delta)$. Obviously, $\mathbf{k}\leq \mathbf{k}'$ implies $\mathbf{k}\preceq_{\mathcal{I}} \mathbf{k}'$ for all indexing $\mathcal{I}$.

 The goal of this subsection is to prove the following theorem.
 
 \begin{theorem}\label{theorem_valuation}
 \begin{enumerate}
 \item For $(D,s)\in \mathcal{D}_{\mathbf{\Sigma}}$,  one has 
 $$ \Tr_{\omega}^{\Delta}([D,s])= \omega^n Z^{\mathbf{v}([D,s])}+ l.t.$$ 
 where $n\in \mathbb{Z}$ and $l.t.$ is a linear combination of elements $Z^{\mathbf{k}}$ with $\mathbf{k}< \mathbf{v}([D,s])$ for the canonical partial ordering $\leq$. In particular, $\mathbf{v}$ does not depend on $\mathcal{I}$.
 \item The valuation $\mathbf{v}$ is strict.
 \end{enumerate} 
 \end{theorem}
 
 \par Note that for $(D,s)\in \mathcal{D}_{\mathbf{\Sigma}}$, by Lemma \ref{lemma_qtr}, one has
 \begin{equation}\label{eq_valuation}
  \mathbf{v}([D,s]) = \mathrm{max}\{ \mathds{k}(D,\hat{s}), (D, \hat{s}) \in \mathrm{St}(D,s) \}.
  \end{equation}
 When $s(v)=+$ for all $v\in \partial D$ (for instance when $D$ is closed), this maximum is uniquely reached for the full state $\hat{s}$ such that $\hat{s}(v)=+$ for all $v\in D \cap \mathcal{E}(\Delta)$. In this case, the valuation $\mathbf{v}$ has a nice geometric interpretation: for $e\in \mathcal{E}(\Delta)$, the integer $\mathbf{v}([D,s]) (e)$ is the geometric intersection of $D$ with $e$. Hence when $D$ is a closed diagram, $\mathbf{v}([D])$ characterizes $D$ and for a punctured surface $\mathbf{\Sigma}=(\Sigma, \mathcal{P})$ with $\Sigma$ closed, the valuation $\mathbf{v}$ is strict. In the closed case, this valuation is the key tool that enabled Bonahon and Wong to prove in \cite{BonahonWongqTrace} that the quantum trace is injective and permitted the authors of \cite{FrohmanKaniaLe_UnicityRep} to characterize the center of the skein algebra at roots of unity (though they did not mention the quantum trace explicitly). Our goal is to generalize this key tool to the case of open punctured surfaces. 
 \vspace{2mm}
 \par First consider the case where $(\alpha, s)$ is a stated arc whose endpoints $v$ and $v'$ and isotope $\alpha$ such that it intersects $\mathcal{E}(\Delta)$ transversally and minimally. Orient $\alpha$ from $v$ to $v'$ and denote by $v=v_0, v_2, \ldots, v_n=v'$ the points of $\alpha \cap \mathcal{E}(\Delta)$ ordered by the orientation of $\alpha$. Let $\alpha_i$ be the subarc of $\alpha$ between the point $v_{i-1}$ and $v_i$ and $\mathbb{T}_i$ be the face of $\Delta$ in which lies $\alpha_i$. For $\mu, \mu' \in \{-,+\}$, we denote by $(\alpha_i)_{\mu \mu'} \in \overline{\mathcal{S}}_{\omega}(\mathbb{T}_i)$ the class of the arc $\alpha_i$ with state $\mu$ at $v_{i-1}$ and $\mu'$ at $v_i$.
\\  Suppose that $s(v)= -$. If $(\alpha_1)_{-+}$ is a bad arc in $\mathbb{T}_1$, then one has $\hat{s}(v_1)=-$ for all admissible full states $\hat{s}\in \mathrm{St}^a(\alpha, s)$. In this case, if $(\alpha_2)_{-+}$ is a bad arc in $\mathbb{T}_2$, then one has $\hat{s}(v_2)=-$ for all admissible full states $\hat{s}\in \mathrm{St}^a(\alpha, s)$ and so-on. This suggests the following:
 
 \begin{notations} We denote by $n_v(\alpha)\in \{0, \ldots, n\}$ the biggest integer such that for all $1\leq i \leq n_v(\alpha)$, the stated arc $(\alpha_i)_{-+}$ is a bad arc in $\mathbb{T}_i$. The \textit{critical part} of $\alpha$ at $v$ is the subarc $\alpha_1 \cup \ldots \cup \alpha_{n_v(\alpha)}\subset \alpha$. By convention, if $(\alpha_1)_{-+}$ is not a bad arc, we state $n_v(\alpha)=0$ and the critical part is $\{v\}$.
 \end{notations}
 
 \begin{figure}[!h] 
\centerline{\includegraphics[width=8cm]{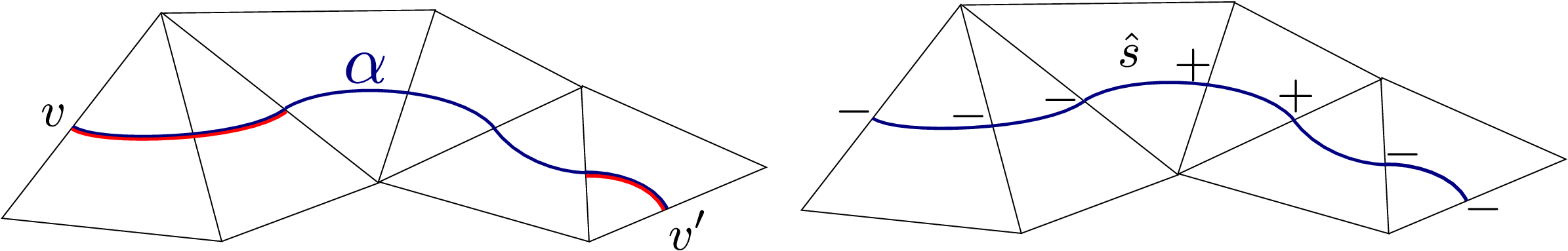} }
\caption{An arc $\alpha$ in blue, its two critical parts in red and the maximal admissible full state when both endpoints have state $-$.} 
\label{fig_critical} 
\end{figure} 
 
 Figure \ref{fig_critical} illustrates the critical parts of an arc. Note that if $s(v)=-$, then $\hat{s}(v_i)=-$ for all admissible states $\hat{s}$ and for all $v_i$ which lies in the critical part of $v$ (\textit{i.e.} for $0\leq i \leq n_v(\alpha)$). We can now describe $\mathbf{v}(\alpha, s)$ explicitly. 
 
 \begin{lemma}\label{lemma_valuation_critical} Let $(\alpha, s)$ a stated arc with endpoints $v,v'$.
 If $(s(v), s(v')) = (-, +)$, then the maximum in Equation \eqref{eq_valuation} is uniquely reached for the admissible full state $\hat{s}^{max}$ with $\hat{s}^{max}(v_i)=\left\{ \begin{array}{ll} - & \mbox{, if }v_i\mbox{ lies in the critical part of }v; \\ + & \mbox{, else.} \end{array} \right.$. If $(s(v), s(v')) = (- , -)$, the maximum is uniquely reached for the admissible full state $\hat{s}^{max}$ with $\hat{s}^{max}(v_i)=\left\{ \begin{array}{ll} - & \mbox{, if }v_i\mbox{ lies in the critical part of }v\mbox{ or of }v'; \\ + & \mbox{, else.} \end{array} \right.$. The case  $(s(v), s(v')) = (+, -)$ is obtained by exchanging the role of $v$ and $v'$. If $(s(v), s(v')) = (+, +)$, then the maximum in Equation \eqref{eq_valuation} is obtained for the admissible full state $\hat{s}^{max}$ sending every $v_i$ to $+$. 
 \end{lemma}

Note that the lemma implies that $\mathbf{v}(\alpha, s)=\mathds{k}(\alpha, \hat{s}^{max})$ does not depend on the choice of $\mathcal{I}$.

\begin{proof}
Consider the expansion $ \Tr_{\omega}^{\Delta} ([\alpha, s]) = \sum_{\hat{s} \in \mathrm{St}(\alpha, s)} P_{\hat{s}}(\omega) Z^{\mathds{k}(\alpha,\hat{s})}$ of Lemma \ref{lemma_qtr}. Suppose that $s(v)=-$ and  let $\widehat{s}\in  \mathrm{St}(\alpha, s)$ be such that $\widehat{s}(w)=+$ for some $w$ in the critical part of $\alpha$  adjacent to $v$ and let us prove that $P_{\hat{s}}(\omega)=0$.
Using the above notations,  the critical part of $\alpha$ adjacent to $v$ has the form $\alpha_1\cup \ldots \cup \alpha_{n_v(\alpha)} \subset \alpha$ with $v\in \partial \alpha_1$ and $w= \partial \alpha_k \cap \partial \alpha_{k+1}$ for some $1 \leq k < n_v(\alpha)$. Let $\mathbb{T}$ be the face of $\Delta$ containing $\alpha_{k}$ and $p$ be the puncture of $\mathbb{T}$ encircled by $\alpha_k$. Without loss of generality, we can suppose that the orientation $\mathfrak{o}_{\Delta}$ has been chosen such that the edges adjacent to $p$ are both oriented either towards $p$ or from $p$.
 Let $(\restriction{\alpha}{\mathbb{T}}, \restriction{\widehat{s}}{\mathbb{T}})$ be the stated diagram of $\mathbb{T}$ obtained by restriction. Then the stated subarc of $(\restriction{\alpha}{\mathbb{T}}, \restriction{\widehat{s}}{\mathbb{T}})$ whose underlying arc is $\alpha_k$, is a bad arc. Moreover there is no arc parallel to $\alpha_k$ between $\alpha_k$ and $p$. So $[\restriction{\alpha}{\mathbb{T}}, \restriction{\widehat{s}}{\mathbb{T}}]= 0 $ in $ \overline{\mathcal{S}}_{\omega}(\mathbb{T})$. This implies that $P_{\hat{s}}(\omega)=0$. Therefore, if $\widehat{s}\in  \mathrm{St}(\alpha, s)$ is such that $P_{\hat{s}}(\omega)\neq 0$, then $\widehat{s}(w)=-$ for all $w$ in the critical parts of $\alpha$ adjacent to an endpoint with state $-$ so $\widehat{s} \leq \widehat{s}^{max}$. This concludes the proof.

\end{proof}

\begin{lemma}\label{lemma_max} For $(D,s)\in \mathcal{D}_{\mathbf{\Sigma}}$,  one has 
 $$ \Tr_{\omega}^{\Delta}([D,s])= \omega^n Z^{\mathbf{v}([D,s])}+ l.t.$$ 
 where $n\in \mathbb{Z}$ and $l.t.$ is a linear combination of elements $Z^{\mathbf{k}}$ with $\mathbf{k}< \mathbf{v}([D,s])$ for the canonical partial ordering $\leq$. In particular, $\mathbf{v}$ does not depend on $\mathcal{I}$.
\end{lemma}

\begin{proof} When $D$ is connected, the results follows from Lemmas \ref{lemma_valuation_critical} and \ref{lemma_qtr}. The general case follows by induction on the number of connected components of $D$ using Lemma \ref{lemma_product}.
\end{proof}

We now prove that the valuation $\mathbf{v}$ is strict.

\begin{proposition}\label{prop_strict} The map $\mathbf{v} : \mathcal{D}_{\mathbf{\Sigma}} \rightarrow K_{\Delta}$ sending $(D,s)$ to $\mathbf{v}(D,s):= \mathbf{v}([D,s])$, is injective.
\end{proposition}

\begin{lemma}\label{lemma_strict_triangle}
When $\mathbf{\Sigma}=\mathbb{T}$, the map $\mathbf{v}_{\mathbb{T}} : \mathcal{D}_{\mathbb{T}} \rightarrow K_{\mathbb{T}}$ is a bijection.
\end{lemma}

\begin{proof}
Note that for $(D,s)\in \mathcal{D}_{\mathbb{T}}$, there exists $n\in \mathbb{Z}$ such that $\Tr_{\omega}^{\mathbb{T}} ([D,s]) = \omega^n Z^{\mathbf{v}_{\mathbb{T}}(D,s)}$. Since the quantum trace is an isomorphism, its sends the basis $\overline{B}= \{ [D,s], (D,s) \in \mathcal{D}_{\mathbb{T}} \}$ to a basis of $\mathcal{Z}_{\omega}(\mathbb{T})$, hence $\mathbf{v}_{\mathbb{T}}$ is a bijection.
\end{proof}

Let $a$ and $b$ be two distinct boundary arcs of $\mathbf{\Sigma}$ and denote by $\mathbf{\Sigma}_{|a\#b}$ the punctured surface obtained by gluing $a$ and $b$ together. Let $\pi : \Sigma \rightarrow \Sigma_{| a\#b}$ be the gluing map and $c=\pi(a)=\pi(b)$. Fix $\Delta$ a triangulation of $\mathbf{\Sigma}$ and write $\Delta_{a\#b}$ for the triangulation of 
$\mathbf{\Sigma}_{|a\#b}$ obtained from $\Delta$ by identifying $a$ and $b$ along $\pi$. Let $(D,s)\in \mathcal{D}_{{\mathbf{\Sigma}}_{|a\#b}}$, with $D$ isotoped such that it intersects $c$ transversally and minimally and denote by $D'$ the diagram of $\mathbf{\Sigma}$ obtained by cutting $D$ along $c$ (hence $\pi(D')=D$). For $v\in c\cap D$, denote by $v_a \in a\cap D'$ and $v_b \in b\cap D'$ the points such that $\pi(v_a)=\pi(v_b)=v$. Let $s'$ be the (not necessarily $\mathfrak{o}^+$ increasing) state of $D'$ such that: $(1)$ $s'(v)=s(\pi(v))$ for $v\notin a\cup b$, $(2)$ $s'(v_a)=s'(v_b)$ for $v\in c\cap D$, $(3)$ $(D',s')$ has no bad arc, $(4)$ $s'$ is maximal (for the partial ordering $\leq$) among the states satisfying $(1),(2)$ and $(3)$. The existence of such a maximum follows from the same argument previously used in Lemma \ref{lemma_max}. See Figure \ref{fig_diagram} for an example.

\begin{lemma}\label{lemma_technical}
\begin{enumerate}
\item There exists a unique $(D'',s'')\in \mathcal{D}_{\Delta}$ such that $[D'',s'']=\omega^n [D',s']$ for some $n\in \mathbb{Z}$.
\item The map $\theta_{a\#b} : \mathcal{D}_{{\mathbf{\Sigma}}_{|a\#b}} \rightarrow \mathcal{D}_{\mathbf{\Sigma}}$ sending $(D,s)$ to $(D'', s'')$ is injective.
\item If $a,b,c,d$ are four distinct boundary arcs, one has $\theta_{a\#b}\circ \theta_{c\# d} = \theta_{c\#d}\circ \theta_{a\#b}$.
\end{enumerate}
\end{lemma}

 \begin{figure}[!h] 
\centerline{\includegraphics[width=10cm]{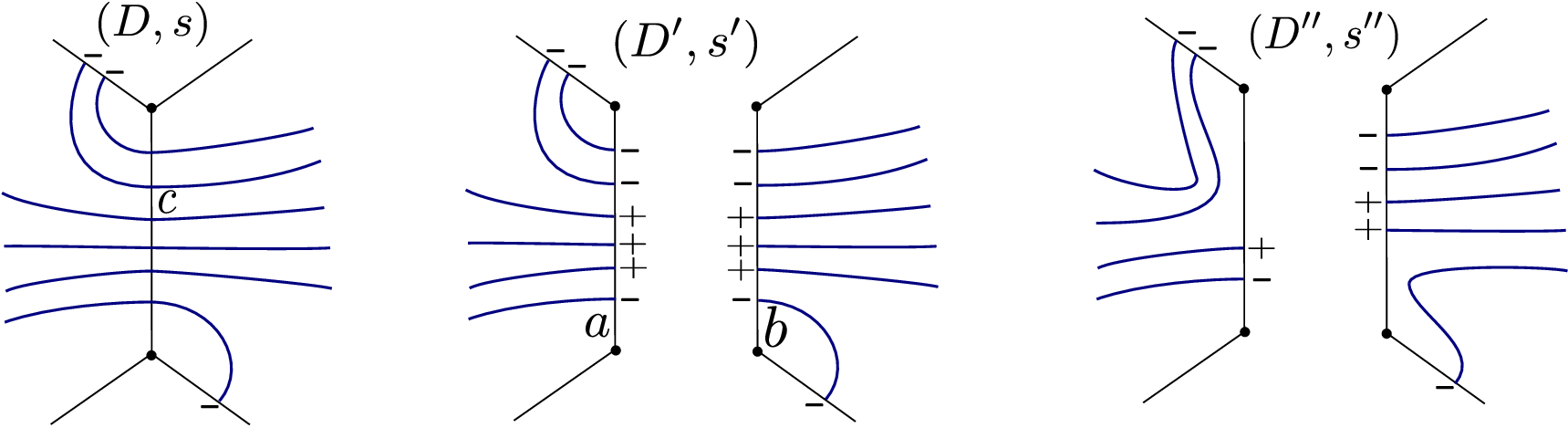} }
\caption{Three stated diagrams $(D,s)$, $(D',s')$ and $(D'',s'')= \theta_{a\#b}(D,s)$.} 
\label{fig_diagram} 
\end{figure}

\begin{proof}
$(1)$ The unicity part of the first statement follows from the fact that $\overline{B}=\{[D,s], (D,s) \in \mathcal{D}_{\mathbf{\Sigma}}\}$ is a basis. To prove the existence part, let us consider $(D,s)\in \mathcal{D}_{\mathbf{\Sigma}_{|a\#b}}$. If $(D',s')$ is $\mathfrak{o}^+$-increasing along $a$ and $b$, then one can choose $(D'',s''):= (D',s')$. Suppose that $(D',s')$ is not $\mathfrak{o}^+$-increasing along $a$. Then there exist $v_1^a, v_2^a \in a \cap \partial D'$ such that $v_1^a <_{\mathfrak{o}^+} v_2^a$ and $s'(v_1^a)=+, s'(v_2^a)=-$. We choose such a pair $(v_1^a,v_2^a)$ such that there is no element of $a\cap \partial D'$ between $v_1^a$ and $v_2^a$. Let $\alpha \subset D'$ be the connected component of $D'$ containing $v_2^a$, let $v'$ be the second endpoint of the arc $\alpha$ and let $\widetilde{\alpha} \subset D$ be the connected component of $D$ whose image through $\pi$ contains $\alpha$. 
\vspace{2mm}
\par Let us show that $\alpha$ is a corner arc and that $s'(v')=-$. Since $s'(v_2)=-$, and by maximality of $s'$, then $\pi(v_2)$ belongs to one of two critical parts of $\widetilde{\alpha}$ and the state $s$ at the corresponding endpoint takes value $-$. By contradiction, suppose that $\pi(v_2)$ belongs to the critical part of the endpoint of $\widetilde{\alpha}$ which is not $v'$, say $v''$. Then the arc $\beta \subset D'$ that connects $v_2^b$ to $v''$ is a corner arc and $s(v'')=-$. Since $D$ is simple, the arc $\beta' \subset D'$ that contains $v_1^b$ must also be corner arc and since $s$ is $\mathfrak{o}^+$-increasing, the value of $s'$ at the other endpoint of $\beta'$ is $-$. Thus the stated arc $\beta'$ is a bad arc and this contradicts the definition of $s'$. Therefore $\alpha$ is a corner arc and $s'(v)=-$.
\vspace{2mm}
\par Let $\alpha'\subset D'$ be the arc containing $v_1^a$. Consider the stated diagram $(D^{(2)}, s^{(2)})$ obtained from $(D', s')$ by gluing $\alpha$ and $\alpha'$ by identifying $v_1^a$ and $v_2^a$ and then pushing the obtained point in the interior of $\Sigma$. Using the cutting arc relation 
\begin{equation*}
\heightcurve =
\omega^{-1}
\heightexch{->}{-}{+}
- \omega^{-5}
\heightexch{->}{+}{-}
\end{equation*}
together with the fact that the $\heightexch{->}{+}{-}$ part vanishes in the reduced stated skein algebra because it contains a bad arc by the preceding discussion, one sees that $[D',s'] = \omega [D^{(2)}, s^{(2)}]$. Let us call the assignation $(D',s') \to (D^{(2)}, s^{(2)})$ a \textit{positive move} along $a$ and its inverse a \textit{negative move} along $a$. The positive and negative moves along $b$ are defined similarly. Since a positive move along $a$  decreases the number of pairs $(v_1^a, v_2^a)$ such that $v_1^a <_{\mathfrak{o}^+} v_2^a$ and $s'(v_1^a)=+, s'(v_2^a)=-$, by a finite sequence of positive moves $(D',s') \to (D^{(2)}, s^{(2)}) \to \ldots \to (D^{(n)}, s^{(n)})=: (D'',s'')$ one obtains $(D'',s'')\in \mathcal{D}_{\mathbf{\Sigma}}$ such that $[D',s']= \omega^n [D'',s'']$ and the first assertion is proved. Figure \ref{fig_diagram} shows an exemple where $(D'',s'')$ is obtained from $(D',s')$ by two positive moves along $a$ and one along $b$.

\vspace{2mm}
\par $(2)$ To prove the injectivity, consider $(D_1,s_1), (D_2,s_2) \in \mathcal{D}_{\mathbf{\Sigma}_{|a\#b}}$ such that $(D''_1, s''_1)=(D''_2,s''_2)=:(D'',s'')$. Then both $(D'_1, s'_1)$ and $(D'_2,s'_2)$ are obtained from $(D'',s'')$ by a finite sequence of negative moves. For $(D',s')$ a stated diagram obtained from $(D'',s'')$ by a finite sequence of negative moves, Let $n_1^a, n_2^a, n_3^a \in \mathbb{N}$ the non negative integers such that when running along $a$ in the $\mathfrak{o}^+$ direction, one crosses $n_1^a$ points of $\partial D'$ with states $-$ followed by $n_2^a$ points with states $+$ followed by $n_3^a$ points with states $-$. By convention, if all the points of $\partial D'$ have states $-$, one set $n_3^a=0$. Define $(n_1^b, n_2^b, n_3^b)$ in a similar manner. A negative move along $a$ increases $n_2^a$ and $n_3^a$ by $+1$ and leaves the other integers unchanged, while a negative move along $b$ increases $n_2^b$ and $n_3^b$ by $+1$ and leaves the other integers unchanged. A simple combinatorial arguments shows that among the set of stated diagrams obtained from $(D'',s'')$ by a finite set of negative moves, there is a single one such that $n_1^a=n_3^b, n_2^a=n_2^b$ and $n_3^a=n_1^b$. Since both $(D'_1,s'_1)$ and $(D'_2,s'_2)$ satisfy these equalities, one has $(D'_1,s'_1)=(D'_2,s'_2)$, thus $(D_1,s_1)=(D_2,s_2)$ and the injectivity is proved. The assertion $(3)$ is a straightforward consequence of the definitions.

\end{proof}

Consider a triangulated punctured surface $(\mathbf{\Sigma}, \Delta)$. Since $\mathbf{\Sigma}$ is obtained from $\bigsqcup_{\mathbb{T}\in F(\Delta)} \mathbb{T}$ by gluing the triangles along boundary arcs $e'$ and $e''$, by composing the applications $\theta_{e'\#e''}$ together, one obtains an injective map $\theta_{\Delta} : \mathcal{D}_{\mathbf{\Sigma}} \hookrightarrow \prod_{\mathbb{T}\in F(\Delta)} \mathcal{D}_{\mathbb{T}}$. Let $\varphi : K_{\Delta} \hookrightarrow \prod_{\mathbb{T}} K_{\mathbb{T}}$ be the injective group morphism sending $\mathbf{k}$ to $\prod_{\mathbb{T}} \mathbf{k}_{\mathbb{T}}$ where given an edge $e_{\mathbb{T}}\in \mathcal{E}(\mathbb{T})$ corresponding to an edge $e\in \mathcal{E}(\Delta)$, one set $\mathbf{k}_{\mathbb{T}}(e_{\mathbb{T}}) := \mathbf{k}(e)$. Note that $\varphi$ was implicitly used in the definition of the gluing map $i^{\Delta}$ in Section \ref{subsec_CF}.

\begin{proof}[Proof of Proposition \ref{prop_strict}]
It follows from the definitions that the following diagram commutes
$$
\begin{tikzcd}
\mathcal{D}_{\mathbf{\Sigma}} 
\arrow[r, hook, "\theta_{\Delta}"] \arrow[d, "\mathbf{v}"] &
\prod_{\mathbb{T}\in F(\Delta)} \mathcal{D}_{\mathbb{T}} 
\arrow[d, "\cong"', "\prod_{\mathbb{T}}\mathbf{v}_{\mathbb{T}}"] \\
K_{\Delta} 
\arrow[r, hook, "\varphi"] &
\prod_{\mathbb{T}\in F(\Delta)} K_{\mathbb{T}}
\end{tikzcd}
$$

Hence the injectivity of $\mathbf{v}$ follows from the injectivity of $\theta_{\Delta}, \mathbf{v}_{\mathbb{T}}$ and $\varphi$.

\end{proof}

\begin{proof}[Proof of Theorem \ref{theorem_valuation}]
The Theorem follows from Lemma \ref{lemma_max} and Proposition \ref{prop_strict}.

\end{proof}

\subsection{Further properties of the valuation}\label{sec_valuation_ppty}
In this subsection we state some technical results concerning the valuation $\mathbf{v}$.
\vspace{2mm}
\par The map  $\mathbf{v} : \mathcal{D}_{\mathbf{\Sigma}} \rightarrow K_{\Delta}$ is not surjective in general. Though, one has the following:

\begin{lemma}\label{lemma_quasi_surjective}
For all $\mathbf{k} \in K_{\Delta}$, there exists $(D_1,s_1), (D_2, s_2) \in \mathcal{D}_{\mathbf{\Sigma}}$ such that $\mathbf{k}=\mathbf{v}(D_1,s_1) - \mathbf{v}(D_2,s_2)$.
\end{lemma}

Let us state an alternative description of the abelian group $K_{\Delta}$ borrowed from \cite{BonahonWong2}. Denote by $\tau_{\Delta}\subset \Sigma$ the train track whose restriction to each face $\mathbb{T}\in F(\Delta)$ looks like Figure \ref{figtraintracks}. The edges of $\tau_{\Delta}$ in $\mathbb{T}$ can thus be identified with the corner arcs of $\mathbb{T}$ and will be called \textit{corner edges}.
Let $\mathcal{W}(\tau_{\Delta}, \mathbb{Z})$ be the abelian group of maps from the set of the edges of $\tau_{\Delta}$ to the set of integers that satisfy the \emph{switch-condition} illustrated in Figure \ref{figtraintracks}, where the group structure is given by the sum of maps. Define a group isomorphism $\Phi : K_{\Delta} \xrightarrow{\cong} \mathcal{W}(\tau_{\Delta}, \mathbb{Z})$ as follows. For $\mathbf{k}\in K_{\Delta}$ and  $\mathfrak{c}$ a corner edge of $\tau_{\Delta}$ that connects two edges $e_a$ and $e_b$ of a face of $\Delta$ with third edge $e_c$, 
   one sets 
   \begin{equation*}
   \Phi(\mathbf{k})(\mathfrak{c})=\frac{\mathbf{k}(e_a)+\mathbf{k}(e_b)-\mathbf{k}(e_c)}{2}. 
   \end{equation*}
   The balanced condition insures that this is an integer. 
   \\
  The inverse map is defined by sending $\phi\in  \mathcal{W}(\tau_{\Delta}, \mathbb{Z})$ to the balanced map $\mathbf{k}$ such that $\mathbf{k}(e)$ is obtained by choosing an arbitrary face $\mathbb{T}$ containing $e$ and  setting $\mathbf{k}(e)=\phi(\mathfrak{c})+\phi(\mathfrak{c}')$ where $\mathfrak{c},\mathfrak{c}'$ are the two edges of the train track lying in $\mathbb{T}$ and intersecting $e$. The switch condition insures that this integer does not depend on the choice of the face $\mathbb{T}$.

 \begin{figure}[!h] 
\centerline{\includegraphics[width=8cm]{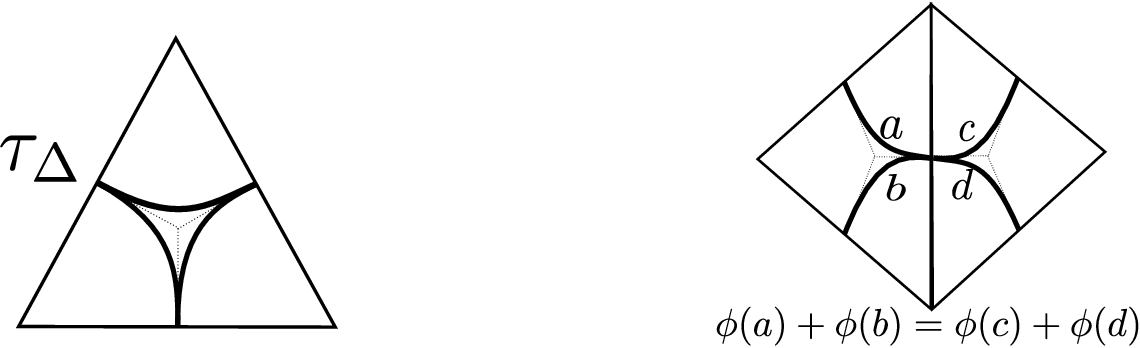} }
\caption{On the left: the train track associated to the triangle. On the right: the switch condition.} 
\label{figtraintracks} 
\end{figure} 

\par Denote by $\mathcal{W}(\tau_{\Delta}, \mathbb{N}) \subset \mathcal{W}(\tau_{\Delta}, \mathbb{Z})$  the subset of maps that take only non-negative values. Let $(D,s^+)\in \mathcal{D}_{\mathbf{\Sigma}}$ be a stated diagram such that $s^+(v)=+$ for all $v\in \partial D$. As discussed previously, $\mathbf{v}(D, s^+)$ sends an edge $e$ to the geometric intersection of $D$ with $e$ (the maximum in Equation \eqref{eq_valuation} is reached for the full state sending every points of $D\cap \mathcal{E}(\Delta)$ to $+$), hence $\Phi(\mathbf{v}(D,s^{+})) \in \mathcal{W}(\tau_{\Delta}, \mathbb{N})$. Conversely, given $\phi \in  \mathcal{W}(\tau_{\Delta}, \mathbb{N})$,  for each corner edge $\mathfrak{c}$ in a face $\mathbb{T}$, draw $\phi(\mathfrak{c})$ parallel copies of the corner $\mathfrak{c}$. The switch condition insures that we can glue the so-obtained collection of arcs to get a diagram $D$.  Let $s^+$ be the state of $D$ sending every point of $\partial D$ to $+$. Then one has $\Phi(\mathbf{v}(D, s^+))=\phi$, therefore the image of $\Phi \circ \mathbf{v}$ contains $\mathcal{W}(\tau_{\Delta}, \mathbb{N})$.

\begin{proof}[Proof of Lemma \ref{lemma_quasi_surjective}]

Let $\mathbf{k}\in K_{\Delta}$ and write $\phi:= \Phi(\mathbf{k})$. For $n\geq 0$, denote by $\phi_{[n]} \in \mathcal{W}(\tau_{\Delta}, \mathbb{N})$ the map sending any corner edge to $n$. Let $n_0\geq 0$ be such that $-n_0 \leq \phi(\mathfrak{c})$ for all corner edge $\mathfrak{c}$. Then both $\phi_{[n_0]}$ and $\phi':= \phi + \phi_{[n_0]}$ are in $\mathcal{W}(\tau_{\Delta}, \mathbb{N})$, hence by the preceding discussion, there exist $(D_1,s_1), (D_2,s_2) \in \mathcal{D}_{\mathbf{\Sigma}}$ such that $\Phi (\mathbf{v}(D_2,s_2))= \phi_{[n_0]}$ and $\Phi(\mathbf{v}(D_1,s_1))=\phi'$. Since $\phi = \phi' - \phi_{[n_0]}$, one has $\mathbf{k}=\mathbf{v}(D_1,s_1) - \mathbf{v}(D_2,s_2)$.
\end{proof}

\begin{lemma}\label{lemma_center_surjective}
One has $\mathbf{v}(\mathcal{D}_{\mathbf{\Sigma}}) + NK_{\Delta} = K_{\Delta}$. 
\\ Said differently, for every $\mathbf{k}\in K_{\Delta}$, there exists $(D,s)\in \mathcal{D}_{\mathbf{\Sigma}}$ and $\mathbf{k}_0 \in K_{\Delta}$ such that $\mathbf{k} = \mathbf{v}(D,s) + N\mathbf{k}_0$.
\end{lemma}

\begin{proof}
Since $\mathcal{W}(\tau_{\Delta}, \mathbb{N}) \subset \Phi(\mathbf{v}(\mathcal{D}_{\mathbf{\Sigma}}))$, it is sufficient to prove that $\mathcal{W}(\tau_{\Delta}, \mathbb{N}) + N\mathcal{W}(\tau_{\Delta}, \mathbb{Z}) = \mathcal{W}(\tau_{\Delta}, \mathbb{Z})$. Let $\phi \in  \mathcal{W}(\tau_{\Delta}, \mathbb{Z})$. Choose $n_0 \geq 0$ such that $-n_0 N \leq \phi(\mathfrak{c})$ for all corner edge $\mathfrak{c}$. Then $\phi':= \phi +N \phi_{[n_0]}\in \mathcal{W}(\tau_{\Delta}, \mathbb{N})$ and $\phi = \phi' - N \phi_{[n_0]} \in \mathcal{W}(\tau_{\Delta}, \mathbb{N}) + N\mathcal{W}(\tau_{\Delta}, \mathbb{Z}).$

\end{proof}

\begin{lemma}\label{lemma_corner_inner}
Let $p\in \mathcal{P}\cap \mathring{\Sigma}$ be an inner puncture and $(D,s)\subset \mathcal{D}_{\mathbf{\Sigma}}$ be such that $D$ does not contain any copy of the peripheral curve $\gamma_p$. Write $\phi := \Phi(\mathbf{v}(D,s)) \in \mathcal{W}(\tau_{\Delta}, \mathbb{Z})$. Then there exists a corner edge $\mathfrak{c}$ adjacent to $p$ such that $\phi(\mathfrak{c})=0$.
\end{lemma}

\begin{proof}
Let us fix some notations for the proof. Let $(D, \hat{s})\in \widehat{\mathcal{D}}_{\Delta}$ be the full stated diagram such that $\mathds{k}(D, \hat{s}) = \mathbf{v}(D,s)$. For a face $\mathbb{T}\in F(\Delta)$, we denote by $(D_{\mathbb{T}}, \hat{s}_{\mathbb{T}})\in \mathcal{D}_{\mathbb{T}}$ the stated diagram obtained by restricting $(D, \hat{s})$ to $\mathbb{T}$. For $\mathfrak{c}$ a corner edge adjacent to $p$ lying in some face $\mathbb{T}$, we denote by $p'$ and $p''$ the two other vertices of $\mathbb{T}$ such that one sees $p, p', p''$ once travelling around $\partial \mathbb{T}$ in the trigonometric direction. Write $\mathfrak{c}'$ and $\mathfrak{c}''$ the corner edges in $\mathbb{T}$ adjacent to $p'$ and $p''$ respectively (see Figure \ref{fig_corner}). Note that the points $p,p', p''$ are not necessary pairwise distinct, whereas the corner edges $\mathfrak{c}, \mathfrak{c}', \mathfrak{c}''$ are.  For $\varepsilon, \varepsilon' \in \{-, +\}$, write $\mathfrak{c}_{\varepsilon \varepsilon'}$ the associated stated arc with the convention that $\mathfrak{c}_{-+}$ is a bad arc (hence $\mathfrak{c}_{+-}$ is not a bad arc). We use the same convention for $\mathfrak{c}', \mathfrak{c}''$. The stated diagram $(D_{\mathbb{T}}, \hat{s}_{\mathbb{T}})$ is a disjoint union of stated arcs of the form $\mathfrak{c}_{\varepsilon \varepsilon'}, \mathfrak{c}'_{\varepsilon \varepsilon'}, \mathfrak{c}''_{\varepsilon \varepsilon'}$ which are not bad arcs. By definition of $\Phi$, one has the equality:
\begin{equation}\label{eq_phi}
\phi(\mathfrak{c}) = \# \mathfrak{c}_{++} - \# \mathfrak{c}_{--} + \# \mathfrak{c}''_{+-} - \# \mathfrak{c}'_{+-},
\end{equation} 
where, for instance, $ \# \mathfrak{c}_{++}$ denotes the number of parallel copies of $\mathfrak{c}_{++}$ in $(D_{\mathbb{T}}, \hat{s}_{\mathbb{T}})$. We denote by $e$ the boundary arc of $\mathbb{T}$ adjacent to both $\mathfrak{c}$ and $\mathfrak{c}'$ (between $p$ and $p'$) and $e'$ the boundary arc adjacent to $\mathfrak{c}'$ and $\mathfrak{c}''$.

 \begin{figure}[!h] 
\centerline{\includegraphics[width=10cm]{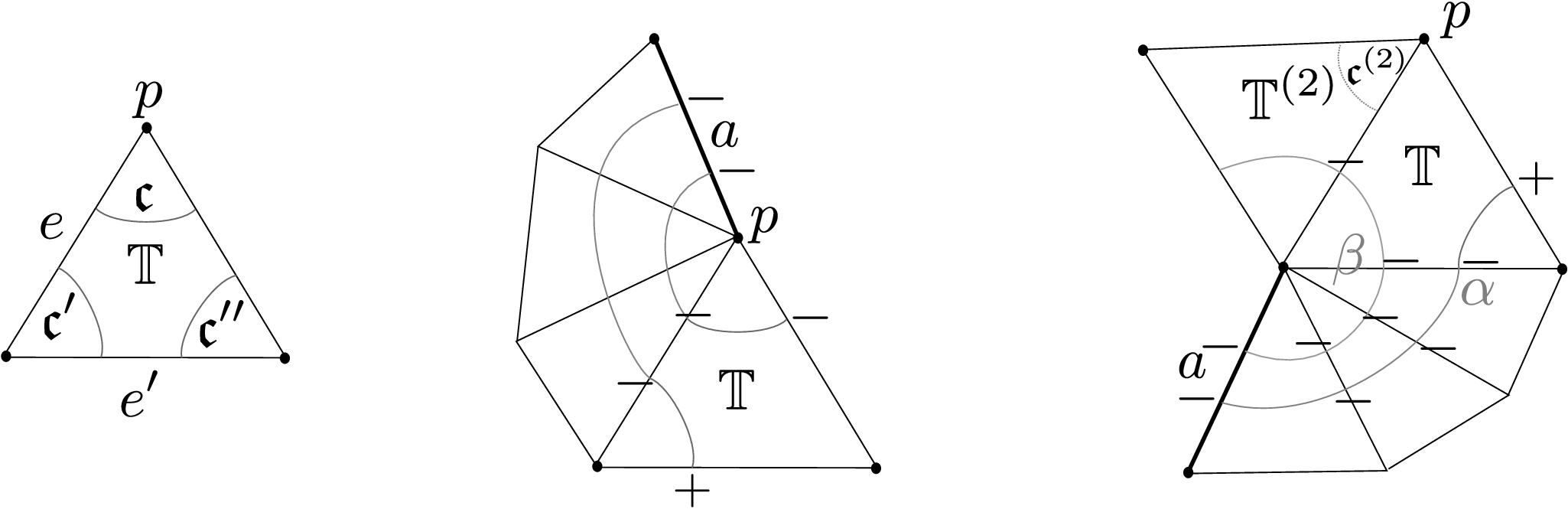} }
\caption{On the left: a face adjacent to $p$ and its corner edges. On the middle: illustration of Fact $1$. On the right: illustration of Fact $2$.} 
\label{fig_corner} 
\end{figure} 

\vspace{2mm}
\par We first prove two elementary facts.
\vspace{2mm}
\par \textit{Fact 1: the stated diagram }  $(D_{\mathbb{T}}, \hat{s}_{\mathbb{T}})$\textit{ does not contain any copy of }$\mathfrak{c}_{--}$\textit{ nor }$\mathfrak{c}'_{+-}$.
\\ By contradiction, suppose that $(D_{\mathbb{T}}, \hat{s}_{\mathbb{T}})$ contains some copy of either $\mathfrak{c}_{--}$ or $\mathfrak{c}'_{+-}$ and let $\alpha \subset D$ be the arc whose restriction to $\mathbb{T}$ contains this copy. Then the point $v$ in $\alpha \cap e$ with $\hat{s}(v)=-$ belongs to a critical part of $\alpha$  hence, once leaving $\mathbb{T}$ at $v$, the arc $\alpha$ spirals around $p$ until it reaches a boundary arc $a$ at which $s$ takes the value $-$ (see Figure \ref{fig_corner}). Hence $p\in a$ is not an inner puncture and we found our contradiction.

\vspace{2mm}
\par \textit{Fact 2: Suppose that }$(D_{\mathbb{T}}, \hat{s}_{\mathbb{T}})$\textit{ contains at least one copy of }$\mathfrak{c}''_{+-}$\textit{ and does not contain any arc of the form }$\mathfrak{c}_{\varepsilon \varepsilon'}$. \textit{ Then for all } $v\in e\cap D$\textit{, one has }$\hat{s}(v)=-$.
\\ Let $\alpha \subset D$ be an arc whose restriction to $\mathbb{T}$ contains a copy of $\mathfrak{c}''_{+-}$ and let $v$ be the endpoint of $\alpha \cap e'$ on which $\hat{s}$ takes value $-$. Then the point $v$ belongs to a critical part of $\alpha$ hence, once leaving $\mathbb{T}$ at $v$, the arc $\alpha$ spirals around $p'$ until it reaches a boundary arc $a$ at which $s$ takes value $-$ (see Figure \ref{fig_corner}).  Suppose that $(D_{\mathbb{T}}, \hat{s}_{\mathbb{T}})$ contains an arc of the form $\mathfrak{c}'_{\varepsilon \varepsilon'}$ and let us prove that $\varepsilon=\varepsilon'=-$. This will imply Fact $2$. Let $\beta\subset D$ be an arc whose restriction to $\mathbb{T}$ contains a copy of $\mathfrak{c}'$ and let $w$ the intersection of this copy with $e'$. Since $D$ is a simple diagram, $\beta$ can not intersect $\alpha$, hence while leaving $\mathbb{T}$ at $w$, it spirals around $p'$ until it reaches $a$ at some point $w'$, thus the copy of $\mathfrak{c}'$ in $\beta$ belongs to its critical part in $w'$. Because $s$ is $\mathfrak{o}^+$ increasing, one has $s(w')=-$, thus $\varepsilon=\varepsilon'=-$.

\vspace{2mm}
\par We now proceed to the proof. Because $D$ does not contain any copy of the peripheral curve $\gamma_p$, there exists a corner edge $\mathfrak{c}$ adjacent to $p$ in some face $\mathbb{T}$, such that with the above notations the stated diagram $(D_{\mathbb{T}}, \hat{s}_{\mathbb{T}})$ does not contain any of the form $\mathfrak{c}_{\varepsilon \varepsilon'}$. If $\phi(\mathfrak{c})=0$ the proof is finished. Else, by Equation \eqref{eq_phi} and Fact $1$, the stated diagram $(D_{\mathbb{T}}, \hat{s}_{\mathbb{T}})$ must contain a copy of $\mathfrak{c}''_{+-}$. By Fact $2$,  for any $v\in e\cap D$, one has $\hat{s}(v)=-$. Let $\mathbb{T}^{(2)}$ be the other face which shares the boundary edge $e$ with $\mathbb{T}$ and let $\mathfrak{c}^{(2)}$ be the corner edge adjacent to $p$ next to $\mathfrak{c}$ in the counterclockwise direction (see Figure \ref{fig_corner}). Because $D\cap e$ contains no point at which $\hat{s}$ has value $+$, the stated diagram $(D_{\mathbb{T}^{(2)}}, \hat{s}_{\mathbb{T}^{(2)}})$ does not contain any copy of $\mathfrak{c}^{(2)}_{++}$ nor $\mathfrak{c}''^{(2)}_{+-}$. Using Fact $1$, Equation \eqref{eq_phi} implies that $\phi(\mathfrak{c}^{(2)})=0$. This concludes the proof.

\end{proof}

\section{The center of the reduced stated skein algebra}

\subsection{Characterisation of the center}

\begin{definition}\label{def_center_skein} Let $\mathbf{\Sigma}$ be a punctured surface.
\begin{enumerate}
\item For $p\in \mathcal{P}\cap \mathring{\Sigma}$ an inner puncture, we call \textit{peripheral curve} a simple closed curve bounding a properly embedded disc in $\Sigma$ whose only intersection with $\mathcal{P}$ is $\{p\}$. We write $\gamma_p \in \overline{\mathcal{S}}_{\omega}(\mathbf{\Sigma})$ the class of a peripheral curve in $p$.
\item For $\partial$ a connected component of $\partial \Sigma$, denote by $D_{\partial}$ the simple diagram made of the disjoint union of every corner arcs $\alpha_p$ for $p\in \mathcal{P} \cap \partial$ and write $s^+$ (resp. $s^-$) the state on $D_{\partial}$ sending every point of $\partial D_{\partial}$ to $+$ (resp. to $-$). The \textit{boundary elements} in $\partial$ are the classes $\alpha_{\partial} := [D_{\partial}, s^+]$ and $\alpha_{\partial}^{-1} := [D_{\partial}, s^-]$ where we use the $\mathfrak{o}^+$ orientation.
\end{enumerate}
\end{definition}

\begin{lemma}\label{lemma_central} The elements $\gamma_p, \alpha_{\partial}$ and $\alpha_{\partial}^{-1}$ are central in $\overline{\mathcal{S}}_{\omega}(\mathbf{\Sigma})$ and $\alpha_{\partial}\alpha_{\partial}^{-1}=1$.
\end{lemma}

\begin{proof}
This is an immediate consequence of the facts that, given a topological triangulation $\Delta$, the quantum trace is injective and  one has  $\Tr_{\omega}^{\Delta}(\gamma_p)=H_p+H_p^{-1}$, $\Tr_{\omega}^{\Delta}(\alpha_{\partial}) = H_{\partial}$ and $\Tr_{\omega}^{\Delta}(\alpha_{\partial}^{-1}) = H_{\partial}^{-1}$, where $H_p, H_{\partial}$ are the central elements of Definition \ref{def_central_elements}.
\end{proof}

The purpose of this subsection is to prove the 
\begin{theorem}\label{theorem_center}
For  $\mathbf{\Sigma} = (\Sigma, \mathcal{P})$ an open punctured surface, the center of the reduced stated skein algebra $\overline{\mathcal{S}}_{\omega}(\mathbf{\Sigma})$ is generated by the image of the Chebyshev morphism together with the peripheral curves and boundary elements.
\end{theorem}
When $\mathbf{\Sigma}$ is closed, the analogue of Theorem \ref{theorem_center} was proved in \cite[Theorem $4.1$]{FrohmanKaniaLe_UnicityRep}. During all the section, we fix a connected punctured surface $\mathbf{\Sigma}=(\Sigma, \mathcal{P})$ with $\partial \Sigma \neq \emptyset$, together with a topological triangulation $\Delta$ and an indexing map $\mathcal{I}$.	
	
\begin{notations}
We denote by $K^0_{\Delta} \subset K_{\Delta}$ the kernel of the pairing 
$$ \left(\cdot, \cdot\right)^{WP}_{N} : K_{\Delta} \times K_{\Delta} \xrightarrow{\left(\cdot, \cdot\right)^{WP}} \mathbb{Z} \rightarrow \mathbb{Z}/N\mathbb{Z}, $$
where $\mathbb{Z} \rightarrow \mathbb{Z}/N\mathbb{Z}$ is the quotient mapping and $\left(\cdot, \cdot \right)^{WP}$ is the Weil-Peterssen pairing defined in Section $2.2$.
\end{notations}	
	
Recall from the discussion at the beginning of Section \ref{subsec_CF} that the center of $\mathcal{Z}_{\omega}(\mathbf{\Sigma}, \Delta)$ is spanned by the elements $Z^{\mathbf{k}}$ with $\mathbf{k} \in K^0_{\Delta}$.

\begin{lemma}\label{lemma_center1}
If $x\neq 0$ is in the center of $\overline{\mathcal{S}}_{\omega}(\mathbf{\Sigma})$, then $\mathbf{v}(x)\in K^0_{\Delta}$.
\end{lemma}

\begin{proof}
Let $x$ be a non-zero central element of $\overline{\mathcal{S}}_{\omega}(\mathbf{\Sigma})$. The image of $x$ through the quantum trace writes
$$ \Tr_{\omega}^{\Delta}(x) = c_x Z^{\mathbf{v}(x)} + \mbox{lower terms}, $$
where $c_x\in \mathbb{C}^*$ and "lower terms" is a linear combination of balanced monomial $Z^{\mathbf{k}}$ with $\mathbf{k}\prec_{\mathcal{I}} \mathbf{v}(x)$. For $(D,s)\in \mathcal{D}_{\mathbf{\Sigma}}$, one has similarly
$$ \Tr_{\omega}^{\Delta}([D,s]) = c_{[D,s]} Z^{\mathbf{v}(D,s)} + \mbox{lower terms}, $$
therefore one has:
$$  \Tr_{\omega}^{\Delta}(x[D,s]) = c_x c_{[D,s]} Z^{\mathbf{v}(x)}Z^{\mathbf{v}(D,s)} + \mbox{lower terms}, \quad \Tr_{\omega}^{\Delta}([D,s]x) = c_x c_{[D,s]} Z^{\mathbf{v}(D,s)}Z^{\mathbf{v}(x)}  + \mbox{lower terms}.$$
Since $x$ is central, one finds the equalities
 $$Z^{\mathbf{v}(x)}Z^{\mathbf{v}(D,s)}= Z^{\mathbf{v}(D,s)}Z^{\mathbf{v}(x)} \Leftrightarrow \omega^{2(\mathbf{v}(x), \mathbf{v}(D,s))^{WP}}=1 \Leftrightarrow (\mathbf{v}(x), \mathbf{v}(D,s))^{WP}_N=0.$$
Since this equality holds for any $(D,s)\in \mathcal{D}_{\mathbf{\Sigma}}$ and since the elements $\mathbf{v}(D,s)$ generate the group $K_{\Delta}$ by Lemma \ref{lemma_quasi_surjective}, one has $\mathbf{v}(x)\in K^0_{\Delta}$.

\end{proof}
	
\begin{definition}\label{def_multiple}
\begin{enumerate}
\item Let $x= \sum_{(D,s)\in \mathcal{D}_{\mathbf{\Sigma}}} x_{[D,s]} [D,s]$ be a non-zero element of $\overline{\mathcal{S}}_{\omega}(\mathbf{\Sigma})$. Let $(D_0,s_0)$ be unique element of $\mathcal{D}_{\mathbf{\Sigma}}$ such that $\mathbf{v}(x)=\mathbf{v}(D_0,s_0)$ (the unicity is guaranteed by Proposition \ref{prop_strict}). The \textit{leading term} of $x$ is the element 
$\mathrm{ld}(x):= x_{(D_0,s_0)} [D_0,s_0]$. 
\item Write $\Tr_{\omega}^{\Delta}(x)=\sum_{\mathbf{k}} x_{\mathbf{k}}Z^{\mathbf{k}}$. For $e\in \mathcal{E}(\Delta)$, the \textit{minimal height} of $x$ in $e$ is the integer
$$ m_e(x)= \mathrm{min} \{ \mathbf{k}(e)| \quad x_{\mathbf{k}}\neq 0\}.$$
\item We denote by $\mathcal{Z}^0\subset \overline{\mathcal{S}}_{\omega}(\mathbf{\Sigma})$ the subalgebra generated by the image of the Chebyshev morphism together with the peripheral curves $\gamma_p$  and the  boundary central elements $\alpha_{\partial}$.
\item For $(D,s)\in \mathcal{D}_{\mathbf{\Sigma}}$, we denote by $[D^{(N)}, s^{(N)}] \in \overline{\mathcal{S}}_{\omega}(\mathbf{\Sigma})$ the linear combination of classes of stated diagrams obtained by replacing each stated arc $\alpha_{\varepsilon \varepsilon'}$ of $(D,s)$ by $N$ parallel copies of $\alpha_{\varepsilon \varepsilon'}$ and replacing each closed curve $\gamma$ by $T_N(\gamma)$.
\item A \textit{basic element} of $\mathcal{Z}^0$ is an element of the form $z=c[D^{(N)}, s^{(N)}] \prod_{p\in \mathcal{P}\cap \mathring{\Sigma}} \gamma_p^{n_p} \prod_{\partial \in \pi_0(\partial \Sigma)} \alpha_{\partial}^{n_{\partial}}$, where $c\in\mathbb{C}^*, n_p \in \mathbb{N}, n_{\partial} \in \mathbb{Z}$ and $(D,s)\in \mathcal{D}_{\mathbf{\Sigma}}$.
\end{enumerate}
\end{definition}

By Lemma \ref{lemma_central} and Theorem \ref{theorem_center_skein},  $\mathcal{Z}^0$ is included in the center of $\overline{\mathcal{S}}_{\omega}(\mathbf{\Sigma})$ thus to prove Theorem \ref{theorem_center}, we need to show the reverse inclusion. Note that if follows from Lemma \ref{lemma_product} that basic elements lies in $ \mathcal{Z}^0$ and  that $\mathbf{v}(D^{(N)}, s^{(N)}) = N \mathbf{v}(D,s)$. Note also that $m_e(xy)=m_e(x)+m_e(y)$ and $m_e(x+y)\geq \mathrm{min}(m_e(x)+m_e(y))$ (once setting $m_e(0)=\infty$), hence $m_e$ is a discrete valuation.
	
\begin{lemma}\label{lemma_basic}
For any $e\in\mathcal{E}(\Delta)$ and  $z\in \mathcal{Z}^0$ a basic element, one has $m_e(z)=m_e(\mathrm{ld}(z))$. 
\end{lemma}

\begin{proof}
Let us start by a preliminary remark. Since $\Tr_{\omega}^{\Delta}(\alpha_{\partial})= Z^{\mathbf{k}_{\partial}}$, for any $e\in \mathcal{E}(\Delta)$ and $x\in \overline{\mathcal{S}}_{\omega}(\mathbf{\Sigma})$ one has $m_e(x\alpha_{\partial}^n) = \left\{ \begin{array}{ll} m_e(x) + n &\mbox{, if }e\subset \partial; \\ m_e(x) &\mbox{, else.} \end{array} \right.$. Hence if $m_e(x)=m_e(\mathrm{ld}(x))$ then $m_e(x\alpha_{\partial})= m_e(\mathrm{ld}(x\alpha_{\partial}))$. Therefore it is sufficient to prove the lemma for a basic element of the form  $z=[(D^{(N)}, s^{(N)})] \prod_{p\in \mathcal{P}\cap \mathring{\Sigma}} \gamma_p^{n_p}$. Fix such an element and denote by $\delta_1, \ldots, \delta_t$ the closed connected components of $D$. For $\mathbf{n}= (n_1,\ldots, n_t) \in \{0, \ldots, N\}^t$, denote by $D(\mathbf{n})\in \mathcal{D}_{\mathbf{\Sigma}}$ the stated diagram obtained from $(D^{(N)}, s^{(N)})$ by replacing the composent $T_N(\delta_i)$ by $n_i$ parallel copies of $\delta_i$ and adding $p$ parallel copies of the peripheral curve $\gamma_p$ for every inner puncture $p$. Write $T_N(X)= \sum_{i=0}^{N} c_i X^i$, where $c_N=1$. By definition, the basic element $z$ decomposes in the basis $\overline{B}$ as: 
$$ z = \sum_{\mathbf{n}\in \{0, \ldots, N\}^t} c_{n_1}\ldots c_{n_t} D(\mathbf{n}).$$
We claim that $\mathrm{ld}(z)=D(N, \ldots, N)$. Indeed, let $s_i^+$ be the full state of $\delta_i$ sending every point of $\delta_i \cap \mathcal{E}(\Delta)$ to $+$ and write $\mathbf{k}_i:=\mathds{k}(\delta_i, s_i^+) \geq \mathbf{0}$. For $\mathbf{n}\in \{0, \ldots, N\}^t$, since $D(N, \ldots, N) = D(\mathbf{n}) \cup \cup_{i=1}^t \delta_i^{N-n_i}$, one has 
$$ \mathbf{v}(D(N,\ldots, N)) = \mathbf{v}(D(\mathbf{n})) + \sum_{i=1}^t (N-n_i) \mathbf{k}_i.$$
Therefore $\mathbf{v}(D(N,\ldots, N)) \geq \mathbf{v}(D(\mathbf{n}))$ and  $\mathrm{ld}(z)=D(N, \ldots, N)$. Next remark, using Equation \eqref{eq_valuation}, that $m_e(\delta_i)=-i(\delta_i, e)\leq 0$ is the opposite of the geometric intersection of $\delta_i$ with $e$, hence
$$ m_e(D(N, \ldots, N)) = m_e(D(\mathbf{n})) - \sum_{e, i} (N-n_i)i(\delta_i, e) \leq m_e(D(\mathbf{n})), $$
hence $m_e(z)=m_e(D(N, \ldots, N))=m_e(\mathrm{ld}(z))$.
\end{proof}	
	
\begin{lemma}\label{lemma_center2}
Let $(D,s)\in \mathcal{D}_{\mathbf{\Sigma}}$ be such that $\mathbf{v}(D,s)= N \mathbf{k}$ for some $\mathbf{k}\in K_{\Delta}$. Then there exists $(D_0, s_0)\in \mathcal{D}_{\mathbf{\Sigma}}$ such that $\mathbf{v}(D_0,s_0)=\mathbf{k}$ and thus $\mathbf{v}(D,s)=\mathbf{v}(D_0^{(N)}, s_0^{(N)})$.
\end{lemma}	

\begin{proof}

First note that since $\mathbf{v}_{\mathbb{T}} : \mathcal{D}_{\mathbb{T}} \rightarrow K_{\mathbb{T}}$ is a bijection (Lemma \ref{lemma_strict_triangle}), the results holds when $\mathbf{\Sigma}=\mathbb{T}$.
\vspace{2mm}
\par We next show a preliminary result: suppose that $\mathbf{\Sigma}_{|a\#b}$ is obtained from $\mathbf{\Sigma}$ by gluing two boundary arcs $a$ and $b$ and that $\theta_{a\#b}(D,s)= (D_1^{(N)}, s_1^{(N)})$ for some $(D_1,s_1)\in \mathcal{D}_{\mathbf{\Sigma}}$; let us show that there exists $(D_0,s_0)\in \mathcal{D}_{\mathbf{\Sigma}_{|a\#b}}$ such that $(D,s)= (D_0^{(N)}, s_0^{(N)})$. Using the notations of Lemma \ref{lemma_technical}, the stated diagram $(D',s')$ is obtained from $(D_1^{(N)}, s_1^{(N)})$ by a finite sequence of, say $m_a$ and $m_b$,  negative moves along $a$ and $b$. By a counting argument, similar to the proof of the injectivity of $\theta_{a\#b}$, one sees that $m_a=N m_a^0$ and $m_b=N m_b^0$ for some $m_a^0, m_b^0 \geq 0$. By performing on $(D_1, s_1)$ a sequence of $m_a^0$ negative moves  along $a$ and $m_b^0$ negative moves along $b$, one obtains a stated diagram $(D'_0,s'_0)$ such that $({D'}_0^{(N)}, {s'}_0^{(N)})= (D',s')$. Gluing $a$ and $b$ together one obtains a stated diagram $(D_0,s_0)$ such that  $(D,s)= (D_0^{(N)}, s_0^{(N)})$ as claimed.
\vspace{2mm}
\par We can now finish the proof. Recall the commutative diagram in the category of abelian groups:

$$
\begin{tikzcd}
\mathcal{D}_{\mathbf{\Sigma}} 
\arrow[r, hook, "\theta_{\Delta}"] \arrow[d, "\mathbf{v}"] &
\prod_{\mathbb{T}\in F(\Delta)} \mathcal{D}_{\mathbb{T}} 
\arrow[d, "\cong"', "\prod_{\mathbb{T}}\mathbf{v}_{\mathbb{T}}"] \\
K_{\Delta} 
\arrow[r, hook, "\varphi"] &
\prod_{\mathbb{T}\in F(\Delta)} K_{\mathbb{T}}
\end{tikzcd}
$$
Consider $(D,s)\in \mathcal{D}_{\mathbf{\Sigma}}$ such that $\mathbf{v}(D,s)= N \mathbf{k}$ for some $\mathbf{k}\in K_{\Delta}$. Then, using the fact that the lemma holds for the triangle,  one has $\theta_{\Delta}(D,s) = \prod_{\mathbb{T}\in F(\Delta)} (D_{\mathbb{T}}^{(N)}, s_{\mathbb{T}}^{(N)})$ for some  $ (D_{\mathbb{T}}^{(N)}, s_{\mathbb{T}}^{(N)}) \in \mathcal{D}_{\mathbb{T}}$. By definition, the morphism $\theta_{\Delta}$ is a composition of morphisms of the form $\theta_{a\#b}$, hence the results follows from the above preliminary result.

\end{proof}

\begin{lemma}\label{lemma_center3}
If $x\neq 0$ is in the center of $\overline{\mathcal{S}}_{\omega}(\mathbf{\Sigma})$, then there exists a basic element $z \in \mathcal{Z}^0$ such that $\mathrm{ld}(x)=\mathrm{ld}(z)$ and $m_e(z)\geq m_e(x)$ for all $e\in \mathcal{E}(\Delta)$.
\end{lemma}
	
\begin{proof}  Write $\mathrm{ld}(x)= c [D',s']= [D,s] \prod_{p \in \mathcal{P}\cap \mathring{\Sigma}} \gamma_p^{m_p}$, where $c\in \mathbb{C}^*$, $(D,s)$ contains no peripheral curve and $m_p\geq 0$. By Lemma \ref{lemma_center1}, one has $\mathbf{v}(D',s')\in K^0_{\Delta}$. Since $\mathbf{v}(D',s') = \mathbf{v}(D,s) + \sum_p m_p\mathbf{k}_p $, one has also $\mathbf{v}(D,s)\in K^0_{\Delta}$. 
\vspace{2mm}
\par Recall from Proposition \ref{prop_center_CF} that $K^0_{\Delta}$ is generated, as a $\mathbb{Z}$-module, by the elements of the form $N\mathbf{k}$, with $\mathbf{k}\in K_{\Delta}$, together with the balanced maps $\mathbf{k}_p$ and $\mathbf{k}_{\partial}$ of Definition \ref{def_central_elements} associated to inner punctures $p$ and boundary components $\partial$. Therefore we can write 
$$ \mathbf{v}(D,s) = N \mathbf{k} + \sum_{p\in \mathcal{P}\cap \mathring{\Sigma}} n_p \mathbf{k}_p + \sum_{\partial \in \pi_0(\partial \Sigma)} n_{\partial} \mathbf{k}_{\partial}, $$
with $\mathbf{k}\in K_{\Delta}$ and $n_p, n_{\partial} \in \mathbb{Z}$. Composing the above equality with the group isomorphism $\Phi$ and writing $\phi:= \Phi(\mathbf{v}(D,s))$, $\phi_p:= \Phi(\mathbf{k}_p)$ and $\phi_{\partial}:= \Phi(\mathbf{k}_{\partial})$, one finds
$$ \phi = N\Phi(\mathbf{k}) + \sum_{p\in \mathcal{P}\cap \mathring{\Sigma}} n_p \phi_p + \sum_{\partial \in \pi_0(\partial \Sigma)} n_{\partial} \phi_{\partial}.$$
Note that $\phi_p(\mathfrak{c})=+1$ if $\mathfrak{c}$ is a corner edge adjacent to $p$ and is null otherwise. For $p\in \mathcal{P}\cap \mathring{\Sigma}$, since $D$ does not contain any copy of the peripheral curve $\gamma_p$, by Lemma \ref{lemma_corner_inner} there exists a corner edge $\mathfrak{c}$ adjacent to $p$ such that $\phi(\mathfrak{c})=0$, hence $N$ divides $n_p$. Therefore, there exists $\mathbf{k}'\in K_{\Delta}$ such that $\mathbf{v}(D,s) = N \mathbf{k}' + \sum_{\partial} n_{\partial} \mathbf{k}_{\partial}$. Since $\mathbf{v}(\mathrm{ld}([D,s] \prod_{\partial} \alpha_{\partial}^{-n_{\partial}})) = N \mathbf{k}'$,  by Lemma \ref{lemma_center2}, there exists $(D_0, s_0) \in \mathcal{D}_{\Delta}$ such that $N\mathbf{k}'= \mathbf{v}(D_0^{(N)}, s_0^{(N)})$. Therefore the element
$ z:= c[(D_0^{(N)}, s_0^{(N)})] \prod_p  \gamma_p^{m_p} \prod_{\partial} \alpha_{\partial}^{n_{\partial}}$
is a basic element in $\mathcal{Z}^0$ such that $\mathrm{ld}(x)=\mathrm{ld}(z)$. For $e\in \mathcal{E}(\Delta)$, using Lemma \ref{lemma_basic}, one has $m_e(z)=m_e(\mathrm{ld}(z))=m_e(\mathrm{ld}(x)) \geq m_e(x)$.
\end{proof}

\begin{proof}[Proof of Theorem \ref{theorem_center}]
Let $x$ be a non zero central element in $\overline{\mathcal{S}}_{\omega}(\mathbf{\Sigma})$. By Lemma \ref{lemma_center3}, there exists $z\in \mathcal{Z}^0$ such that the element $x':= x-z$ is either null or satisfies $\mathbf{v}(x')\prec_{\mathcal{I}} \mathbf{v}(x)$ and $m_e(x')\geq m_e(x)$ for all $e\in \mathcal{E}(\Delta)$. Let us say that $x'$ is obtained from $x$ by a \textit{central move}.
Note that there is only a finite number of balanced maps $\mathbf{k}$ such that $\mathbf{k}\prec_{\mathcal{I}} \mathbf{v}(x)$ and $\mathbf{k}(e)\geq m_e(x)$ for all $e\in \mathcal{E}(\Delta)$, therefore after a finite number of central moves one obtains zero, so $x$ is a sum of elements of $\mathcal{Z}^0$.

\end{proof}	
		


\subsection{The rank of a reduced stated skein algebra over its center}	

During all the section, we fix a triangulated punctured surface $(\mathbf{\Sigma}, \Delta)$ and an indexing map $\mathcal{I}$.

\begin{notations}
Let $\mathcal{Z}$  and $\mathcal{Z}^0$ be the centers of $\mathcal{Z}_{\omega}(\mathbf{\Sigma}, \Delta)$ and $\overline{\mathcal{S}}_{\omega}(\mathbf{\Sigma})$ respectively. Let $Q(\mathcal{Z})$ and $Q(\mathcal{Z}^0)$ their fraction fields.
We write  $R:= \dim_{Q(\mathcal{Z})} (\mathcal{Z}_{\omega}(\mathbf{\Sigma}, \Delta)\otimes_{\mathcal{Z}}Q(\mathcal{Z})$  and  $R^0= \dim_{Q(\mathcal{Z}^0)} (\overline{\mathcal{S}}_{\omega}(\mathbf{\Sigma})\otimes_{\mathcal{Z}^0}Q(\mathcal{Z}^0)$.
\end{notations}
The goal of this subsection is to prove that $R=R^0$.

\begin{lemma}\label{lemma_rank1}
One has $R^0 \geq R$.
\end{lemma}

\begin{proof}
Since $\mathcal{Z}_{\omega}(\mathbf{\Sigma}, \Delta)$ is a quantum torus, $R$ is the index of $K^0_{\Delta}$ in $K_{\Delta}$ and for $\{ [\mathbf{k}_1], \ldots, [\mathbf{k}_R] \}$ a family of representatives of the coset $\quotient{K_{\Delta}}{K^0_{\Delta}}$, the set $\{ Z^{\mathbf{k}_1}, \ldots, Z^{\mathbf{k}_R}\}$ is a generating set for the $\mathcal{Z}$-module $\mathcal{Z}_{\omega}(\mathbf{\Sigma}, \Delta)$. It follows from Lemma \ref{lemma_center_surjective} that the composition
 $$\mathbf{v}_0:  \mathcal{D}_{\Delta} \xrightarrow{\mathbf{v}} K_{\Delta} \twoheadrightarrow \quotient{K_{\Delta}}{K^0_{\Delta}} $$
 is surjective, hence we can choose a family $\{ (D_1,s_1), \ldots, (D_R,s_R) \} \subset \mathcal{D}_{\Delta}$ such that $\mathbf{v}_0(D_i,s_i)= [\mathbf{k}_i]$ for $1\leq i \leq R$. Let us prove that the elements of the  family  $\{ [D_1,s_1], \ldots, [D_R,s_R] \} $ are linearly independent in the $\mathcal{Z}^0$-module $\overline{\mathcal{S}}_{\omega}(\mathbf{\Sigma})$; this will imply the desired inequality $R^0 \geq R$. Suppose that 
 $$ \sum_{i=1}^R z_i [D_i, s_i] = 0 \quad \mbox{, for some }z_i \in \mathcal{Z}^0.$$
 Let us show that for $i\neq j$, if $z_iz_j\neq 0$ then one has $\mathbf{v}(z_i[D_i,s_i])\neq \mathbf{v}(z_j[D_j, s_j])$. Indeed, if $\mathbf{v}(z_i[D_i,s_i])= \mathbf{v}(z_j[D_j, s_j])$ and $z_iz_j\neq 0$ then $\mathbf{v}(D_i,s_i) - \mathbf{v}(D_j,s_j) = \mathbf{v}(z_j)-\mathbf{v}(z_j) \in K^0_{\Delta}$ (by Lemma \ref{lemma_center1}), hence one has $[k_i]=[k_j]$ in the coset $\quotient{K_{\Delta}}{K^0_{\Delta}}$, thus $i=j$. It follows that if the $z_i$ are not all null,  there exists $i_0$ such that $z_{i_0}\neq 0$ and  
 $$ -\infty = \mathbf{v}\left(  \sum_{i=1}^R z_i [D_i, s_i] \right) = \mathbf{v}(z_{i_0}[D_{i_0}, s_{i_0}]) = \mathbf{v}(z_{i_0}) + \mathbf{v}(D_{i_0}, s_{i_0}).$$
 Since $\mathbf{v}(D_{i_0}, s_{i_0}) \neq -\infty$, one has $\mathbf{v}(z_{i_0})=-\infty$, thus $z_{i_0}=0$ and we have a contradiction. Therefore for all $1 \leq i \leq R$, one has $z_i=0$ and the family is free.

\end{proof}	
	
It follows from Remark \ref{remark_prime} and Theorem \ref{theorem_center} that the center $\mathcal{Z}^0$ of the reduced stated skein algebra is a finitely generated commutative algebra without zero divisor, hence $\mathrm{Specm}(\mathcal{Z}^0)$ is an irreducible affine variety (it could also be derived from the Artin-Tate theorem as in \cite{FrohmanKaniaLe_UnicityRep}). Moreover Theorem \ref{theorem_center} shows that the quantum trace embeds $\mathcal{Z}^0$ into $\mathcal{Z}$ hence defines a dominant map $\psi : \mathrm{Specm}(\mathcal{Z}) \rightarrow \mathrm{Specm}(\mathcal{Z}^0)$. Since $\psi$ is dominant and 
$\mathrm{Specm}(\mathcal{Z}^0)$ irreducible, the image $\psi (\mathrm{Specm}(\mathcal{Z}))$ contains a dense Zariski open subset $\mathcal{O} \subset  \mathrm{Specm}(\mathcal{Z}^0)$.
	
\begin{lemma}\label{lemma_rank2}
One has $R\geq R^0$.
\end{lemma}

\begin{proof} Let $\chi : \mathrm{Irrep}(\overline{\mathcal{S}_{\omega}}(\mathbf{\Sigma}) \rightarrow \mathrm{Specm}(\mathcal{Z}^0)$ be the character map. Since the reduced stated skein algebra is affine almost-Azumaya (Proposition \ref{prop_almost_Azumaya}), Theorem \ref{theorem_FKL} implies that $\chi$ is surjective and that  there exists a dense Zariski open subset $\mathcal{O}' \subset \mathrm{Specm}(\mathcal{Z}^0)$ such that $\chi : \chi^{-1}(\mathcal{O}') \rightarrow \mathcal{O}'$ is a bijection and any irreducible representation in $\chi^{-1}(\mathcal{O}')$ has dimension $\sqrt{R^0}$. Since both $\mathcal{O}$ and $\mathcal{O}'$ are open dense subsets, their intersection is non empty. Let $\rho_0 \in \mathcal{O}\cap \mathcal{O}'$ and $\rho \in  \mathrm{Specm}(\mathcal{Z})$ such that $\psi (\rho)=\rho_0$. Since the balanced Chekhov-Fock algebra is Azumaya, there exists an irreducible representation $r : \mathcal{Z}_{\omega}(\mathbf{\Sigma}, \Delta) \rightarrow \mathrm{End}(V)$ (unique up to isomorphism) whose induced character on $\mathcal{Z}$ is $\rho$ and whose dimension is $\sqrt{R}$. Hence the quantum Teichm\"uller representation $r_0 : \overline{\mathcal{S}}_{\omega}(\mathbf{\Sigma}) \rightarrow \mathrm{End}(V)$, defined as the composition $r_0:= r\circ \Tr_{\omega}^{\Delta}$, induces the character $\rho_0$ over $\mathcal{Z}^0$. Therefore, $r_0$ decomposes as a direct sum of sub-representations of dimension $\sqrt{R_0}$. This proves that $\sqrt{R_0}$ divides $\sqrt{R}$, hence $R\geq R^0$. 
\end{proof}

\begin{proof}[Proof of Theorem \ref{theorem1} and Corollary \ref{corollary1}]
Theorem \ref{theorem1} is a consequence of Proposition \ref{prop_almost_Azumaya} and Lemmas \ref{lemma_rank1}, \ref{lemma_rank2}. Corollary \ref{corollary1} follows using $\mathcal{U}:= \mathcal{O}\cap \mathcal{O}'$ (with the notations of the proof of Lemma \ref{lemma_rank2}) and $\mathcal{V}:= \psi^{-1}(\mathcal{U})$.

\end{proof}

\bibliographystyle{amsalpha}
\bibliography{biblio}

\end{document}